\documentclass[a4paper]{amsart}
\usepackage{amscd}
\usepackage{amsmath}
\usepackage{mathrsfs}
\usepackage{amssymb}
\usepackage{amsthm}
\usepackage{bbm}
\usepackage{stmaryrd}
\usepackage{comment}
\usepackage[colorlinks=true, allcolors=blue]{hyperref}

\PassOptionsToPackage{pdftex}{graphicx}
\usepackage{tikz,circuitikz}

\usepackage[T1]{fontenc}

\makeatletter
\@namedef{subjclassname@2020}{%
  \textup{2020} Mathematics Subject Classification}
\makeatother

\newif\ifpdf
\ifx\pdfoutput\undefined
   \pdffalse        % we are not running PDFLaTeX
\else
   \pdfoutput=1     % we are running PDFLaTeX
   \pdftrue
\fi

\ifpdf
   \usepackage[pdftex]{graphicx}
\else
   \usepackage{graphicx}
\fi

\numberwithin{equation}{section} \swapnumbers

\newtheorem{satz}{Satz}[section]

\newtheorem{theorem}[satz]{Theorem}
\newtheorem{proposition}[satz]{Proposition}
\newtheorem{corollary}[satz]{Corollary}
\newtheorem{lemma}[satz]{Lemma}
\newtheorem{assumption}[satz]{Assumption}

\newtheorem{definition}[satz]{Definition}

\newtheorem{remark}[satz]{Remark}

\newtheorem{example}[satz]{Example}

\newcommand{\bbr}{\mathbb{R}}
\newcommand{\bbe}{\mathbb{E}}
\newcommand{\bbn}{\mathbb{N}}
\newcommand{\bbp}{\mathbb{P}}

\newcommand{\cala}{\mathcal{A}}
\newcommand{\scra}{\mathscr{A}}

\newcommand{\calf}{\mathscr{F}}

\newcommand{\call}{\mathcal{L}}

\newcommand{\calm}{\mathscr{M}}

\newcommand{\cals}{\mathscr{S}}
\newcommand{\calv}{\mathscr{V}}

\newcommand*{\cA}{\mathcal{A}}
\newcommand*{\cB}{\mathcal{B}}

\newcommand*{\cF}{\mathcal{F}}

\newcommand*{\cH}{\mathcal{H}}

\newcommand*{\fA}{\mathfrak{A}}
\newcommand*{\fB}{\mathfrak{B}}

\newcommand*{\IN}{\mathbb{N}}

\newcommand*{\IP}{\mathbb{P}}

\newcommand*{\E}{\mathbb{E}}

\newcommand*{\R}{\mathbb{R}}

\newcommand{\Id}{{\rm Id}}

\newcommand{\tr}{{\rm tr}}

\newcommand{\loc}{{\rm loc}}
\newcommand{\rank}{{\rm rank}}
\newcommand{\bild}{{\rm Im}}

\newcommand{\diag}{{\rm diag}}

\newcommand{\la}{\langle}
\newcommand{\ra}{\rangle}

\newcommand{\bbI}{\mathbbm{1}}

\calclayout

\begin{document}

\title[Stochastic Passivity in Stochastic Differential Equations]{Stochastic Passivity in Stochastic Differential Equations: A Port-Hamiltonian Perspective}
\author{Julia Ackermann}
\address{University of Wuppertal, Department of Mathematics and Natural Sciences, Gau\ss{}-stra\ss{}e 20, D-42097 Wuppertal, Germany}
\email{jackermann@uni-wuppertal.de}
\author{Thomas Kruse}
\address{University of Wuppertal, Department of Mathematics and Natural Sciences, Gau\ss{}-stra\ss{}e 20, D-42097 Wuppertal, Germany}
\email{tkruse@uni-wuppertal.de}
\author{Stefan Tappe}
%\address{Albert Ludwig University of Freiburg, Department of Mathematical Stochastics, Ernst-Zermelo-Stra\ss{}e 1, D-79104 Freiburg, Germany}
%\email{stefan.tappe@math.uni-freiburg.de}
\address{University of Wuppertal, Department of Mathematics and Natural Sciences, Gau\ss{}-stra\ss{}e 20, D-42097 Wuppertal, Germany}
\email{tappe@uni-wuppertal.de}
\begin{abstract}
We extend deterministic port-Hamiltonian systems (PHS) to a stochastic framework by means of stochastic differential equations. 
As the dissipation inequality plays a crucial role for deterministic PHS, 
we develop several passivity concepts for stochastic input-state-output systems and characterize these in terms of the parameters of the system. 
Afterwards, we examine properties of a certain class of linear stochastic systems that can be 
regarded 
as an extension of 
linear deterministic PHS to 
a stochastic passivity framework.
\end{abstract}

\maketitle

\tableofcontents

\section{Introduction}

Port-Hamiltonian systems (PHS) are widely used for the modeling 
and the control 
of complex physical systems with energy-storing and energy-dissipating elements. 
By now, there is a substantial literature about deterministic PHS, and a growing interest in stochastic extensions of PHS. 
For an introduction to deterministic PHS we refer to, for instance, the survey \cite{vanderschaftsurvey}
and the textbooks \cite{vanderschaft2009port}, \cite{jacob2012book}, and \cite[Section~6]{EhKr_vandSchaft2016l2}.
In this article we focus on finite-dimensional PHS in an input-state-output formulation.
More specifically, the deterministic literature frequently considers PHS to be 
$\bbr^d \times \bbr^n$-valued input-state-output systems of the form 
\begin{align}\label{det-system}
\left\{
\begin{array}{rcl}
\dot{x}(t) & = & b(x(t),u(t))
\\ y(t) & = & f(x(t),u(t))
\\ x(0) & = & x_0 \in \bbr^d
\\ u & \in & \cala
\end{array}
\right.
\end{align}
with certain structural properties.
Note that in \eqref{det-system}, $b : \bbr^d \times \bbr^n \to \bbr^d$ and $f : \bbr^d \times \bbr^n \to \bbr^n$ are appropriate mappings and $\cala\subseteq \{u\, \vert\, u : \bbr_+:=[0,\infty) \to \bbr^n\}$ is a suitable set of controls. 
In \eqref{det-system} we call $x(t)$ the \emph{state vector}, $u(t)$ the \emph{input vector} and $y(t)$ the \emph{output vector} at time $t \in \bbr_+$. Additionally, we also consider a function $H : \bbr^d \to \bbr$, which is called the \emph{storage function} or \emph{Hamiltonian}. It represents the energy of the system \eqref{det-system} and is typically a nonnegative function. 

A basic property of any PHS is 
that 
its \emph{energy balance}
leads to the \emph{dissipation inequality} 
\begin{align}\label{passive-ineqn}
\frac{d}{dt} H(x(t)) \leq \la u(t),y(t) \ra \quad \forall t \in \bbr_+
\end{align}
(see, e.g., \cite[Section~2.6.1]{vanderschaft2009port} or \cite[Section~6.1]{EhKr_vandSchaft2016l2}).
Note that this inequality is satisfied if and only if the function
\begin{align}\label{z-decreasing}
z : \bbr_+ \to \bbr, \quad z(t) := H(x(t)) - \int_0^t \la u(s),y(s) \ra ds
\end{align}
is decreasing. The inequality \eqref{passive-ineqn} states that the change of the system's internal energy $\frac{d}{dt} H(x(t))$ is always smaller than or equal to the external energy $\la u(t),y(t) \ra$ supplied to the system. This property is also known as the \emph{passivity property} of the system.

Of course, the passivity property \eqref{passive-ineqn} is not satisfied for every input-state-output system of the form \eqref{det-system}. It rather requires conditions on the parameters $(b,f,H)$. In the particular situation of 
deterministic linear systems, 
this problem has thoroughly been investigated in the literature (see, e.g., \cite{willems1972dissipative}) and it is well-established that, under a minimality assumption, such a system is passive if and only if it possesses a port-Hamiltonian structure (see, e.g., \cite[Corollary 4]{beattie2025port}).

From a practical perspective, it is desirable to take into account random perturbations of the system \eqref{det-system}. 
We therefore consider $\bbr^d \times \bbr^n$-valued input-state-output systems
\begin{align}\label{system-intro}
\left\{
\begin{array}{rcl}
dX_t & = & b(X_t,u_t) dt + \sigma(X_t,u_t) dW_t
\\ Y_t & = & f(X_t,u_t)
\\ X_0 & = & x_0 \in \bbr^d
\\ u & \in & \cala
\end{array}
\right.
\end{align}
driven by an $\bbr^k$-valued Wiener process $W$. In \eqref{system-intro} we have an additional diffusion coefficient $\sigma : \bbr^d \times \bbr^n \to \bbr^{d \times k}$, and now the set of controls $\cala$ is a suitable set of stochastic processes.
The purpose of this paper is to rigorously examine passivity properties of the stochastic system~\eqref{system-intro} and thereby extend PHS to a stochastic framework, in particular leading to a class of stochastic PHS that are automatically passive in a suitable sense.

Stochastic variants of finite-dimensional PHS have been investigated before in, for example, 
\cite{EhKr_cordoni2020variable}, 
\cite{EhKr_cordoni2019stochastic}, 
\cite{EhKr_cordoni2023weak}, 
\cite{EhKr_cordoni2022discrete},
\cite{EhKr_fang2017stabilization}, 
\cite{EhKr_haddad2018energy}, 
\cite{EhKr_satoh2017input},
\cite{EhKr_satoh2012passivity}, 
\cite{EhKr_satoh2014bounded} 
with applications in, for instance, teleoperation (\cite{EhKr_cordoni2021bilateral}, \cite{EhKr_cordoni2021stabilization}), 
statistical physics (\cite{EhKr_DELVENNE2014123},\cite{lanchares2023thermo}), 
or traffic flow modeling (\cite{EhKr_ackermann_tordeux24}, \cite{EhKr_Eh24}, 
\cite{ruediger2024carfollowing}). 
Most of these references consider stochastic variants of PHS in an input-state-output form that specializes~\eqref{system-intro}, whereas \cite{EhKr_cordoni2019stochastic} takes a geometric approach and \cite{EhKr_cordoni2022discrete} studies discrete versions of the stochastic PHS of \cite{EhKr_cordoni2019stochastic}. 
The recent article \cite{di2025port} builds on the geometric approach of  \cite{EhKr_cordoni2019stochastic} and proposes a framework that combines PHS with neural networks.

One of the early works on stochastic variants of PHS is \cite{EhKr_satoh2012passivity}, where under suitable assumptions stochastic generalized canonical transformations are employed to transform the system to a passive one and to subsequently perform stabilization by output feedback.
In \cite{EhKr_satoh2017input} an external disturbance term is added and stochastic input-to-output stability of the system is investigated by similar methods. 
A method for bounded stabilization that does not require the noise to vanish at the origin is studied in \cite{EhKr_satoh2014bounded}.
One of the few papers where the input, in addition, appears in the noise term and in the output is \cite{EhKr_fang2017stabilization}; with techniques similar to \cite{EhKr_satoh2012passivity}, stabilization results are derived. 
In \cite{EhKr_haddad2018energy}, 
a stochastic energy-shaping approach is developed to obtain stabilization results for stochastic variants of PHS. 
The works \cite{EhKr_cordoni2020variable} and \cite{EhKr_cordoni2023weak} 
consider stochastic variants of PHS and 
adopt a weak passivity notion suitable for weak stabilization involving invariant measures (see also \cite{fang2023weak}).
More broadly, there is literature on dissipative stochastic systems that does not specifically deal with PHS. 
For example, let us mention \cite{EhKr_Haddad2023DissStochDynSys} and their characterization results of what we will call the passive supermartingale property by the infinitesimal generator.

In the sequel, we will introduce the reader to the passivity notions that we consider in this article. 
Let us recall that the passivity property in the deterministic setting just means that the function $z$ defined in \eqref{z-decreasing} is decreasing. Therefore, for a fixed starting point $x_0 \in \bbr^d$ and a fixed control $u \in \cala$ we introduce the stochastic process $Z := Z^{(x_0,u)}$ as
\begin{align*}
Z_t := H(X_t) - \int_0^t \la u_s,Y_s \ra ds, \quad t \in \bbr_+.
\end{align*}
Then we can distinguish the following notions of passivity: 
i) We can demand that every sample path of $Z$ is decreasing. In this case we say that the system \eqref{system-intro} is \emph{passive}, or ii) we can relax this condition and merely demand that the expectation decreases; that is $t \mapsto \bbe[Z_t]$ is a decreasing function. In this case we say that the system \eqref{system-intro} is \emph{stochastically passive}. 
Our goal is to provide necessary and sufficient conditions for these two passivity concepts by means of the parameters $(b,\sigma,f,H)$ of the system \eqref{system-intro}. 
For the passivity of the system \eqref{system-intro} we will provide 
such a  
characterization in Theorem \ref{thm-system-1}. The characterization of stochastic passivity is more involved, and for this purpose it will be useful to consider further notions of passivity:
\begin{itemize}
\item Instead of stochastic passivity, we can demand the slightly stronger property that $Z$ is a supermartingale. In this case we say that the system \eqref{system-intro} has the \emph{passive supermartingale property}.

\item It will be useful to relax this property a bit and to demand that $Z$ is only a local supermartingale. Then we say that the system \eqref{system-intro} has the \emph{passive local supermartingale property}. The advantage of this concept is that local supermartingales can be characterized in terms of special semimartingales; see Proposition \ref{prop-super-MT}, which may be regarded as a local version of the Doob--Meyer decomposition theorem.
\end{itemize}
As we will see, the passive local supermartingale property can indeed be characterized by means of the parameters $(b,\sigma,f,H)$ of the system \eqref{system-intro}; see Theorem \ref{thm-system-2}. Moreover, in many situations the passive local supermartingale property even implies the passive supermartingale property, which in turn implies stochastic passivity; see Propositions \ref{prop-passive-SMP}, \ref{prop-system-3} and Theorems \ref{thm-H-p}, \ref{thm-f-p}.
A brief, simplified overview of the relations between our passivity notions is given in Figure~\ref{fig:passivity-notions}.

\begin{figure}[ht]
	\centering
	\resizebox{1\textwidth}{!}{%
		\begin{circuitikz}
			\tikzstyle{every node}=[font=\normalsize]
			\draw [->, >=Stealth, dashed] (12.75,8.75) -- (9.75,8.75)node[pos=0.5,below, fill=white]{Prop.~\ref{prop-passive-SMP}, \ref{prop-system-3}};
			\draw [<->, >=Stealth] (14.25,8.5) -- (14.25,7)node[pos=0.5, fill=white]{Thm.~\ref{thm-system-2}};
			\draw [<->, >=Stealth] (2.25,8.5) -- (2.25,7)node[pos=0.5, fill=white]{Thm.~\ref{thm-system-1}};
			\draw  (0.75,10) rectangle  node {\normalsize passive} (3.75,8.5);
			\draw  (6.75,13.75) rectangle  node {\normalsize \begin{tabular}{l}stochastically\\passive\end{tabular}} (9.75,12.25);
			\draw  (6.75,10) rectangle  node {\normalsize \begin{tabular}{l}passive\\supermartingale\\property\end{tabular}} (9.75,8.5);
			\draw  (12.75,10) rectangle  node {\normalsize \begin{tabular}{l}passive local\\supermartingale\\property\end{tabular}} (15.75,8.5);
			\draw [->, >=Stealth, dashed] (14.25,10) -- (9.75,13)node[pos=0.7,right]{Thm.~\ref{thm-H-p}, \ref{thm-f-p}};
			\draw [->, >=Stealth] (3.75,9.25) -- (6.75,9.25)node[pos=0.48,above]{if $Z_t \in L^1 \, \forall t\ge 0$};
			\draw [->, >=Stealth] (9.75,9.75) -- (12.75,9.75);
			\draw [->, >=Stealth] (8.25,10) -- (8.25,12.25);
			\draw [ rounded corners = 22.5] (0.75,7) rectangle  node {\normalsize \begin{tabular}{l}$\mathcal{L}H\le 0$\\and $\Sigma = 0$\end{tabular}} (3.75,5.5);
			\draw [ rounded corners = 22.5] (12.75,7) rectangle  node {\normalsize $\mathcal{L}H \le 0$} (15.75,5.5);
		\end{circuitikz}
	}%
	\caption{The diagram illustrates the relationship between the different passivity notions for the system~\eqref{system-intro} that we consider in this work. Note that $\mathcal{L}H$ and $\Sigma$ are defined in~\eqref{generator} and~\eqref{capital-sigma}, respectively.}
	\label{fig:passivity-notions}
\end{figure}
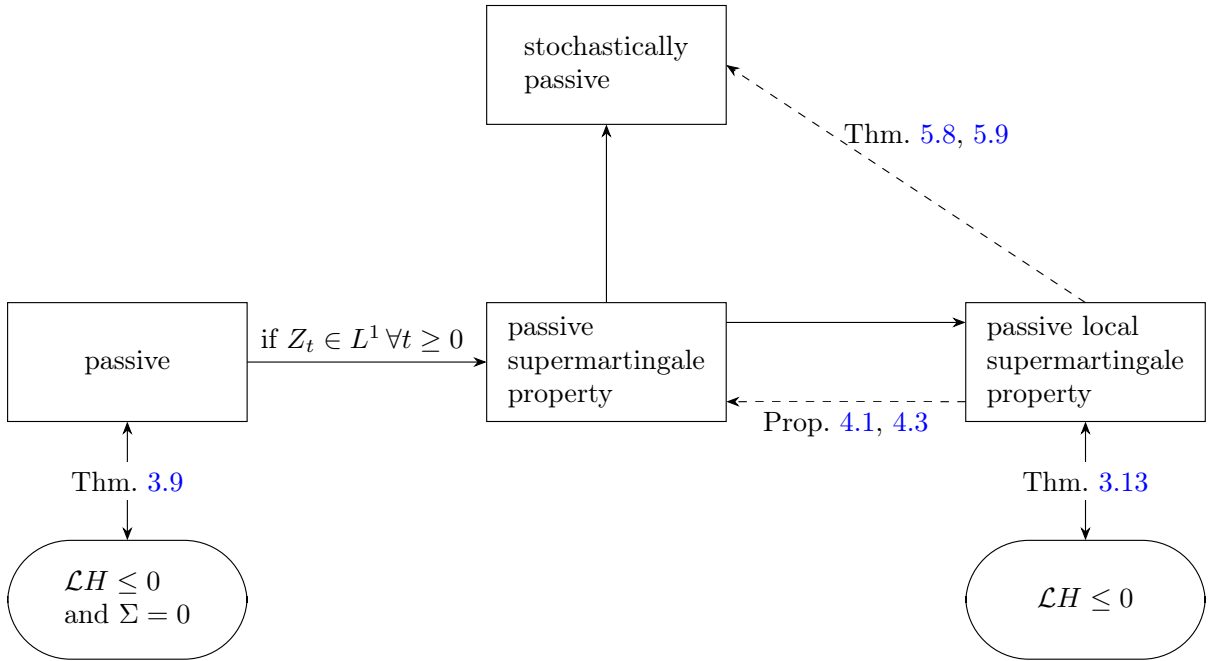

Note that for 
deterministic homogeneous linear time-invariant systems, 
the dissipation inequality~\eqref{passive-ineqn} is linked to a certain linear matrix inequality (see, e.g., \cite[Theorem~3]{willems1972dissipative}). This provides the possibility to check passivity, which is a property of the state trajectories, by an algebraic condition.  
We thus further investigate our passivity concepts in the 
case of a 
stochastic linear system of the form 
\begin{align}\label{system-intro-linear}
	\left\{
	\begin{array}{rcl}
		dX_t & = & (A X_t + B u_t) dt + \sum_{j=1}^k ( \mathfrak{A}(j) X_t + \mathfrak{B}(j) u_t ) dW_t^j
		\\ Y_t & = & C X_t + D u_t
		\\ X_0 & = & x_0 \in \bbr^d
		\\ u & \in & \cala 
	\end{array}
	\right.
\end{align}
with matrices $A \in \bbr^{d \times d}$, $B \in \bbr^{d \times n}$, $C \in \bbr^{n \times d}$, $D \in \bbr^{n \times n}$ and $\fA(j)\in\bbr^{d\times d}$, $\fB(j)\in\bbr^{d\times n}$, $j=1,\ldots,k$.  
Under suitable assumptions on the set of controls $\cala$, we obtain in Theorem~\ref{thm-linear-3} for the system~\eqref{system-intro-linear} with the quadratic Hamiltonian $H(x)=\frac12 \langle Qx,x \rangle$, $x\in\bbr^d$, where $Q\in\bbr^{d\times d}$ is a symmetric matrix, that 
negative semidefiniteness of the matrix $\mathfrak{M}_Q$ defined in~\eqref{drift-cond-matrix}, the passive local supermartingale property, the passive supermartingale property, and stochastic passivity are equivalent. 
In contrast, the characterization of the passivity of the system~\eqref{system-intro-linear} involves a further algebraic condition, see Theorem~\ref{thm-linear-passive}.

To extend linear deterministic PHS to stochastic ones, we propose a definition that preserves the structure of linear deterministic PHS and also ensures that the linear matrix inequality $\mathfrak{M}_Q\le 0$ is satisfied, see Definition~\ref{def-SLTIPHS} and Lemma~\ref{rem-SLTIPHS-params-SLTIS}. 
In particular, this guarantees that stochastic PHS are stochastically passive (see also Corollary~\ref{cor-SLTIPHS-passivityprop}). 
Generalizing results from the deterministic literature (see, e.g., \cite[Theorem~2.3]{vanderschaft2009port}), we moreover in Lemma~\ref{lemma:lmiimpliesphs} provide an algebraic condition that assures that $\mathfrak{M}_Q\le 0$ implies that the stochastic linear system~\eqref{system-intro-linear} is a stochastic PHS.
Note that already in the deterministic case, without such an additional assumption this implication can be false (see, e.g., \cite[Example~2]{cherifi2024difference}).
To summarize, under the algebraic condition of Lemma~\ref{lemma:lmiimpliesphs} and suitable assumptions on the set of controls~$\cala$, we obtain that the stochastic linear system~\eqref{system-intro-linear} with the quadratic Hamiltonian $H(x)=\frac12 \langle Qx,x \rangle$, $x\in\bbr^d$, where $Q\in\bbr^{d\times d}$ is a positive semidefinite matrix, 
is a stochastic PHS if and only if it is stochastically passive. 
In Corollary~\ref{cor:pass_and_obs_imply_phs} we replace the assumption of a quadratic Hamiltonian and the algebraic condition of Lemma~\ref{lemma:lmiimpliesphs} by the assumption that the stochastic linear system~\eqref{system-intro-linear} is observable (see Definition~\ref{def-observable}). 

A further central property of deterministic PHS is the fact that any energy-conserving interconnection of two such systems again results in a PHS. We verify in Proposition~\ref{lemma_interconnection} that this 
stability-under-interconnection property also 
holds for the linear stochastic PHS introduced above. Moreover, we extend simple mechanical and electrical examples to the stochastic setting.

The remainder of this paper is organized as follows. In Section \ref{sec-notation} we provide the stochastic framework and introduce the required concepts. In Section \ref{sec-results-iff} we present results which characterize the passivity and the passive local supermartingale property. In Section \ref{sec-local-pass-implies-pass} we present results about the passive supermartingale property; more precisely, we provide conditions which ensure that the passive local supermartingale property implies the passive supermartingale property. In Section \ref{sec-existence-sol} we present further results about the passive supermartingale property; this comes along with an existence and uniqueness result for strong solutions as well as moment estimates. In Section \ref{sec-uncontrolled} we apply our findings to the particular situation of uncontrolled SDEs. Afterwards, we focus on the linear case. 
Namely, in Section \ref{sec-SLTIS} we use the results from previous sections to obtain some basic results for stochastic, linear, time-invariant input-state-output systems (SLTIS). 
In Section \ref{sec-SLTIS-2} we examine observability and stochastic passivity of such systems. 
In Section \ref{sec-linear-phs} we 
introduce and characterize 
linear stochastic PHS. 
For convenience of the reader, in Appendix \ref{app-semimartingales} we provide the required results about semimartingales, including a local version of the Doob--Meyer decomposition theorem, and in Appendix \ref{app-SDEs} we provide an existence result for SDEs.

\section{Stochastic input-state-output systems}\label{sec-notation}

In this section we introduce the stochastic input-state-output system and the relevant concepts. Let $(\Omega,\calf,(\calf_t)_{t \geq 0}, \bbp)$ be a filtered probability space satisfying the usual conditions. We consider the following $\bbr^d \times \bbr^n$-valued input-state-output system
\begin{align}\label{system}
\left\{
\begin{array}{rcl}
dX_t & = & b(X_t,u_t) dt + \sigma(X_t,u_t) dW_t
\\ Y_t & = & f(X_t,u_t)
\\ X_0 & = & x_0 \in \bbr^d
\\ u & \in & \cala
\end{array}
\right.
\end{align}
driven by an $\bbr^k$-valued Wiener process $W$, where the set $\cala$ of control processes is a set of $\bbr^n$-valued progressively measurable processes. Furthermore, in \eqref{system} we have measurable mappings $b : \bbr^d \times \bbr^n \to \bbr^d$, $\sigma : \bbr^d \times \bbr^n \to \bbr^{d \times k}$ and $f : \bbr^d \times \bbr^n \to \bbr^n$. For each $j = 1,\ldots,k$ the mapping $\sigma^j : \bbr^d \times \bbr^n \to \bbr^d$ denotes the $j$-th column given by
\begin{align*}
\sigma^j(x,v) := \sigma(x,v)_{\bullet j} \quad \forall x \in \bbr^d \quad \forall v \in \bbr^n.
\end{align*}
We also fix a measurable mapping $H : \bbr^d \to \bbr$, the so-called storage function.

\begin{remark}\label{rem-notation}
We agree on the following notation:
\begin{enumerate}
\item For $x \in \bbr^d$ we denote by
\begin{align*}
\| x \| := \bigg( \sum_{i=1}^d |x_i|^2 \bigg)^{1/2}
\end{align*}
the \emph{Euclidean norm} of $x$.

\item For $A \in \bbr^{d \times k}$ we denote by
\begin{align*}
\| A \| := \bigg( \sum_{i=1}^d \sum_{j=1}^k |A_{ij}|^2 \bigg)^{1/2}
\end{align*}
the \emph{Frobenius norm} of $A$.
\end{enumerate}
\end{remark}

\begin{definition}
Let $x_0 \in \bbr^d$ and $u \in \cala$ be arbitrary. We call an adapted continuous process $X = (X_t)_{t \geq 0}$ a \emph{strong solution} to the system \eqref{system} with $X_0 = x_0$ and control~$u$ if we have $\bbp$-almost surely
\begin{align*}
\int_0^t \big( \| b(X_s,u_s) \| + \| \sigma(X_s,u_s) \|^2 \big) ds < \infty \quad \forall t \in \bbr_+:=[0,\infty) 
\end{align*}
as well as
\begin{align*}
X_t = x_0 + \int_0^t b(X_s,u_s) ds + \int_0^t \sigma(X_s,u_s) dW_s \quad \forall t \in \bbr_+.
\end{align*}
\end{definition}

\begin{remark}
Let $x_0 \in \bbr^d$ and $u \in \cala$ be arbitrary, and let $X$ be a strong solution to the system \eqref{system} with $X_0 = x_0$ and control $u$. As indicated in \eqref{system}, in this case we define the $\bbr^n$-valued process $Y = (Y_t)_{t \geq 0}$ as $Y_t := f(X_t,u_t)$ for each $t \in \bbr_+$, and we will also call $(X,Y) = (X^{(x_0,u)},Y^{(x_0,u)})$ a \emph{strong solution} to the system \eqref{system}.
\end{remark}

\begin{definition}
We say that \emph{strong uniqueness} for the system \eqref{system} holds true if for each $x_0 \in \bbr^d$, each $u \in \cala$, and any two strong solutions $X^1$ and $X^2$ to the system \eqref{system} with $X_0^1 = X_0^2 = x_0$ and control $u$ we have $X^1 = X^2$ up to indistinguishability.
\end{definition}

\begin{definition}
We say that \emph{existence and uniqueness of strong solutions} for the system \eqref{system} holds true if the following conditions are fulfilled:
\begin{enumerate}
\item For each $x_0 \in \bbr^d$ and each $u \in \cala$ there exists a strong solution to the system \eqref{system} with $X_0 = x_0$ and control $u$.

\item Strong uniqueness for the system \eqref{system} holds true.
\end{enumerate}
\end{definition}

In what follows, we assume that the following assumption is fulfilled.

\begin{assumption}\label{ass-solutions}
We suppose that existence and uniqueness of strong solutions for the system \eqref{system} holds true.
\end{assumption}

In Section \ref{sec-existence-sol} we will provide sufficient conditions which ensure that Assumption \ref{ass-solutions} is satisfied.

\begin{assumption}\label{ass-int-cond-for-passivity}
We suppose that for all $x_0 \in \bbr^d$ and every control $u \in \cala$ we have $\bbp$-almost surely
\begin{align}\label{int-cond-for-passivity}
\int_0^t | \la u_s,f(X_s^{(x_0,u)},u_s) \ra | ds < \infty \quad \forall t \in \bbr_+.
\end{align}
\end{assumption}

\begin{remark}
It is easy to find sufficient conditions such that Assumption \ref{ass-int-cond-for-passivity} is fulfilled. Indeed, suppose, for example, that the mapping $f$ is continuous and that each control process $u \in \cala$ is continuous. Then the integrability condition \eqref{int-cond-for-passivity} is satisfied, because for all $x_0 \in \bbr^d$ and all $u \in \cala$ the solution $X^{(x_0,u)}$ is continuous.
\end{remark}

For $x_0 \in \bbr^d$ and $u \in \cala$ we introduce the process $Z^{(x_0,u)}$ as
\begin{align}\label{Z-definition}
Z_t^{(x_0,u)} := H(X_t^{(x_0,u)}) - \int_0^t \la u_s,Y_s^{(x_0,u)} \ra ds, \quad t \in \bbr_+.
\end{align}
For the next definition, recall that $\calv^-$ denotes the convex cone of all c\`{a}dl\`{a}g, adapted and decreasing processes $A$ with $A_0 = 0$.

\begin{definition}
We say that the system \eqref{system} is \emph{passive} with respect to the storage function $H$ if for all $x_0 \in \bbr^d$ and $u \in \cala$ we have $Z - Z_0 \in \calv^-$, where $Z := Z^{(x_0,u)}$.
\end{definition}

\begin{definition}
We say that the system \eqref{system} is \emph{stochastically  passive} with respect to the storage function $H$ if for all $x_0 \in \bbr^d$ and $u \in \cala$ we have $Z_t^{(x_0,u)} \in L^1$, $t \in \bbr_+$ and the mapping $\bbr_+ \to \bbr$, $t \mapsto \bbe [ Z_t^{(x_0,u)} ]$ is decreasing.
\end{definition}

\begin{definition}
We say that the system \eqref{system} has the \emph{passive supermartingale property} with respect to the storage function $H$ if for all $x_0 \in \bbr^d$ and $u \in \cala$ the process $Z^{(x_0,u)}$ is a supermartingale.
\end{definition}

\begin{definition}
We say that the system \eqref{system} has the \emph{passive local supermartingale property} with respect to the storage function $H$ if for all $x_0 \in \bbr^d$ and $u \in \cala$ the process $Z^{(x_0,u)}$ is a local supermartingale.
\end{definition}

\begin{remark}\label{rem-concepts}
Note the following relations between these passivity concepts:
\begin{enumerate}
\item If the system \eqref{system} is passive, then it has the passive local supermartingale property. This is an immediate consequence of Proposition \ref{prop-super-MT}.

\item If the system \eqref{system} is passive, and for all $x_0 \in \bbr^d$ and $u \in \cala$ we have $Z_t^{(x_0,u)}\in L^1$ for each $t \geq 0$, then the system \eqref{system} has the passive supermartingale property.

\item If the system \eqref{system} has the passive supermartingale property, then it is stochastically passive and has the passive local supermartingale property.
\end{enumerate}
\end{remark}

\section{Results about passivity and the passive local supermartingale property}\label{sec-results-iff}

In this section we present results about passivity and the passive local supermartingale property. We will make the following additional assumption.

\begin{assumption}\label{ass-H-C2}
We assume that the storage function $H$ is of class $C^2$.
\end{assumption}

\begin{definition}
We define $\call H : \bbr^d \times \bbr^n \to \bbr$ as
\begin{align}\label{generator}
\call H(x,v) := \la \nabla H(x), b(x,v) \ra + \frac{1}{2} \tr \big( \sigma(x,v) \sigma(x,v)^{\top} D^2 H(x) \big) - \la v,f(x,v) \ra.
\end{align}
Furthermore, we define $\Sigma : \bbr^d \times \bbr^n \to \bbr^k$ as
\begin{align}\label{capital-sigma}
\Sigma(x,v) := \nabla H(x)^\top \cdot \sigma(x,v).
\end{align}
\end{definition}

The following auxiliary result is an immediate consequence of It\^{o}'s formula.

\begin{lemma}\label{lemma-generator}
Let $x_0 \in \bbr^d$ and $u \in \cala$ be arbitrary. Then we have
\begin{align*}
Z_t^{(x_0,u)} = H(x_0) + \int_0^t \call H(X_s,u_s) ds + \int_0^t \Sigma(X_s,u_s) dW_s, \quad t \geq 0.
\end{align*}
\end{lemma}

Furthermore, the following auxiliary result will be useful.

\begin{lemma}\label{lemma-integrals}
Let $\varphi : \bbr_+ \to \bbr$ be a locally integrable function which is continuous in zero. We define $\phi : [0,T] \to \bbr$ as
\begin{align*}
\phi(t) := \int_0^t \varphi(s) ds, \quad t \geq 0.
\end{align*}
Then the following statements are true:
\begin{enumerate}
\item If $\phi(t) = 0$ for all $t \geq 0$, then we have $\varphi(0) = 0$.

\item If $\phi$ is decreasing, then we have $\varphi(0) \leq 0$.
\end{enumerate}
\end{lemma}

\begin{proof}
(1) Suppose that $\varphi(0) \neq 0$. Without loss of continuity, we assume that $\varphi(0) > 0$. Set $c := \frac{\varphi(0)}{2} > 0$. By the continuity of $\varphi$ in zero there exists $\delta > 0$ such that $\varphi(t) \geq c$ for all $t \in [0,\delta]$. Hence we arrive at the contradiction $\Phi(\delta) \geq c \delta > 0$.

\noindent(2) Suppose that $\varphi(0) > 0$. Set $c := \frac{\varphi(0)}{2} > 0$. By the continuity of $\varphi$ in zero there exists $\delta > 0$ such that $\varphi(t) \geq c$ for all $t \in [0,\delta]$. Hence we deduce $\Phi(\delta) \geq c \delta > 0 = \Phi(0)$, contradicting the assumption that $\Phi$ is decreasing.
\end{proof}

From now on, let us also make the following assumptions.

\begin{assumption}\label{ass-b-sigma}
The mappings $b$ and $\sigma$ are continuous.
\end{assumption}

\begin{assumption}\label{ass-f}
The mapping $f$ is continuous.
\end{assumption}

\begin{remark}\label{rem-gen-cont}
Due to Assumptions \ref{ass-H-C2}, \ref{ass-b-sigma} and \ref{ass-f}, the mappings $\call H$ and $\Sigma$ given by \eqref{generator} and \eqref{capital-sigma} are also continuous.
\end{remark}

\begin{assumption}\label{ass-control-cont-in-zero}
For each $v \in \bbr^n$ there exists a control process $u \in \cala$ with $u_0 = v$ that is continuous at zero.
\end{assumption}

\begin{theorem}\label{thm-system-1}
Suppose that Assumptions \ref{ass-solutions}, \ref{ass-int-cond-for-passivity}, \ref{ass-H-C2}, \ref{ass-b-sigma}, \ref{ass-f} and \ref{ass-control-cont-in-zero} are fulfilled. Then the following statements are equivalent:
\begin{enumerate}
\item[(i)] The system \eqref{system} is passive with respect to the storage function $H$.

\item[(ii)] We have
\begin{align}\label{cond-drift}
\call H(x,v) \leq 0 \quad \forall x \in \bbr^d \quad \forall v \in \bbr^n,
\\ \label{cond-diffusion} \Sigma(x,v) = 0 \quad \forall x \in \bbr^d \quad \forall v \in \bbr^n.
\end{align}
\end{enumerate}
\end{theorem}

\begin{proof}
(ii) $\Rightarrow$ (i): Let $x_0 \in \bbr^d$ and $u \in \cala$ be arbitrary, and denote by $(X,Y) = (X^{(x_0,u)},Y^{(x_0,u)})$ the solution to the system \eqref{system}. Furthermore, we set $Z := Z^{(x_0,u)}$. Taking into account \eqref{cond-diffusion}, by Lemma \ref{lemma-generator} we have
\begin{align*}
Z_t = H(x_0) + \int_0^t \call H(X_s,u_s) ds, \quad t \geq 0.
\end{align*}
Therefore, by condition \eqref{cond-drift} we have $Z - Z_0 \in \calv^-$.

\noindent(i) $\Rightarrow$ (ii): Let $x_0 \in \bbr^d$ and $v \in \bbr^n$ be arbitrary. By Assumption \ref{ass-control-cont-in-zero} there exists a control process $u \in \cala$ with $u_0 = v$ that is continuous at zero. We denote by $(X,Y) = (X^{(x_0,u)},Y^{(x_0,u)})$ the solution to the system \eqref{system}, and set $Z := Z^{(x_0,u)}$. By Lemma \ref{lemma-generator} we have
\begin{align*}
Z_t = H(x_0) + \int_0^t \call H(X_s,u_s) ds + \int_0^t \Sigma(X_s,u_s) dW_s, \quad t \geq 0.
\end{align*}
Hence $Z$ is a special semimartingale with canonical decomposition $Z = Z_0 + M + A$, where $M \in \calm_{\loc}^0$ and the predictable process $A \in \calv$ are given by
\begin{align*}
M_t &= \int_0^t \Sigma(X_s,u_s) dW_s, \quad t \geq 0,
\\ A_t &= \int_0^t \call H(X_s,u_s) ds, \quad t \geq 0.
\end{align*}
Furthermore, we have $Z = Z_0 + (Z - Z_0)$ with $Z - Z_0 \in \calv^-$. By the uniqueness of the canonical decomposition of $Z$ (see Proposition \ref{prop-can-decomp}) we deduce that $M = 0$ and $A = Z - Z_0$. Moreover, we have $\la M,M \ra = 0$ and
\begin{align*}
\la M,M \ra_t = \int_0^t \| \Sigma(X_s,u_s) \|^2 ds, \quad t \geq 0.
\end{align*}
By the continuity of $\Sigma$ (see Remark \ref{rem-gen-cont}) the sample paths of $s \mapsto \Sigma(X_s,u_s)$ are continuous in zero, and hence by Lemma \ref{lemma-integrals} we deduce that $\Sigma(x_0,v) = 0$.

Furthermore, recall that $A \in \calv^-$. By the continuity of $\call H$ (see Remark \ref{rem-gen-cont}), the sample paths of $s \mapsto \call H(X_s,u_s)$ are continuous at zero, and hence, by Lemma \ref{lemma-integrals} we deduce that $\call H(x_0,v) \leq 0$. Since $x_0 \in \bbr^d$ and $v \in \bbr^n$ were arbitrary, we arrive at \eqref{cond-drift} and \eqref{cond-diffusion}.
\end{proof}

In view of the upcoming results let us agree on the following notation. If $\sigma(\cdot,u) : \bbr^d \to \bbr^{d \times k}$ is differentiable for some $u \in \bbr^n$, then for any $j=1,\ldots,k$ the derivative of the mapping $\bbr^d \to \bbr^d$, $x \mapsto \sigma^j(x,u)$ is called the partial derivative of $\sigma^j$ at $u$, and is denoted by $\bbr^d \to \bbr^{d \times d}$, $x \mapsto D_x \sigma^j(x,u)$.

\begin{lemma}\label{lem:stratonovic}
Suppose that $\sigma(\cdot,u) \in C^1$ for all $u \in \bbr^n$, and that \eqref{cond-diffusion} is fulfilled. Then we have
\begin{align*}
\tr \big( \sigma(x,u) \sigma(x,u)^{\top} D^2 H(x) \big) = - \bigg\la \nabla H, \sum_{j=1}^k (D_x \sigma^j)(x,u) \sigma^j(x,u) \bigg\ra.
\end{align*}
In particular, condition \eqref{cond-drift} is fulfilled if and only if
\begin{align}\label{eq:cond-drift-strato}
\bigg\la \nabla H(x), b(x,u) - \frac{1}{2} \sum_{j=1}^k (D_x \sigma^j)(x,u)\sigma^j(x,u) \bigg\ra - \la u,f(x,u) \ra \leq 0.
\end{align}
\end{lemma}

\begin{proof}
Note that \eqref{cond-diffusion} implies for all $j\in \{1,\ldots,k\}$, $i\in\{1,\ldots,d\}$ that
$$
0=\frac{\partial}{\partial x_i}\langle \nabla H,\sigma^j\rangle
= \sum_{l=1}^d\bigg((D^2H)_{li}\sigma_{lj}+\frac{\partial H}{\partial x_l}\frac{\partial \sigma_{lj}}{\partial x_i}\bigg).
$$
This implies that
\begin{equation*}
    \begin{split}
        \tr ( \sigma \sigma^{\top} D^2 H )
        &= \sum_{i=1}^d\sum_{l=1}^d\sum_{j=1}^k \sigma_{ij}\sigma_{lj}(D^2H)_{li}
        = - \sum_{i=1}^d\sum_{l=1}^d\sum_{j=1}^k \sigma_{ij}\frac{\partial H}{\partial x_l}\frac{\partial \sigma_{lj}}{\partial x_i}\\
        &=-\sum_{j=1}^k\sum_{l=1}^d \frac{\partial H}{\partial x_l} \sum_{i=1}^d(D_x\sigma^j)_{li}\sigma_{ij}=-\left\langle \nabla H, \sum_{j=1}^k D_x \sigma^j \cdot \sigma^j \right \rangle.
    \end{split}
\end{equation*}
This proves the claim.
\end{proof}

\begin{proposition}\label{prop-system-1b}
Suppose that Assumptions \ref{ass-solutions}, \ref{ass-int-cond-for-passivity}, \ref{ass-H-C2}, \ref{ass-b-sigma}, \ref{ass-f} and \ref{ass-control-cont-in-zero} are fulfilled, and that $\sigma(\cdot,u) \in C^1$ for all $u \in \bbr^n$. Then the following statements are equivalent:
\begin{enumerate}
\item[(i)] The system \eqref{system} is passive with respect to the storage function $H$.

\item[(ii)] We have \eqref{cond-diffusion} and
\begin{align*}
\bigg\la \nabla H(x), b(x,v) - \frac{1}{2} \sum_{j=1}^k (D_x \sigma^j)(x,v)\sigma^j(x,v) \bigg\ra - \la v,f(x,v) \ra \leq 0
\end{align*}
for all $x \in \bbr^d$ and all $v \in \bbr^n$.
\end{enumerate}
\end{proposition}

\begin{proof}
This is an immediate consequence of Theorem \ref{thm-system-1} and Lemma \ref{lem:stratonovic}.
\end{proof}

\begin{remark}
    Note that the term $\frac{1}{2} \sum_{j=1}^k (D_x \sigma^j)\sigma^j$ in Lemma \ref{lem:stratonovic} represents the so-called Stratonovich correction term. Its presence is expected, since under the assumption that $\sigma(\cdot,u) \in C^1$ for all $u \in \bbr^n$, we could have transformed \eqref{system} into Stratonovich form. Then by applying the Stratonovich chain rule instead of It\^o's formula in the proof of Theorem \ref{thm-system-1}, we would have directly arrived at \eqref{eq:cond-drift-strato}.
\end{remark}

\begin{theorem}\label{thm-system-2}
Suppose that Assumptions \ref{ass-solutions}, \ref{ass-int-cond-for-passivity}, \ref{ass-H-C2}, \ref{ass-b-sigma}, \ref{ass-f} and \ref{ass-control-cont-in-zero} are fulfilled. Then the following statements are equivalent:
\begin{enumerate}
\item[(i)] The system \eqref{system} has the passive local supermartingale property with respect to the storage function $H$.

\item[(ii)] We have \eqref{cond-drift}.
\end{enumerate}
\end{theorem}

\begin{proof}
(ii) $\Rightarrow$ (i): Let $x_0 \in \bbr^d$ and $u \in \cala$ be arbitrary, and denote by $(X,Y) = (X^{(x_0,u)},Y^{(x_0,u)})$ the solution to the system \eqref{system}, and set $Z := Z^{(x_0,u)}$. By Lemma \ref{lemma-generator} we have
\begin{align*}
Z_t = H(x_0) + \int_0^t \call H(X_s,u_s) ds + \int_0^t \Sigma(X_s,u_s) dW_s, \quad t \geq 0.
\end{align*}
Therefore $Z$ is a special semimartingale with canonical decomposition $Z = Z_0 + M + A$, where $M \in \calm_{\loc}^0$ and the predictable process $A \in \calv$ are given by
\begin{align*}
A_t &= \int_0^t \call H(X_s,u_s) ds, \quad t \geq 0,
\\ M_t &= \int_0^t \Sigma(X_s,u_s) dW_s, \quad t \geq 0.
\end{align*}
By \eqref{cond-drift} we have $A \in \calv^-$. Therefore, by Proposition \ref{prop-super-MT} we infer that $Z$ is a local supermartingale.

\noindent(i) $\Rightarrow$ (ii): Let $x_0 \in \bbr^d$ and $v \in \bbr^n$ be arbitrary. By Assumption \ref{ass-control-cont-in-zero} there exists a control process $u \in \cala$ with $u_0 = v$ that is continuous at zero. We denote by $(X,Y) = (X^{(x_0,u)},Y^{(x_0,u)})$ the solution to the system \eqref{system}, and set $Z := Z^{(x_0,u)}$. By Lemma \ref{lemma-generator} we have
\begin{align*}
Z_t = H(x_0) + \int_0^t \call H(X_s,u_s) ds + \int_0^t \Sigma(X_s,u_s) dW_s, \quad t \geq 0.
\end{align*}
Hence $Z$ is a special semimartingale with canonical decomposition $Z = Z_0 + M + A$, where $M \in \calm_{\loc}^0$ and the predictable process $A \in \calv$ are given by
\begin{align*}
M_t &= \int_0^t \Sigma(X_s,u_s) dW_s, \quad t \geq 0,
\\ A_t &= \int_0^t \call H(X_s,u_s) ds, \quad t \geq 0.
\end{align*}
Since $Z$ is a local supermartingale, by Proposition \ref{prop-super-MT} and the uniqueness of the canonical decomposition of $Z$ (see Proposition \ref{prop-can-decomp}) we deduce that $A \in \calv^-$. By the continuity of $\call H$ (see Remark \ref{rem-gen-cont}), the sample paths of $s \mapsto \call H(X_s,u_s)$ are continuous in zero, and hence by Lemma \ref{lemma-integrals} we deduce that $\call H(x_0,v) \leq 0$. Since $x_0 \in \bbr^d$ and $v \in \bbr^n$ were arbitrary, we arrive at \eqref{cond-drift}.
\end{proof}

\begin{remark}
As a consequence of Theorems \ref{thm-system-1} and \ref{thm-system-2}, we see that passivity of the system \eqref{system} implies the local supermartingale property. This confirms the first statement made in Remark \ref{rem-concepts}.
\end{remark}

\begin{remark}
    Sometimes in the literature the condition~\eqref{cond-drift} 
    is used to define stochastic passivity notions. 
    For example, in \cite[Definition~4.1]{EhKr_florchinger1999passive} (see also \cite[Definition~3.1]{florchinger2016global}) a system is called passive if there exists a positive definite (i.e., $H(0)=0$ and $\forall x\in\bbr^d\backslash\{0\}\colon H(x)>0$) storage function $H$ (or Lyapunov function) of class $C^2$ satisfying~\eqref{cond-drift}. This definition is particularly suitable for examining stability properties of the system. In particular, in this situation the solution $0$ of the uncontrolled stochastic differential equation \eqref{SDE} is stable in probability (see, e.g., \cite[Theorem~5.3]{EhKr_khasminskii2011stochastic}, cf.\ also \cite[Remark~3.3]{florchinger2016global}).
    Moreover, under the sufficient conditions given in \cite[Theorem~3.4]{florchinger2016global} or \cite[Theorem~4.5]{florchinger2016global}, the system~\eqref{system} is globally asymptotically stabilizable in probability by an output feedback control or by a state feedback control, respectively. For more details on stability and stabilization of possibly nonlinear systems we refer to, for instance, \cite{EhKr_florchinger1999passive,florchinger2016global,EhKr_khasminskii2011stochastic}. The linear case and stabilization in mean-square are discussed in Remark~\ref{rem-SLTIPHS-stability} below.
\end{remark}

\section{Results about the passive supermartingale property}\label{sec-local-pass-implies-pass}

So far, we have derived necessary and sufficient conditions for passivity and for the passive local supermartingale property of the system \eqref{system}. Now, we are interested in the question when the passive local supermartingale property implies the passive supermartingale property. Recall from Remark \ref{rem-concepts} that the passive supermartingale property in particular gives us that the system \eqref{system} is stochastically passive.

The norms in conditions \eqref{int-cond-1}--\eqref{int-cond-4b} below are Euclidean and Frobenius norms; recall the conventions made in Remark \ref{rem-notation}.

\begin{proposition}\label{prop-passive-SMP}
Suppose that Assumptions \ref{ass-solutions} and \ref{ass-H-C2} are fulfilled.
Moreover, suppose that for all $x_0 \in \bbr^d$, all $u \in \cala$ and all $t \in \bbr_+$ we have
\begin{align}\label{int-cond-1}
&\int_0^t \bbe \big[ \| \nabla H(X_s) \| \cdot \| b(X_s,u_s) \| \big] ds < \infty,
\\ \label{int-cond-2} &\int_0^t \bbe \big[ \| \sigma(X_s,u_s) \|^2 \cdot \| D^2 H(X_s) \| \big] ds < \infty,
\\ \label{int-cond-3} &\int_0^t \bbe \big[ \| u_s \| \cdot \| f(X_s,u_s) \| \big] ds < \infty,
\\ \label{int-cond-4a} &\bbe\bigg[\sup_{s\in [0,t]}\|\nabla H(X_s)\|^2\bigg]<\infty,
\\ \label{int-cond-4b} &\bbe\bigg[\int_0^t\|\sigma (X_s,u_s)\|^2ds\bigg]<\infty,
\end{align}
where $X = X^{(x_0,u)}$ denotes the strong solution to the system \eqref{system} with $X_0 = x_0$ and control $u$.
We then have the following: 
\begin{enumerate}
    \item[(i)]
    For all $x_0 \in \bbr^d$, $u \in \cala$ and $t\ge 0$ it holds that 
    \begin{align*}
    \int_0^t \call H(X_s,u_s) ds
    \in L^1 .
    \end{align*}

    \item[(ii)]
    For all $x_0 \in \bbr^d$ and all $u \in \cala$ it holds that $M=(M_t)_{t\ge 0}$ defined by 
    \begin{align*}
    M_t &= \int_0^t \Sigma(X_s,u_s) dW_s, \quad t \geq 0,
    \end{align*}
    is a martingale.
    
    \item[(iii)]
    If the system \eqref{system} has the passive local supermartingale property with respect to the storage function $H$, then it has the passive supermartingale property with respect to the storage function $H$.
\end{enumerate}
\end{proposition}

\begin{remark}
Note that Assumption \ref{ass-int-cond-for-passivity} is automatically fulfilled here, because condition \eqref{int-cond-3} implies \eqref{int-cond-for-passivity}.
\end{remark}

\begin{proof}[Proof of Proposition \ref{prop-passive-SMP}]
Let $x_0 \in \bbr^d$ and $u \in \cala$ be arbitrary. By Lemma \ref{lemma-generator} the process $Z = Z^{(x_0,u)}$ is given by $Z = Z_0 + M + A$, where
\begin{align*}
M_t &= \int_0^t \Sigma(X_s,u_s) dW_s, \quad t \geq 0,
\\ A_t &= \int_0^t \call H(X_s,u_s) ds, \quad t \geq 0.
\end{align*}
Let $x \in \bbr^d$ and $v \in \bbr^n$ be arbitrary. Then we have
\begin{align*}
&| \la \nabla H(x), b(x,v) \ra | \leq \| \nabla H(x) \| \cdot \| b(x,v) \|.
\end{align*}
Furthermore, denoting by $\| \cdot \|_{L(\bbr^d)}$ the operator norm and by $\| \cdot \|_{L_1(\bbr^d)}$ the nuclear norm, we obtain
\begin{align*}
| \tr \big( \sigma(x,v) \sigma(x,v)^{\top} D^2 H(x) \big) | &\leq \| \sigma(x,v) \sigma(x,v)^{\top} D^2 H(x) \|_{L_1(\bbr^d)}
\\ &\leq \| \sigma(x,v) \sigma(x,v)^{\top} \|_{L_1(\bbr^d)} \| D^2 H(x) \|_{L(\bbr^d)}
\\ &\leq \| \sigma(x,v) \| \, \| \sigma(x,v)^{\top} \| \, \| D^2 H(x) \|
\\ &= \| \sigma(x,v) \|^2 \| D^2 H(x) \|.
\end{align*}
Furthermore, we have the estimate
\begin{align*}
| \la v,f(x,v) \ra | \leq \| v \| \cdot \| f(x,v) \|.
\end{align*}
Therefore, recalling the definition \eqref{generator} of $\call H$, by \eqref{int-cond-1}--\eqref{int-cond-3} we deduce that $A_t \in L^1$ for all $t \geq 0$. This proves (i). 
Now, we show that $M$ is a martingale. According to Lemma \ref{lemma-martingale-suff}, a sufficient condition for this is that for all $t\in \bbr_+$ it holds
\begin{equation}\label{eq:suff_cond_mart}
    \bbe\bigg[\sup_{s\in [0,t]}|M_s|\bigg]<\infty.
\end{equation}
The Burkholder--Davis--Gundy inequality ensures that there exists $C>0$ such that for all $t\in \bbr_+$ we have
$$
\bbe\bigg[\sup_{s\in [0,t]}|M_s|\bigg]\le C \, \bbe\bigg[\sqrt{\langle M \rangle_t}\bigg].
$$
Moreover, recalling the definition \eqref{capital-sigma} of $\Sigma$, for all $x \in \bbr^d$ and $v \in \bbr^n$ we have
\begin{align*}
\| \Sigma(x,v) \| \leq \| \nabla H(x)^{\top} \| \cdot \| \sigma(x,v) \|_{L(\bbr^k,\bbr^d)} \leq \| \nabla H(x) \| \cdot \| \sigma(x,v) \|.
\end{align*}
This implies for all $t\in \bbr_+$ that
\begin{align*}
\bbe\bigg[\sup_{s\in [0,t]}|M_s|\bigg] &\le C \, \bbe\bigg[\sqrt{\int_0^t\|\Sigma(X_s,u_s)\|^2ds}\bigg]
\le C \, \bbe\bigg[\sqrt{\int_0^t\|\nabla H(X_s)\|^2\|\sigma (X_s,u_s)\|^2ds}\bigg]
\\ &\le C \, \bbe\bigg[\sup_{s\in [0,t]}\|\nabla H(X_s)\|\sqrt{\int_0^t\|\sigma (X_s,u_s)\|^2ds}\bigg].
\end{align*}
The Cauchy--Schwarz inequality and conditions \eqref{int-cond-4a}, \eqref{int-cond-4b} finally prove that
$$
\bigg(\bbe\bigg[\sup_{s\in [0,t]}|M_s|\bigg]\bigg)^2
\le C \, \bbe\bigg[\sup_{s\in [0,t]}\|\nabla H(X_s)\|^2\bigg]\, \bbe\bigg[\int_0^t\|\sigma (X_s,u_s)\|^2ds\bigg] < \infty
$$
for all $t \in \bbr_+$. 
This establishes (ii). 
The result (iii) is now 
a consequence of Lemma \ref{lemma-supermartingale-integrable}.
\end{proof}

\begin{proposition}\label{prop-system-3}
Suppose that Assumptions \ref{ass-solutions} and \ref{ass-H-C2} are fulfilled, and that the system \eqref{system} has the passive local supermartingale property with respect to the storage function $H$. Moreover, suppose that $H$ is bounded from below and that for all $x_0 \in \bbr^d$, all $u \in \cala$ and all $t \in \bbr_+$ it holds that
\begin{align}\label{int-part-u-Y}
\bbe\bigg[\int_0^t|\langle u_s, Y_s\rangle | ds\bigg]<\infty,
\end{align}
where $Y_s = f(X_s,u_s)$, $s \geq 0$ with $X = X^{(x_0,u)}$ denoting the strong solution to the system \eqref{system} with $X_0 = x_0$ and control $u$. Then the system \eqref{system} has the passive supermartingale property with respect to the storage function $H$.
\end{proposition}

\begin{remark}
Note that Assumption \ref{ass-int-cond-for-passivity} is automatically fulfilled here, because condition~\eqref{int-part-u-Y} implies~\eqref{int-cond-for-passivity}.
\end{remark}

\begin{proof}
Let $x_0 \in \bbr^d$ and $u \in \cala$ be arbitrary. By the definition \eqref{Z-definition} of the local martingale $Z = Z^{(x_0,u)}$ we have $Z = H(X) + A$, where $A \in \calv$ is given by
\begin{align*}
A_t := -\int_0^t \la u_s,Y_s \ra ds, \quad t \in \bbr_+.
\end{align*}
Note that $|A| \leq B$, where $B \in \calv^+$ is given by
\begin{align*}
B_t := \int_0^t |\la u_s,Y_s \ra| ds, \quad t \in \bbr_+.
\end{align*}
By \eqref{int-part-u-Y} we have $B_t \in L^1$ for all $t \in \bbr_+$. Now, Lemma \ref{lemma-supermartingale-below} applies and gives us that $Z$ is a supermartingale.
\end{proof}

\section{Further results about the passive supermartingale property}\label{sec-existence-sol}

In this section we present sufficient conditions for existence and uniqueness of strong solutions as well as further results concerning the passive supermartingale property and stochastic passivity of the system \eqref{system}.

\begin{assumption}\label{ass-existence-suff-Lip}
For each $R \in \bbr_+$ there is a measurable mapping $L^R : \bbr^n \to \bbr_+$ such that for all $\omega \in \Omega$, all $t \in \bbr_+$, and every control process $u \in \cala$ we have
\begin{align}\label{cond-LR-u}
\int_0^t | L^R(u_s(\omega)) |^2 ds < \infty,
\end{align}
and for all $v \in \bbr^n$, and all $x,y \in \bbr^d$ with $\| x \|, \| y \| \leq R$ we have
\begin{align}\label{cond-1}
\| b(x,v)-b(y,v) \| + \| \sigma(x,v) - \sigma(y,v) \| \leq L^R(v) \| x-y \|.
\end{align}
\end{assumption}

\begin{assumption}\label{ass-existence-suff-growth}
There is a measurable mapping $L : \bbr^n \to \bbr_+$ such that for all $\omega \in \Omega$, all $t \in \bbr_+$, and every control process $u \in \cala$ we have
\begin{align}\label{cond-L-u}
\int_0^t | L(u_s(\omega)) |^2 ds < \infty,
\end{align}
and for all $v \in \bbr^n$ and all $x \in \bbr^d$ we have
\begin{align}\label{cond-2}
\| b(x,v) \| + \| \sigma(x,v) \| \leq L(v) (1 + \| x \|).
\end{align}
\end{assumption}

\begin{proposition}\label{prop-system-existence}
Suppose that Assumptions \ref{ass-b-sigma}, \ref{ass-existence-suff-Lip} and \ref{ass-existence-suff-growth} are fulfilled. Then Assumption \ref{ass-solutions} is fulfilled; that is, existence and uniqueness of strong solutions for the system \eqref{system} holds true.
\end{proposition}

\begin{proof}
Let $u \in \cala$ be arbitrary. We define the coefficients
\begin{align*}
&\widetilde{b} : \Omega \times \bbr_+ \times \bbr^d \to \bbr^d,
\\ &\widetilde{\sigma} : \Omega \times \bbr_+ \times \bbr^d \to \bbr^{d \times k}
\end{align*}
by setting
\begin{align*}
\widetilde{b}(\omega,t,x) &:= b(x,u_t(\omega)),
\\ \widetilde{\sigma}(\omega,t,x) &:= \sigma(x,u_t(\omega)).
\end{align*}
Then for each $x \in \bbr^d$ the processes
\begin{align*}
&\Omega \times \bbr_+ \to \bbr^d, \quad (\omega,t) \mapsto \widetilde{b}(\omega,t,x),
\\ &\Omega \times \bbr_+ \to \bbr^{d \times k}, \quad (\omega,t) \mapsto \widetilde{\sigma}(\omega,t,x)
\end{align*}
are progressively measurable, because $u$ is progressively measurable and $b$ and $\sigma$ are continuous. We define the nonnegative processes $\widetilde{L}^R$, $R \in \bbr_+$ and $\widetilde{L}$ as
\begin{align*}
\widetilde{L}_t^R(\omega) &:= L^R(u_t(\omega)), \quad (\omega,t) \in \Omega \times \bbr_+,
\\ \widetilde{L}_t(\omega) &:= L(u_t(\omega)), \quad (\omega,t) \in \Omega \times \bbr_+.
\end{align*}
Then $\widetilde{L}^R$ and $\widetilde{L}$ are adapted, because $u$ is adapted and $L^R$ and $L$ are measurable. Let $\omega \in \Omega$ and $t \in \bbr_+$ be arbitrary. By \eqref{cond-LR-u} for every $R \in \bbr_+$ we have
\begin{align*}
\int_0^t | \widetilde{L}_s^R(\omega) |^2 ds = \int_0^t | L^R(u_s(\omega)) |^2 ds < \infty,
\end{align*}
showing \eqref{cond-LR-SDE}, and similarly, by \eqref{cond-L-u} we have
\begin{align*}
\int_0^t | \widetilde{L}_s(\omega) |^2 ds = \int_0^t | L(u_s(\omega)) |^2 ds < \infty,
\end{align*}
showing \eqref{cond-L-SDE}. Furthermore, using \eqref{cond-1} for all $\omega \in \Omega$, all $t,R \in \bbr_+$ and all $x,y \in \bbr^d$ with $\| x \|, \| y \| \leq R$ we have
\begin{align*}
&\| \widetilde{b}(\omega,t,x)-\widetilde{b}(\omega,t,y) \| + \| \widetilde{\sigma}(\omega,t,x) - \widetilde{\sigma}(\omega,t,y) \|
\\ &= \| b(x,u_t(\omega))-b(y,u_t(\omega)) \| + \| \sigma(x,u_t(\omega)) - \sigma(y,u_t(\omega)) \|
\\ &\leq L^R(u_t(\omega)) \| x-y \| = \widetilde{L}_t^R(\omega) \| x-y \|,
\end{align*}
showing \eqref{cond-1-SDE}. Moreover, using \eqref{cond-2} for all $\omega \in \Omega$, all $t \in \bbr_+$ and all $x \in \bbr^d$ we have
\begin{align*}
&\| \widetilde{b}(\omega,t,x) \| + \| \widetilde{\sigma}(\omega,t,x) \| = \| b(x,u_t(\omega)) \| + \| \sigma(x,u_t(\omega)) \|
\\ &\leq L(u_t(\omega)) (1 + \| x \|) = \widetilde{L}_t(\omega) (1 + \| x \|),
\end{align*}
showing \eqref{cond-2-SDE}. Consequently, applying Theorem \ref{thm-SDE-2} proves that existence and uniqueness of strong solutions for the system \eqref{system} holds true.
\end{proof}

We proceed with the existence of moments of the system \eqref{system}. For this purpose, we require the following auxiliary result.

\begin{lemma}\label{lemma-BDG-multi}
For each $p \in [2,\infty)$ and each $T \in \bbr_+$ there is a constant $C_{T,p} > 0$ such that for every $\bbr^{d \times k}$-valued progressively measurable process $\Phi$, and every bounded stopping time $\tau \leq T$ such that
\begin{align*}
\int_0^{\tau} \| \Phi_s \|^2 ds < \infty \quad \text{$\bbp$-almost surely}
\end{align*}
we have
\begin{align*}
\bbe \Bigg[ \sup_{t \in [0,\tau]} \bigg\| \int_0^t \Phi_s dW_s \bigg\|^p \Bigg] &\leq C_{T,p} \cdot \bbe \bigg[ \int_0^{\tau} \| \Phi_s \|^p ds \bigg].
\end{align*}
\end{lemma}

\begin{proof}
For each dimension $e \in \bbn$ there are constants $m_e,M_e > 0$ such that
\begin{align}\label{norms-p-2}
m_e \| x \|_p^p \leq \| x \|^p \leq M_e \| x \|_p^p \quad \forall x \in \bbr^e.
\end{align}
In the following calculation, for a matrix $A \in \bbr^{d \times k}$ and $i=1,\ldots,d$ we agree to denote by $A^i := (A^{i1},\ldots,A^{ik}) \in \bbr^k$ the $i$-th row of $A$. Then by \eqref{norms-p-2}, the Burkholder--Davis--Gundy inequality and H\"{o}lder's inequality we obtain
\begin{align*}
\bbe \Bigg[ \sup_{t \in [0,\tau]} \bigg\| \int_0^t \Phi_s dW_s \bigg\|^p \Bigg] 
& \leq M_d \sum_{i=1}^d \bbe \Bigg[ \sup_{t \in [0,\tau]} \bigg| \int_0^t \Phi_s^i dW_s \bigg|^p \Bigg]
= M_d \sum_{i=1}^d \bbe \Bigg[ \sup_{t \in [0,\tau]} \bigg| \sum_{j=1}^k \int_0^t \Phi_s^{ij} dW_s^j \bigg|^p \Bigg]
\\ &\leq M_d c_p \sum_{i=1}^d \bbe \Bigg[ \bigg( \sum_{j=1}^k \int_0^{\tau} |\Phi_s^{ij}|^2 ds \bigg)^{p/2} \Bigg]
\leq M_d c_p T^{\frac{p}{2} - 1} \sum_{i=1}^d \bbe \bigg[ \sum_{j=1}^k \int_0^{\tau} |\Phi_s^{ij}|^p ds \bigg]
\\ &\leq M_d c_p T^{\frac{p}{2} - 1} m_{dk}^{-1} \bbe \bigg[ \int_0^{\tau} \| \Phi_s \|^p ds \bigg],
\end{align*}
where the constant $c_p > 0$ stems from the Burkholder--Davis--Gundy inequality.
\end{proof}

\begin{assumption}\label{ass-existence-stronger}
There are a measurable mapping $\kappa_1 : \bbr^n \to \bbr_+$ and a constant $\kappa_2 \in \bbr_+$ such that for all $\omega \in \Omega$, all $t \in \bbr_+$, and every control process $u \in \cala$ we have
\begin{align}\label{cond-kappa-u}
\int_0^t | \kappa_1(u_s(\omega)) |^2 ds < \infty,
\end{align}
and for all $v \in \bbr^n$ and all $x \in \bbr^d$ we have
\begin{align}\label{growth-moments}
\| b(x,v) \| + \| \sigma(x,v) \| \leq \kappa_1(v) + \kappa_2 \| x \|.
\end{align}
\end{assumption}

\begin{remark}\label{rem-ass-solutions}
Note that Assumption \ref{ass-existence-stronger} implies Assumption \ref{ass-existence-suff-growth}. To see this, we define the measurable mapping $L : \bbr^n \to \bbr_+$ as $L(v) := \max \{ \kappa_1(v), \kappa_2 \}$. Then \eqref{cond-kappa-u} implies \eqref{cond-L-u}, and \eqref{growth-moments} implies \eqref{cond-2}.
\end{remark}

\begin{proposition}\label{prop-moments}
Suppose that Assumptions \ref{ass-b-sigma}, \ref{ass-existence-suff-Lip} and \ref{ass-existence-stronger} are fulfilled. Then existence and uniqueness of strong solutions for the system \eqref{system} holds true. Moreover, let $p \in [2,\infty)$ be such that
\begin{align}\label{moment-cond}
\bbe \bigg[ \int_0^t |\kappa_1(u_s)|^p ds \bigg] < \infty \quad \forall u \in \cala \quad \forall t \in \bbr_+.
\end{align}
Then for each $T \in \bbr_+$ there is a constant $C_{T,p} \in \bbr_+$ such that for all $x_0 \in \bbr^d$ and all $u \in \cala$ we have
\begin{align}\label{E-sup}
\bbe \bigg[ \sup_{t \in [0,T]} \| X_t \|^p \bigg] \leq C_{T,p} \big( 1 + \| x_0 \|^p \big),
\end{align}
where $X = X^{(x_0,u)}$ denotes the solution to the system \eqref{system} with $X_0 = x_0$ and control $u$.
\end{proposition}

\begin{proof}
Existence and uniqueness of strong solutions for the system \eqref{system} is a consequence of Remark \ref{rem-ass-solutions} and Proposition \ref{prop-system-existence}. Now, let $x_0 \in \bbr^d$ and $u \in \cala$ be arbitrary, and denote by $X$ a solution to the system \eqref{system} with $X_0 = x_0$ and control $u$. Furthermore, let $T \in \bbr_+$ be arbitrary, and let $\tau$ be an arbitrary stopping time. Then we have
\begin{align*}
&\sup_{s \in [0,t]} \| X_s^{\tau} \|^p = \sup_{s \in [0,t \wedge \tau]} \| X_s \|^p
\\ &\leq 3^{p-1} \bigg( \| x_0 \|^p + \sup_{s \in [0,t \wedge \tau]} \bigg\| \int_0^s b(X_r,u_r) dr \bigg\|^p + \sup_{s \in [0,t \wedge \tau]} \bigg\| \int_0^s \sigma(X_r,u_r) d W_r \bigg\|^p \bigg).
\end{align*}
By H\"{o}lder's inequality and \eqref{growth-moments} we have
\begin{align*}
&\bbe \Bigg[ \sup_{s \in [0,t \wedge \tau]} \bigg\| \int_0^s b(X_r,u_r) dr \bigg\|^p \Bigg] \leq \bbe \Bigg[ \sup_{s \in [0,t \wedge \tau]} s^{p-1} \int_0^s \| b(X_r,u_r) \|^p dr \Bigg]
\\ &\leq T^{p-1} \, \bbe \bigg[ \int_0^{t \wedge \tau} \big( \kappa_1(u_r) + \kappa_2 \| X_r \| \big)^p dr \bigg].
\end{align*}
By Lemma \ref{lemma-BDG-multi} and \eqref{growth-moments} there is a constant $C_{T,p} > 0$ such that
\begin{align*}
&\bbe \Bigg[ \sup_{s \in [0,t \wedge \tau]} \bigg\| \int_0^s \sigma(X_r,u_r) d W_r \bigg\|^p \Bigg] \leq C_{T,p} \cdot \bbe \bigg[ \int_0^{t \wedge \tau} \| \sigma(X_r,u_r) \|^p ds \bigg]
\\ &\leq C_{T,p} \cdot \bbe \bigg[ \int_0^{t \wedge \tau} \big( \kappa_1(u_r) + \kappa_2 \| X_r \| \big)^p dr \bigg].
\end{align*}
Moreover, we have
\begin{align*}
&\bbe \bigg[ \int_0^{t \wedge \tau} \big( \kappa_1(u_r) + \kappa_2 \| X_r \| \big)^p dr \bigg]
\leq 2^{p-1} \bigg( \bbe \bigg[ \int_0^t \kappa_1(u_s)^p ds \bigg] + \kappa_2^p \cdot \bbe \bigg[ \int_0^{t \wedge \tau} \| X_r \|^p dr \bigg] \bigg)
\end{align*}
as well as
\begin{align*}
&\bbe \bigg[ \int_0^{t \wedge \tau} \| X_s \|^p ds \bigg] \leq \bbe \bigg[ \int_0^{t \wedge \tau} \sup_{r \in [0,s]} \| X_r \|^p ds \bigg]
\leq \bbe \bigg[ \int_0^{t} \sup_{r \in [0,s]} \| X_r^{\tau} \|^p ds \bigg] = \int_0^t \bbe \bigg[ \sup_{r \in [0,s]} \| X_r^{\tau} \|^p \bigg] ds.
\end{align*}
Consequently, taking into account \eqref{moment-cond}, by the previous estimates there is a constant $K_{T,p} \in \bbr_+$ such that for every stopping time $\tau$ we have
\begin{align}\label{E-sup-in-proof}
\bbe \bigg[ \sup_{s \in [0,t]} \| X_s^{\tau} \|^p \bigg] \leq K_{T,p} \bigg( \| x_0 \|^p + 1 + \int_0^t \bbe \bigg[ \sup_{r \in [0,s]} \| X_r^{\tau} \|^p \bigg] ds \bigg), \quad t \in [0,T].
\end{align}
Now, we define the localizing sequence of stopping times $(\tau_n)_{n \in \bbn}$ as
\begin{align*}
\tau_n := \inf \{ t \in \bbr_+ : \| X_t \| \geq n \}.
\end{align*}
Let $n \in \bbn$ with $n > \| x_0 \|$ and $T \in \bbr_+$ be arbitrary. By \eqref{E-sup-in-proof} and Gronwall's lemma there is a constant $C_{T,p} \in \bbr_+$ such that
\begin{align*}
\bbe \bigg[ \sup_{t \in [0,T]} \| X_t^{\tau_n} \|^p \bigg] \leq C_{T,p} \big( 1 + \| x_0 \|^p \big).
\end{align*}
Using Fatou's lemma we obtain
\begin{align*}
&\bbe \bigg[ \sup_{t \in [0,T]} \| X_t \|^p \bigg] = \bbe \bigg[ \lim_{n \to \infty} \sup_{t \in [0,T]} \| X_t^{\tau_n} \|^p \bigg]
\leq \liminf_{n \to \infty} \, \bbe \bigg[ \sup_{t \in [0,T]} \| X_t^{\tau_n} \|^p \bigg] \leq C_T \big( 1 + \| x_0 \|^p \big),
\end{align*}
completing the proof.
\end{proof}

\begin{theorem}\label{thm-H-p}
Suppose that Assumptions \ref{ass-H-C2}, \ref{ass-b-sigma}, \ref{ass-f}, \ref{ass-existence-suff-Lip} and Assumption \ref{ass-existence-stronger} with $\kappa_1(v) = K \| v \|$, $v \in \bbr^n$, for some $K > 0$ are fulfilled. Moreover, suppose there is $p \in [2,\infty)$ such that we have
\begin{align}\label{u-moment-p}
\bbe \bigg[ \int_0^t \| u_s \|^p ds \bigg] < \infty \quad \forall u \in \cala \quad \forall t \in \bbr_+,
\end{align}
and for some constant $C > 0$ we have
\begin{align}\label{f-growth}
\| f(x,v) \| &\leq C \big( 1 + \| x \|^{p-1} + \| v \|^{p-1} \big) \quad \forall x \in \bbr^d \quad \forall v \in \bbr^n.
\\ \label{H-derivative-1} \| \nabla H(x) \| &\leq C ( 1 + \| x \|^{p/2} ), \quad \forall x \in \bbr^d,
\\ \label{H-derivative-2} \| D^2 H (x) \| &\leq C ( 1 + \| x \|^{p-2} ), \quad \forall x \in \bbr^d.
\end{align}
Then existence and uniqueness of strong solutions for the system \eqref{system} holds true, \eqref{int-cond-1}--\eqref{int-cond-4b} hold true, and for each $T \in \bbr_+$ there is a constant $C_{T,p} > 0$ such that for all $x_0 \in \bbr^d$ and all $u \in \cala$ we have the moment estimate \eqref{E-sup}. Moreover, if in addition Assumption \ref{ass-control-cont-in-zero} is fulfilled, then the following statements are equivalent:
\begin{enumerate}
\item[(i)] The system \eqref{system} has the passive local supermartingale property with respect to the storage function~$H$.

\item[(ii)] The system \eqref{system} has the passive supermartingale property with respect to the storage function~$H$.

\item[(iii)] We have \eqref{cond-drift}.
\end{enumerate}
If these equivalent conditions are fulfilled, then the system \eqref{system} is stochastically passive with respect to the storage function $H$.
\end{theorem}

\begin{proof}
According to Proposition \ref{prop-moments} existence and uniqueness of strong solutions for the system \eqref{system} holds true, and for each $T \in \bbr_+$ there is a constant $C_{T,p} > 0$ such that for all $x_0 \in \bbr^d$ and all $u \in \cala$ we have \eqref{E-sup}.

Now, let $x_0 \in \bbr^d$ and $u \in \cala$ be arbitrary, and denote by $X = X^{(x_0,u)}$ the solution to the system \eqref{system} with $X_0 = x_0$ and control $u$. Moreover, let $t \in \bbr_+$ be arbitrary. Note that for all $0 < q \leq r < \infty$ we have
\begin{align}\label{est-qr-xi}
\xi^q \leq 1 + \xi^r \quad \forall \xi \in \bbr_+,
\end{align}
and hence
\begin{align}\label{est-qr-x}
1 + \| x \|^q \leq 2 ( 1 + \| x \|^r ) \quad \forall x \in \bbr^d.
\end{align}
Since $p \geq 2$, we have $\frac{p}{2} \leq p-1$. Therefore, by \eqref{H-derivative-1} and \eqref{est-qr-x} we have
\begin{align}\label{H-derivative-1b}
\| \nabla H(x) \| &\leq 2C ( 1 + \| x \|^{p-1} ), \quad \forall x \in \bbr^d.
\end{align}
Moreover, by \eqref{growth-moments} we may assume that
\begin{align}\label{b-sigma-growth}
\| b(x,v) \| + \| \sigma(x,v) \| \leq C ( \| x \| + \| v \| ) \quad \forall x \in \bbr^d \quad \forall v \in \bbr^n.
\end{align}
Let $q \in (1,\infty)$ be the unique number such that $p^{-1} + q^{-1} = 1$. By \eqref{H-derivative-1b}, \eqref{b-sigma-growth}, H\"{o}lder's inequality and \eqref{E-sup}, \eqref{u-moment-p} we obtain
\begin{align*}
&\int_0^t \bbe \big[ \| \nabla H(X_s) \| \cdot \| b(X_s,u_s) \| \big] ds \leq 2C^2 \int_0^t \bbe \big[ ( 1 + \| X_s \|^{p-1} ) \cdot ( \| X_s \| + \| u_s \| ) \big] ds
\\ &\leq 2C^2 \int_0^t \bbe \big[ \| X_s \| + \| u_s \| \big] ds
 + 2C^2 \int_0^t \bbe \big[ \| X_s \|^p \big] ds + 2C^2 \int_0^t \bbe \big[ \| X_s \|^p \big]^{1/q} \, \bbe \big[ \| u_s \|^p \big]^{1/p} ds < \infty,
\end{align*}
where we note that by \eqref{est-qr-xi} we have
\begin{equation}\label{est-1-p}
\begin{aligned}
&\int_0^t \bbe \big[ \| X_s \|^p \big]^{1/q} \, \bbe \big[ \| u_s \|^p \big]^{1/p} ds \leq \int_0^t \Big( 1 + \bbe \big[ \| X_s \|^p \big] \Big) \Big( 1 + \bbe \big[ \| u_s \|^p \big] \Big) ds
\\ &\leq t + \int_0^t \bbe \big[ \| X_s \|^p \big] ds + \int_0^t \bbe \big[ \| u_s \|^p \big] ds + \bbe \bigg[ \sup_{s \in [0,t]} \| X_s \|^p \bigg] \int_0^t \bbe \big[ \| u_s \|^p \big] ds < \infty.
\end{aligned}
\end{equation}
Let $r \in (1,\infty]$ be the unique number such that $(\frac{p}{2})^{-1} + r^{-1} = 1$. Then by \eqref{b-sigma-growth}, \eqref{H-derivative-2} and \eqref{E-sup}, \eqref{u-moment-p} we have
\begin{align*}
&\int_0^t \bbe \big[ \| \sigma(X_s,u_s) \|^2 \cdot \| D^2 H(X_s) \| \big] ds \leq 2C^3 \int_0^t \bbe \big[  ( \| X_s \|^2 + \| u_s \|^2 ) ( 1 + \| X_s \|^{p-2} ) \big]
\\ &\leq 2C^3 \int_0^t \bbe \big[ \| X_s \|^2 + \| u_s \|^2 \big] ds
+ 2C^3 \int_0^t \bbe \big[ \| X_s \|^p \big] ds + 2C^3 \int_0^t \bbe \big[ \| u_s \|^p \big]^{2/p} \bbe \big[ \| X_s \|^p \big]^{1/r} ds < \infty,
\end{align*}
where with the help of \eqref{est-qr-xi} the last integral is estimated as in \eqref{est-1-p}. By \eqref{H-derivative-1} and \eqref{E-sup} we have
\begin{align*}
&\bbe\bigg[\sup_{s\in [0,t]}\|\nabla H(X_s)\|^2\bigg] \leq C^2 \bigg( 1 + \bbe\bigg[\sup_{s\in [0,t]}\| X_s \|^p \bigg] \bigg) < \infty.
\end{align*}
Moreover, by \eqref{b-sigma-growth} and \eqref{E-sup}, \eqref{u-moment-p} we have
\begin{align*}
\bbe\bigg[\int_0^t\|\sigma (X_s,u_s)\|^2ds\bigg] \leq 2 C^2 \int_0^t \bbe \big[ \| X_s \|^2 + \| u_s \|^2 \big] < \infty.
\end{align*}
By \eqref{f-growth}, H\"{o}lder's inequality and \eqref{E-sup}, \eqref{u-moment-p} we obtain
\begin{align*}
\int_0^t \bbe \big[ \| u_s \| \cdot \| f(X_s,u_s) \| \big] ds
&\leq C \int_0^t \bbe \big[ \| u_s \| \cdot ( \| X_s \|^{p-1} + \| u_s \|^{p-1} ) \big] ds
\\ &\leq C \int_0^t \bbe \big[ \| X_s \|^p \big]^{1/q} \, \bbe \big[ \| u_s \|^p \big]^{1/p} ds + C \int_0^t \bbe \big[ \| u_s \|^{p} \big] ds < \infty,
\end{align*}
where with the help of \eqref{est-qr-xi} the first integral is estimated as in \eqref{est-1-p}. Thus, applying Theorem \ref{thm-system-2}, Proposition \ref{prop-passive-SMP} and Remark \ref{rem-concepts} concludes the proof.
\end{proof}

If the storage function $H$ is bounded from below, then the result simplifies as follows.

\begin{theorem}\label{thm-f-p}
Suppose that Assumptions \ref{ass-H-C2}, \ref{ass-b-sigma}, \ref{ass-f}, \ref{ass-control-cont-in-zero}, \ref{ass-existence-suff-Lip} and Assumption \ref{ass-existence-stronger} with $\kappa_1(v) = K \| v \|$, $v \in \bbr^n$ for some $K > 0$ are fulfilled, and that $H$ is bounded from below. Moreover, suppose there is $p \in [2,\infty)$ such that \eqref{u-moment-p} is satisfied, and for some constant $C > 0$ we have \eqref{f-growth}. Then existence and uniqueness of strong solutions for the system \eqref{system} holds true, and for each $T \in \bbr_+$ there is a constant $C_{T,p} > 0$ such that for all $x_0 \in \bbr^d$ and all $u \in \cala$ we have the moment estimate \eqref{E-sup}. Moreover, the following statements are equivalent:
\begin{enumerate}
\item[(i)] The system \eqref{system} has the passive local supermartingale property with respect to the storage function~$H$.

\item[(ii)] The system \eqref{system} has the passive supermartingale property with respect to the storage function~$H$.

\item[(iii)] We have \eqref{cond-drift}.
\end{enumerate}
If these equivalent conditions are fulfilled, then the system \eqref{system} is stochastically passive with respect to the storage function $H$.
\end{theorem}

\begin{proof}
According to Proposition \ref{prop-moments} existence and uniqueness of strong solutions for the system \eqref{system} holds true, and for each $T \in \bbr_+$ there is a constant $C_{T,p} > 0$ such that for all $x_0 \in \bbr^d$ and all $u \in \cala$ we have \eqref{E-sup}.

Now, let $x_0 \in \bbr^d$ and $u \in \cala$ be arbitrary, and denote by $X = X^{(x_0,u)}$ the solution to the system \eqref{system} with $X_0 = x_0$ and control $u$. Furthermore, we set $Y := f(X,u)$. Let $q \in (1,\infty)$ be the unique number such that $p^{-1} + q^{-1} = 1$. Proceeding as in the proof of Theorem \ref{thm-H-p}, for all $t \in \bbr_+$ we obtain
\begin{align*}
\bbe \bigg[ \int_0^t |\la u_s,Y_s \ra| ds \bigg] \leq \int_0^t \bbe \big[ \| u_s \| \cdot \| f(X_s,u_s) \| \big] ds < \infty.
\end{align*}
Thus, applying Theorem \ref{thm-system-2}, Proposition \ref{prop-system-3} and Remark \ref{rem-concepts} concludes the proof.
\end{proof}

\section{Uncontrolled stochastic differential equations}\label{sec-uncontrolled}

In this section we apply our findings to the particular situation of uncontrolled stochastic differential equations. More precisely, we consider an $\bbr^d$-valued SDE of the form
\begin{align}\label{SDE}
\left\{
\begin{array}{rcl}
dX_t & = & b(X_t)dt + \sigma(X_t) dW_t
\\ X_0 & = & x_0
\end{array}
\right.
\end{align}
driven by an $\bbr^k$-valued Wiener process $W$ with continuous coefficients $b : \bbr^d \to \bbr^d$ and $\sigma : \bbr^d \to \bbr^{d \times k}$. We also fix a mapping $H : \bbr^d \to \bbr$. The following assumption is in particular satisfied if the coefficients $b$ and $\sigma$ are Lipschitz continuous.

\begin{assumption}\label{ass-SDE-1}
We suppose that existence and uniqueness of strong solutions for the SDE \eqref{SDE} holds true.
\end{assumption}

In what follows, we will also assume that the storage function $H$ is of class $C^2$.

\begin{definition}
We define $\call H : \bbr^d \to \bbr$ as
\begin{align*}
\call H(x) := \la \nabla H(x), b(x) \ra + \frac{1}{2} \tr \big( \sigma(x) \sigma(x)^{\top} D^2 H(x) \big).
\end{align*}
Furthermore, we define $\Sigma : \bbr^d \to \bbr^k$ as
\begin{align*}
\Sigma(x) := \nabla H(x)^{\top} \cdot \sigma(x).
\end{align*}
\end{definition}

\begin{corollary}
Suppose that Assumption \ref{ass-SDE-1} is fulfilled, and that $H$ is of class $C^2$. Then the following statements are equivalent:
\begin{enumerate}
\item[(i)] For each $x_0 \in \bbr$ we have $H(X) \in \calv^-$, where $X = X^{(x_0)}$ denotes the strong solution to the SDE \eqref{SDE} with $X_0 = x_0$.

\item[(ii)] We have
\begin{align}\label{SDE-cond-1}
\call H(x) &\leq 0 \quad \forall x \in \bbr^d,
\\ \label{SDE-cond-2} \Sigma(x) &= 0 \quad \forall x \in \bbr^d.
\end{align}
\end{enumerate}
\end{corollary}

\begin{proof}
This is an immediate consequence of Theorem \ref{thm-system-1}.
\end{proof}

\begin{corollary}
Suppose that Assumption \ref{ass-SDE-1} is fulfilled, and that $H$ is of class $C^2$. Moreover, we assume that $\sigma$ is of class $C^1$. Then the following statements are equivalent:
\begin{enumerate}
\item[(i)] For each $x_0 \in \bbr$ we have $H(X) \in \calv^-$, where $X = X^{(x_0)}$ denotes the strong solution to the SDE \eqref{SDE} with $X_0 = x_0$.

\item[(ii)] We have \eqref{SDE-cond-2} and
\begin{align*}
\bigg\la \nabla H(x), b(x) - \frac{1}{2} \sum_{j=1}^m D \sigma^j(x) \sigma^j(x) \bigg\ra \leq 0 \quad \forall x \in \bbr^d.
\end{align*}
\end{enumerate}
\end{corollary}

\begin{proof}
This is an immediate consequence of Proposition \ref{prop-system-1b}.
\end{proof}

\begin{corollary}
Suppose that Assumption \ref{ass-SDE-1} is fulfilled, and that $H$ is of class $C^2$. Then the following statements are equivalent:
\begin{enumerate}
\item[(i)] For each $x_0 \in \bbr$ the process $H(X)$ is a local supermartingale, where $X = X^{(x_0)}$ denotes the strong solution to the SDE \eqref{SDE} with $X_0 = x_0$.

\item[(ii)] We have \eqref{SDE-cond-1}.
\end{enumerate}
\end{corollary}

\begin{proof}
This is an immediate consequence of Theorem \ref{thm-system-2}.
\end{proof}

\begin{corollary}
Suppose that $b$ and $\sigma$ are Lipschitz continuous, and that $H$ is of class $C^2$. Moreover, we assume that $H$ is bounded from below or that for some $p \in [2,\infty)$ we have \eqref{H-derivative-1} and \eqref{H-derivative-2}. Then the following statements are equivalent:
\begin{enumerate}
\item[(i)] For each $x_0 \in \bbr$ the process $H(X)$ is a local supermartingale, where $X = X^{(x_0)}$ denotes the strong solution to the SDE \eqref{SDE} with $X_0 = x_0$.

\item[(ii)] For each $x_0 \in \bbr$ the process $H(X)$ is a supermartingale, where $X = X^{(x_0)}$ denotes the strong solution to the SDE \eqref{SDE} with $X_0 = x_0$.

\item[(iii)] We have \eqref{SDE-cond-1}.
\end{enumerate}
If these equivalent conditions are fulfilled, then for each $x_0 \in \bbr$ we have $H(X_t) \in L^1$ for each $t \in \bbr_+$ and the mapping $\bbr_+ \to \bbr$, $t \mapsto \bbe[H(X_t)]$ is decreasing, where $X = X^{(x_0)}$ denotes the strong solution to the SDE \eqref{SDE} with $X_0 = x_0$.
\end{corollary}

\begin{proof}
Since $b$ and $\sigma$ are Lipschitz continuous, they also satisfy the linear growth condition
\begin{align*}
\| b(x) \| + \| \sigma(x) \| \leq K ( 1 + \| x \| ) \quad  \forall x \in \bbr^d
\end{align*}
for some constant $K > 0$. Therefore, the result is an immediate consequence of Theorems \ref{thm-H-p} and \ref{thm-f-p}.
\end{proof}

\section{Stochastic linear time-invariant input-state-output systems}\label{sec-SLTIS}

In this section we consider the case where the system \eqref{system} is a stochastic, linear, time-invariant input-state-output system (SLTIS).

\begin{assumption}\label{ass-linear-system}
    We suppose that the following conditions are fulfilled:
    \begin{enumerate}
        \item There exist $A \in \bbr^{d \times d}$, $B \in \bbr^{d \times n}$ such that for all $x \in \bbr^d$, $v\in\bbr^n$ it holds that $b(x,v) = Ax + Bv$.

        \item For each $j\in\{1,\dots,k\}$ there exist $\mathfrak{A}(j) \in \bbr^{d \times d}$, $\mathfrak{B}(j) \in \bbr^{d \times n}$ such that for all $x \in \bbr^d$, $v\in\bbr^n$ it holds that
        $\sigma^j(x,v) = \mathfrak{A}(j) x + \mathfrak{B}(j) v$.

        \item There exist $C \in \bbr^{n \times d}$, $D \in \bbr^{n \times n}$ such that for all $x \in \bbr^d$, $v\in\bbr^n$ it holds that $f(x,v) = Cx + Dv$.
    \end{enumerate}
\end{assumption}

Note that under Assumption \ref{ass-linear-system}, the system \eqref{system} takes the form
\begin{align}\label{eq:SLTIS}
\left\{
\begin{array}{rcl}
dX_t & = & (A X_t + B u_t) dt + \sum_{j=1}^k ( \mathfrak{A}(j) X_t + \mathfrak{B}(j) u_t ) dW_t^j
\\ Y_t & = & C X_t + D u_t
\\ X_0 & = & x_0 \in \bbr^d
\\ u & \in & \cala.
\end{array}
\right.
\end{align}
We also call $A,B,(\fA(j))_{j=1,\ldots,k},(\fB(j))_{j=1,\ldots,k},C,D$ the system matrices.

\begin{remark}\label{rem-linear-assumptions}
Suppose that Assumption \ref{ass-linear-system} is fulfilled.
\begin{enumerate}
\item In the sequel, we will often consider the storage function $H(x) = \frac{1}{2} \la Qx,x \ra$, $x\in\bbr^d$ for some symmetric matrix $Q \in \bbr^{d \times d}$. Then Assumption \ref{ass-H-C2} is satisfied, and for all $x \in \bbr^d$ we have
\begin{align*}
D H(x) v &= \la Qx,v \ra \quad \forall v \in \bbr^d,
\\ D^2 H(x)(v,w) &= \la Qv,w \ra \quad \forall v,w \in \bbr^d.
\end{align*}
Moreover, \eqref{H-derivative-1} and \eqref{H-derivative-2} hold true. 

\item Assumptions \ref{ass-b-sigma} and \ref{ass-f} are fulfilled, and we have
\begin{align*}
\| f(x,v) \| \leq K ( \| x \| + \| v \| ) \quad \forall x \in \bbr^d \quad \forall v \in \bbr^n,
\end{align*}
where $K := \| C \| + \| D \|$.
In particular, \eqref{f-growth} holds true. 

\item We can write $\sigma : \bbr^d \times \bbr^n \to \bbr^{d \times k}$ as
\begin{align*}
\sigma(x,v) = \mathfrak{A} x + \mathfrak{B} v \quad \forall x \in \bbr^d \quad \forall v \in \bbr^n
\end{align*}
with continuous linear operators $\mathfrak{A} \in L(\bbr^d,\bbr^{d \times k})$ and $\mathfrak{B} \in L(\bbr^n,\bbr^{d \times k})$. Therefore, Assumption \ref{ass-existence-suff-Lip} is satisfied with
\begin{align*}
L^R(v) := \| A \| + \| \mathfrak{A} \|,
\end{align*}
independent of $R \in \bbr_+$ and $v \in \bbr^n$. Furthermore, if
\begin{align}\label{int-cond-u-linear}
\bbe \bigg[ \int_0^t \| u_s \|^2 ds \bigg] < \infty \quad \forall u \in \cala \quad \forall t \in \bbr_+,
\end{align}
then Assumption \ref{ass-existence-stronger} is fulfilled with
\begin{align*}
\kappa_1(v) &:= \big( \| B \| + \| \mathfrak{B} \| \big) \| v \| \quad \forall v \in \bbr^n,
\\ \kappa_2 &:= \| A \| + \| \mathfrak{A} \|.
\end{align*}
Thus, using Proposition \ref{prop-moments}, under the additional condition \eqref{int-cond-u-linear} existence and uniqueness of strong solutions for the SLTIS \eqref{system} holds true, and for each $T \in \bbr_+$ there is a constant $C_{T} > 0$ such that for all $x_0 \in \bbr^d$ and all $u \in \cala$ we have the moment estimate
\begin{align}\label{moment-estimate-linear-system}
\bbe \bigg[ \sup_{t \in [0,T]} \| X_t \|^2 \bigg] \leq C_{T} \big( 1 + \| x_0 \|^2 \big),
\end{align}
where $X = X^{(x_0,u)}$ denotes the strong solution to the SLTIS~\eqref{system} with $X_0 = x_0$ and control $u$.

\item 
Note that under the additional condition \eqref{int-cond-u-linear} we have that item (2), the Cauchy--Schwarz inequality, and \eqref{moment-estimate-linear-system} imply for all $x_0 \in \bbr^d$ and all $u \in \cala$ that 
\begin{equation*}
    \bbe\bigg[ \int_0^t \lvert \langle u_s , f(X_s^{(x_0,u)},u_s) \rangle \rvert ds \bigg] < \infty \quad \forall t \in \bbr_+,
\end{equation*}
where $X^{(x_0,u)}$ denotes the strong solution to the SLTIS~\eqref{system} with $X^{(x_0,u)}_0=x_0$ and control $u$. 
In particular, Assumption \ref{ass-int-cond-for-passivity} is satisfied under the additional condition \eqref{int-cond-u-linear}.

\item Consider the storage function $H(x) = \frac{1}{2} \la Qx,x \ra$, $x\in\bbr^d$, for some symmetric matrix $Q \in \bbr^{d \times d}$ and suppose that condition \eqref{int-cond-u-linear} is fulfilled. We then have from \eqref{moment-estimate-linear-system} and item (4) that 
for all $x_0 \in \bbr^d$, $u \in \cala$, and $t \in \bbr_+$ it holds that $Z_t^{(x_0,u)} \in L^1$. 
\end{enumerate}
\end{remark}

In the setting of Assumption~\ref{ass-linear-system}
with a quadratic storage function as in item~(1) in Remark~\ref{rem-linear-assumptions}, 
we have the following representations for 
\eqref{generator} and \eqref{capital-sigma}.

\begin{lemma}\label{lemma-generator-linear-case}
    Suppose that Assumption \ref{ass-linear-system} is fulfilled and assume that there exists a symmetric matrix $Q \in \bbr^{d \times d}$ such that $H(x) = \frac{1}{2} \la Qx,x \ra$, $x\in\bbr^d$.
    Then we have for all $x \in \bbr^d$ and all $v \in \bbr^n$ that 
    \begin{align*}
        (\Sigma(x,v))_j & = \langle x, Q \fA(j) x + Q \fB(j) v \rangle 
        \quad \forall j=1,\ldots,k,\\
        \call H(x,v) & = 
            \big\langle x, \big( QA + \tfrac12 \textstyle{\sum_{j=1}^k} \fA^\top(j) Q \fA(j) \big) x \big\rangle 
            + \big\langle x, \big( QB - C^\top + \textstyle{\sum_{j=1}^k} \fA^\top(j) Q \fB(j) \big) v \big\rangle \\
            & \quad + \big\langle v, \big(\tfrac12 \textstyle{\sum_{j=1}^k} \fB^\top(j) Q \fB(j) - D \big) v \big\rangle .
    \end{align*}
\end{lemma}

\begin{proof}
    The representation for $\Sigma$ follows from \eqref{capital-sigma}, the fact that $\forall x \in \bbr^d\colon \nabla H(x) = Qx$, the symmetry of $Q$, and item~(2) in Assumption~\ref{ass-linear-system}. 
    To obtain the representation for $\call H$, note that \eqref{generator}, item~(1) in Remark~\ref{rem-linear-assumptions}, Assumption~\ref{ass-linear-system}, and the fact that $Q$ is symmetric imply for all $x \in \bbr^d$ and all $v \in \bbr^n$ that 
    \begin{align*}
    \call H(x,v) &= \la \nabla H(x), b(x,v) \ra + \frac{1}{2} \sum_{j=1}^k D^2 H(x) ( \sigma^j(x,v), \sigma^j(x,v) ) - \la v,f(x,v) \ra
    \\ &= \la Qx, b(x,v) \ra 
    - \la v,f(x,v) \ra 
    + \frac{1}{2} \sum_{j=1}^k \la Q \sigma^j(x,v), \sigma^j(x,v) \ra 
    \\ &= \la Qx, Ax + Bv \ra
    - \la v,Cx + Dv \ra
    \\ &\quad + \frac{1}{2} \sum_{j=1}^k \Big( \la Q \mathfrak{A}(j) x, \mathfrak{A}(j) x \ra + \la Q \mathfrak{A}(j) x, \mathfrak{B}(j) v \ra + \la Q \mathfrak{B}(j) v, \mathfrak{A}(j) x \ra 
    + \la Q \mathfrak{B}(j) v, \mathfrak{B}(j) v \ra \Big)
    \\ &= \la x, QA x \ra + \la x, QBv \ra
    - \la x,C^{\top} v \ra - \la v,Dv \ra
    \\ &\quad + \frac{1}{2} \sum_{j=1}^k \Big( \la x, \mathfrak{A}^{\top}(j) Q \mathfrak{A}(j) x \ra + 2 \la x, \mathfrak{A}^{\top}(j) Q \mathfrak{B}(j) v \ra + \la v, \mathfrak{B}^{\top}(j) Q \mathfrak{B}(j) v \ra \Big)
    \\
    & = 
    \Big\langle x, \Big( QA + \frac12 \sum_{j=1}^k \fA^\top(j) Q \fA(j) \Big) x \Big\rangle 
    + \Big\langle v, \Big(\frac12 \sum_{j=1}^k \fB^\top(j) Q \fB(j) - D \Big) v \Big\rangle
    \\
    & \quad + \Big\langle x, \Big( QB - C^\top + \sum_{j=1}^k \fA^\top(j) Q \fB(j) \Big) v \Big\rangle .
    \end{align*}
    This completes the proof.
\end{proof}

Under Assumption \ref{ass-linear-system} we introduce for each symmetric matrix $Q\in \bbr^{d\times d}$ the $\bbr^{(d+n)\times (d+n)}$-matrix
	\begin{equation}\label{drift-cond-matrix}
		\mathfrak{M}_Q :=
		\begin{pmatrix}
			QA+A^\top Q+\sum_{j=1}^k\fA^\top(j) Q \fA(j) & QB-C^\top +\sum_{j=1}^k\fA^\top(j) Q \fB(j)\\
			B^\top Q-C +\sum_{j=1}^k\fB^\top(j) Q \fA(j) & - (D+D^\top) +\sum_{j=1}^k\fB^\top(j) Q \fB(j)
		\end{pmatrix}
        .
	\end{equation}

\begin{remark}
    \begin{enumerate}
        \item It is well known that passivity of deterministic linear systems can be characterized by a linear matrix inequality (LMI) (see, e.g., \cite[Section 5]{willems1972dissipative}). The condition $\mathfrak{M}_Q\le 0$ with $\mathfrak{M}_Q$ given by \eqref{drift-cond-matrix} extends this deterministic LMI to the stochastic case (in the sense that if $\fA(j)=0$ and $\fB(j)=0$ for all $j\in \{1,\ldots,k\}$, the inequality $\mathfrak{M}_Q\le 0$ reduces to the inequality from the deterministic theory). Below we will show that, similar to the deterministic case, under suitable conditions the negative semidefiniteness of $\mathfrak{M}_Q$ characterizes stochastic passivity properties of the SLTIS~\eqref{system}.

        \item 
        Linear matrix inequalities $\mathfrak{M}_Q\le 0$ with $\mathfrak{M}_Q$ of the form in \eqref{drift-cond-matrix} arise in infinite-horizon optimal stochastic linear-quadratic control problems (see, e.g., \cite{rami2000linear}). This connection is also reflected by our proof of Proposition~\ref{propo-linear-quadraticstorage-1} (see also Remark~\ref{rem:available_storage}), where we solve the optimal stochastic linear-quadratic control problem of computing the available storage to obtain a quadratic storage function for a passive stochastic system.
    \end{enumerate}
\end{remark}

\begin{lemma}\label{rem-linear-generator}
    Suppose that Assumption \ref{ass-linear-system} is fulfilled. Assume that there exists a symmetric matrix $Q \in \bbr^{d \times d}$ such that $H(x) = \frac{1}{2} \la Qx,x \ra$, $x\in\bbr^d$.
    Then 
    it holds for all $x\in \bbr^d$ and all $v\in\bbr^n$ that 
    \begin{equation*}
        \call H(x,v)
        = \frac12 
        \begin{pmatrix}
            x^\top & v^\top
        \end{pmatrix}
        \mathfrak{M}_Q 
        \begin{pmatrix}
            x\\
            v
        \end{pmatrix}
        .
    \end{equation*}
\end{lemma}

\begin{proof}
    The claim follows from Lemma~\ref{lemma-generator-linear-case}.
\end{proof}

\begin{corollary}\label{lemma-linear-driftcond-lmi}
    Suppose that Assumption \ref{ass-linear-system} is fulfilled. Assume that there exists a symmetric matrix $Q \in \bbr^{d \times d}$ such that $H(x) = \frac{1}{2} \la Qx,x \ra$, $x\in\bbr^d$.
    Then
    condition~\eqref{cond-drift} is satisfied if and only if
    $\mathfrak{M}_Q\le 0$.
\end{corollary}

\begin{proof}
    This is an immediate consequence of 
    Lemma~\ref{rem-linear-generator}.
\end{proof}

By combining Corollary~\ref{lemma-linear-driftcond-lmi}
and Theorem~\ref{thm-H-p}, we obtain a characterization of the passive supermartingale property of an SLTIS with a quadratic storage function in terms of a linear matrix inequality.

\begin{theorem}\label{thm-linear-1}
Suppose that Assumptions \ref{ass-control-cont-in-zero}, \ref{ass-linear-system} and condition \eqref{int-cond-u-linear} are fulfilled, and that there exists a symmetric matrix $Q \in \bbr^{d \times d}$ such that $H(x) = \frac{1}{2} \la Qx,x \ra$, $x\in\bbr^d$. Then the following statements are equivalent:
\begin{enumerate}
\item[(i)] The SLTIS \eqref{system} has the passive local supermartingale property with respect to the storage function~$H$.

\item[(ii)] The SLTIS \eqref{system} has the passive supermartingale property with respect to the storage function~$H$.

\item[(iii)] We have $\mathfrak{M}_Q \le 0$.
\end{enumerate}
If these equivalent conditions are fulfilled, then the SLTIS \eqref{system} is stochastically passive with respect to the storage function~$H$.
\end{theorem}

\begin{proof}
Taking into account Remark \ref{rem-linear-assumptions} and Corollary \ref{lemma-linear-driftcond-lmi}, this is a consequence of Theorem \ref{thm-H-p}.
\end{proof}

\begin{remark}
Under the (typical) additional condition $Q \geq 0$ the statement of Theorem \ref{thm-linear-1} is also a consequence of Theorem \ref{thm-f-p}.
\end{remark}

We next characterize passivity of SLTIS in Theorem~\ref{thm-linear-passive}. To this end, we in Lemma~\ref{lemma-linear-cond-sigma} first obtain a condition that is equivalent to condition~\eqref{cond-diffusion}.

\begin{lemma}\label{lemma-linear-cond-sigma}
Suppose that Assumption \ref{ass-linear-system} is fulfilled. Assume that there exists a symmetric matrix $Q \in \bbr^{d \times d}$ such that $H(x) = \frac{1}{2} \la Qx,x \ra$, $x\in\bbr^d$.
Then
condition~\eqref{cond-diffusion} is satisfied if and only if $Q \mathfrak{A}(j)$ is skew-symmetric and $Q \mathfrak{B}(j) = 0$ for all $j=1,\ldots,k$.
\end{lemma}

\begin{proof}
    Note that two matrices $\widetilde A \in \bbr^{d \times d}$ and $\widetilde B \in \bbr^{d \times n}$ satisfy 
    \begin{align*}
    \la x, \widetilde Ax + \widetilde Bv \ra = 0 \quad \forall x \in \bbr^d \quad \forall v \in \bbr^n
    \end{align*}
    if and only if $\widetilde A$ is skew-symmetric and $\widetilde B = 0$.
    Therefore, the claim follows from the representation in  Lemma~\ref{lemma-generator-linear-case}.
\end{proof}

\begin{theorem}\label{thm-linear-passive}
Suppose that Assumptions \ref{ass-control-cont-in-zero}, \ref{ass-linear-system} and condition \eqref{int-cond-u-linear} are fulfilled, and that there exists a symmetric matrix $Q \in \bbr^{d \times d}$ such that $H(x) = \frac{1}{2} \la Qx,x \ra$, $x\in\bbr^d$. Then the following statements are equivalent:
\begin{enumerate}
\item[(i)] The SLTIS \eqref{system} is passive with respect to the storage function $H$.

\item[(ii)] $Q \mathfrak{A}(j)$ is skew-symmetric and $Q \mathfrak{B}(j) = 0$ for all $j=1,\ldots,k$, and we have
\begin{align*}
\left(
\begin{array}{cc}
QA + A^{\top} Q - Q \sum_{j=1}^k \mathfrak{A}(j)^2 & QB - C^{\top}
\\ B^{\top} Q - C & -D - D^{\top}
\end{array}
\right) \leq 0.
\end{align*}
\end{enumerate}
\end{theorem}

\begin{proof}
    Note first that if $Q \mathfrak{A}(j)$ is skew-symmetric and $Q \mathfrak{B}(j) = 0$ for all $j=1,\ldots,k$, then we have that
    \begin{equation*}
    \begin{split}
		\mathfrak{M}_Q
        & =
		\begin{pmatrix}
			QA+A^\top Q+\sum_{j=1}^k (Q\fA(j))^\top \fA(j) & QB-C^\top +\sum_{j=1}^k\fA^\top(j) Q \fB(j)\\
			B^\top Q-C +\sum_{j=1}^k (\fA^\top(j) Q \fB(j))^\top  & - D-D^\top +\sum_{j=1}^k\fB^\top(j) Q \fB(j)
		\end{pmatrix}
        \\
        & =
        \begin{pmatrix}
            QA + A^{\top} Q - Q \sum_{j=1}^k \mathfrak{A}(j)^2 & QB - C^{\top}\\
            B^{\top} Q - C & -D - D^{\top}
        \end{pmatrix}
        .
    \end{split}
	\end{equation*}

    (i) $\Rightarrow$ (ii):
    Since the SLTIS \eqref{system} is passive, Theorem~\ref{thm-system-1} implies that the conditions \eqref{cond-drift} and \eqref{cond-diffusion} are satisfied.
    From \eqref{cond-drift} and Corollary~\ref{lemma-linear-driftcond-lmi} we obtain that $\mathfrak{M}_Q \le 0$.
    Combining this, \eqref{cond-diffusion}, and Lemma~\ref{lemma-linear-cond-sigma} yields (ii).

    (ii) $\Rightarrow$ (i):
    Since $\mathfrak{M}_Q \le 0$, we obtain from Corollary~\ref{lemma-linear-driftcond-lmi} that \eqref{cond-drift} is satisfied.
    Lemma~\ref{lemma-linear-cond-sigma} and the fact that $Q \mathfrak{A}(j)$ is skew-symmetric and $Q \mathfrak{B}(j) = 0$ for all $j=1,\ldots,k$ show that \eqref{cond-diffusion} holds.
    Hence, Theorem~\ref{thm-system-1} ensures (i).
\end{proof}

\section{Stochastic passivity results for linear time-invariant systems}\label{sec-SLTIS-2}

In this section we consider an SLTIS as in Section~\ref{sec-SLTIS} and examine stochastic passivity.
We first introduce the following assumption, which strengthens Assumption \ref{ass-control-cont-in-zero} by additionally requiring an integrability property of the control for small time horizons.

\begin{assumption}\label{ass-control-cont-in-zero-and-integrable}
For each $v \in \bbr^n$ there exist a control process $u \in \cala$ and $\delta\in(0,\infty)$ such  that $u$ is continuous at zero with $u_0 = v$ 
and 
$\bbe[\sup_{t\in[0,\delta]} \lVert u_t \rVert^2]<\infty$. 
\end{assumption}

Under this assumption we derive that stochastic passivity of an SLTIS with a quadratic storage function as in item (1) in Remark~\ref{rem-linear-assumptions} implies that the associated matrix defined in \eqref{drift-cond-matrix} is negative semidefinite.

\begin{lemma}\label{lemma-linear-stoch-passive-lmi}
    Suppose that Assumptions \ref{ass-linear-system}, \ref{ass-control-cont-in-zero-and-integrable}  and condition \eqref{int-cond-u-linear} are fulfilled, that there exists a symmetric matrix $Q \in \bbr^{d \times d}$ such that $H(x) = \frac{1}{2} \la Qx,x \ra$, $x\in\bbr^d$, 
    and that for all $x_0\in\bbr^d$ and all $u\in\cala$ there exists $\varepsilon \in (0,\infty)$ such that 
    \begin{align}\label{stoch-passivity-close-to-zero}
        \sup_{t\in[0,\varepsilon]} \bbe[Z_{t}^{(x_0,u)}]\le H(x_0) .
    \end{align}
	Then it holds that $\mathfrak{M}_Q \le 0$.
\end{lemma}

\begin{proof}
    Let $x_0 \in \bbr^d$ and $v\in \bbr^n$ be arbitrary. 
    Due to Assumption~\ref{ass-control-cont-in-zero-and-integrable}, 
    there exist $u\in\cala$ with $u_0=v$ and $\delta\in(0,\infty)$ such that 
    $u$ is continuous at zero
    and satisfies $\bbe[\sup_{t\in[0,\delta]} \lVert u_t \rVert^2]<\infty$. 
    Note that Lemma~\ref{lemma-generator}
    implies for all $t\in [0,\infty)$ that 
    \begin{equation*}
        \begin{split}
            Z_{t}^{(x_0,u)} - H(x_0) 
            & = \int_0^t \call H(X_{s}^{(x_0,u)},u_s) ds 
            + \int_0^t \Sigma(X_{s}^{(x_0,u)},u_s) dW_s .
        \end{split}
    \end{equation*}
    Moreover, observe that 
    Remark~\ref{rem-linear-assumptions}, Theorem~\ref{thm-H-p}, and Proposition~\ref{prop-passive-SMP} show that 
    the process $M$ with $M_t = \int_0^t \Sigma(X_{s}^{(x_0,u)},u_s) dW_s$, $t\in[0,\infty)$, is a martingale and that for all $t\in[0,\infty)$ we have that $\int_0^t \call H(X_{s}^{(x_0,u)},u_s) ds \in L^1$. 
    From Lemma~\ref{rem-linear-generator} we thus obtain for all $t\in[0,\infty)$ that 
    \begin{equation}\label{eq:Ito_quadratic}
        \begin{split}
            \bbe[Z_{t}^{(x_0,u)}] - H(x_0) 
            & = \frac{1}{2}\bbe\bigg[\int_{0}^{t}
			\begin{pmatrix}
				(X_s^{(x_0,u)})^\top & u_s^\top
			\end{pmatrix}
			\mathfrak{M}_Q
			\begin{pmatrix}
				X_s^{(x_0,u)}\\
				u_s
			\end{pmatrix}
			ds\bigg].
        \end{split}
    \end{equation}
    It follows from \eqref{stoch-passivity-close-to-zero} for all $h\in [0,\varepsilon]$ that 
        \begin{equation}\label{proof-stoch-pass-lmi-nonpositive}
        \begin{split}
            0\ge \frac{2}{h} \big( \bbe[Z_{h}^{(x_0,u)}] - H(x_0) \big)
            & = \frac{1}{h}\bbe\bigg[\int_{0}^{h}
			\begin{pmatrix}
				(X_s^{(x_0,u)})^\top & u_s^\top
			\end{pmatrix}
			\mathfrak{M}_Q
			\begin{pmatrix}
				X_s^{(x_0,u)}\\
				u_s
			\end{pmatrix}
			ds\bigg].
        \end{split}
    \end{equation} 
	Since $X^{(x_0,u)}$ is a continuous process with initial value $X^{(x_0,u)}_0=x_0$ and $u$ is continuous at zero with $u_0=v$, we can show that $\IP$-a.s.\ it holds that
	\begin{equation*}
		\lim_{h\downarrow 0}
		\frac{1}{h} \int_{0}^{h}
		\begin{pmatrix}
			(X_s^{(x_0,u)})^\top & u_s^\top
		\end{pmatrix}
		\mathfrak{M}_Q
		\begin{pmatrix}
			X_s^{(x_0,u)}\\
			u_s
		\end{pmatrix}
		ds
		= \begin{pmatrix}
			x_0^\top & v^\top
		\end{pmatrix}
		\mathfrak{M}_Q
		\begin{pmatrix}
			x_0\\
			v
		\end{pmatrix} .
	\end{equation*}
	Moreover, \eqref{moment-estimate-linear-system} and the fact that $\bbe[\sup_{t\in[0,\delta]} \lVert u_t \rVert^2]<\infty$ ensure that
	\begin{equation*}
    \sup_{s\in[0,\delta]}
    \bigg\lvert
		\begin{pmatrix}
			(X_s^{(x_0,u)})^\top & u_s^\top
		\end{pmatrix}
		\mathfrak{M}_Q
		\begin{pmatrix}
			X_s^{(x_0,u)}\\
			u_s
		\end{pmatrix}
		\bigg\rvert
        \in L^1.
	\end{equation*} 
	Therefore, the dominated convergence theorem implies that
	\begin{equation*}
		\begin{split}
			& \lim_{h\downarrow 0}
			\frac{1}{h} \E\bigg[\int_{0}^{h}
			\begin{pmatrix}
				(X_s^{(x_0,u)})^\top & u_s^\top
			\end{pmatrix}
			\mathfrak{M}_Q
			\begin{pmatrix}
				X_s^{(x_0,u)}\\
				u_s
			\end{pmatrix}
			ds\bigg]
			=
			\begin{pmatrix}
				x_0^\top & v^\top
			\end{pmatrix}
			\mathfrak{M}_Q
			\begin{pmatrix}
				x_0\\
				v
			\end{pmatrix}
			.
		\end{split}
	\end{equation*}
	It thus follows from \eqref{proof-stoch-pass-lmi-nonpositive} that
	\begin{equation*}
		0 \ge
		\begin{pmatrix}
			x_0^\top & v^\top
		\end{pmatrix}
		\mathfrak{M}_Q
		\begin{pmatrix}
			x_0\\
			v
		\end{pmatrix}
		.
	\end{equation*}
    This proves that $\mathfrak{M}_Q\le 0$.
\end{proof}

Under Assumption~\ref{ass-control-cont-in-zero-and-integrable} we can now add stochastic passivity as a further equivalent statement in Theorem~\ref{thm-linear-1}.

\begin{theorem}\label{thm-linear-3}
    Suppose that Assumptions \ref{ass-linear-system}, \ref{ass-control-cont-in-zero-and-integrable}  and condition \eqref{int-cond-u-linear} are fulfilled,
    and that there exists a symmetric matrix $Q \in \bbr^{d \times d}$ such that $H(x) = \frac{1}{2} \la Qx,x \ra$, $x\in\bbr^d$.
    Then the following statements are equivalent:
    \begin{enumerate}
        \item[(i)] The SLTIS \eqref{system} is stochastically passive with respect to the storage function~$H$.

        \item[(ii)] The SLTIS \eqref{system} has the passive locale supermartingale property with respect to the storage function~$H$.

        \item[(iii)] The SLTIS \eqref{system} has the passive supermartingale property with respect to the storage function~$H$.
            
        \item[(iv)] We have
        $\mathfrak{M}_Q \le 0$.
    \end{enumerate}
\end{theorem}

\begin{proof}
    Note that Theorem~\ref{thm-linear-1} shows that 
    (ii), (iii), and (iv) are equivalent.
    Moreover, it holds that (iii) implies (i) (see also the third statement in Remark~\ref{rem-concepts}). 
    It hence remains to show that (i) implies (iv).   
    To this end, 
    let the SLTIS~\eqref{system} be stochastically passive with respect to the storage function~$H$.
    Then for all $x_0\in\bbr^d$, all $u\in\cala$  and all $t\in\bbr_+$ it holds that 
    \begin{equation*}
        H(x_0)=\bbe[Z_0^{(x_0,u)}]
        \ge \bbe[Z_t^{(x_0,u)}] .
    \end{equation*}
    Hence, Lemma~\ref{lemma-linear-stoch-passive-lmi} yields that $\mathfrak{M}_Q\le 0$.
\end{proof}

In the next result we provide a sufficient condition under which a stochastically passive SLTIS 
(with respect to some storage function that is bounded from below)
is stochastically passive with respect to a quadratic storage function.
For this result we use techniques similar to those in \cite[Section 3]{ait2000well} established under the assumption that the system is stabilizable (see also Remark \ref{rem:available_storage} below).

\begin{proposition}\label{propo-linear-quadraticstorage-1}
    Assume that $(\calf_t)_{t\ge 0}$ is the natural filtration of $W$ augmented by all the $\IP$-null sets in~$\calf$. 
    For all $T\ge 0$ let $\cala_T$ denote the set of all $(\cF_t)_{t\in[0,T]}$-progressively measurable processes $u\colon \Omega \times [0,T] \to \bbr^n$ that satisfy 
    \begin{equation*}
        \bbe\bigg[ \int_0^T \lVert u_s \rVert^2 ds \bigg] < \infty.
    \end{equation*}
    Suppose that Assumption \ref{ass-linear-system} 
    and condition \eqref{int-cond-u-linear}
    are fulfilled and that for all $T >0$ 
    and all $u \in \cala_T$ 
    there exists $\tilde u\in\cala$ such that $\tilde u(\omega,t)=u(\omega,t)$ for all $\omega \in \Omega$ and all $t\in[0,T]$. 
    Assume that the SLTIS~\eqref{system} is stochastically passive with respect to 
    a storage function $H$ that is bounded from below. 
    Then
    there exists a positive semidefinite matrix $Q \in \R^{d\times d}$ such that the SLTIS~ \eqref{system} is stochastically passive with respect to $\cH\colon \R^d \to [0,\infty)$, $\cH(x)=\frac12 \langle Q x, x \rangle$. Moreover, $Q$ is the minimal solution of $\mathfrak M_Q\le 0$ in the set of positive semidefinite matrices.
\end{proposition}

\begin{proof}
    For $T\in [0,\infty)$ we consider the finite-horizon stochastic linear-quadratic optimal control problem
\begin{equation}\label{eq:lq_control_prob}
    V(T,x)=\inf_{u\in \cala_T}\E\left[\int_0^T\langle u_s,Y^{(x,u)}_s\rangle ds \right], \quad x\in \R^d.
\end{equation}
    The fact that the system is stochastically passive with respect to $H$, and boundedness of $H$ from below, 
    show 
    for 
    all $T\in [0,\infty)$ and all $x\in \R^d$ that
\begin{equation}\label{eq:val_fct_bdd_below_c}
\begin{split}
    V(T,x) & = \inf_{u\in \cA_T}\E\bigg[\int_0^T\langle u_s,Y^{(x,u)}_s\rangle ds \bigg] 
    \ge \inf_{u\in \cA_T} \Big( \E\big[H\big(X^{(x,u)}_T\big)\big]-H(x) \Big) 
    \ge \inf_{y\in\mathbb{R}^d} H(y) - H(x)
    >-\infty.
\end{split}
\end{equation}
In particular, for any $T\in[0,\infty)$, the problem \eqref{eq:lq_control_prob} is finite at zero (see Definition~2.1 in \cite{Sun2016open}).
Therefore, Proposition~5.1 in \cite{Sun2016open} implies that for all $T\in[0,\infty)$ there exists a symmetric matrix $K(T)\in \R^{d\times d}$ such that for all $x\in \R^d$ it holds
\begin{equation}\label{eq:VeqKtT}
    V(T,x)=\langle x, K(T) x\rangle.
\end{equation}
Note that $V$ is nonpositive since one can always choose the zero control. Hence, we obtain from \eqref{eq:VeqKtT} for all $T\in [0,\infty)$ that $K(T)$ is negative semidefinite.

Furthermore, observe that for all $x\in \R^d$ the function $(0,\infty)\ni T\mapsto V(T,x)\in \R$ is nonincreasing. 
Moreover, we have by \eqref{eq:val_fct_bdd_below_c} that for all $x\in \R^d$ the function $[0,\infty)\ni T\mapsto V(T,x)\in \R$ is bounded from below.
This implies that there exists a symmetric matrix $\overline K$ such that $\lim_{T\to \infty}K(T)=\overline{K}$.
Let $Q=-2\overline{K}$ and $\cH(x)=\frac12 \langle x, Qx \rangle$, $x\in\R^d$.
Note that for all $x \in \bbr^d$ it holds that $\langle x, Qx \rangle  = -2 \lim_{T\to\infty} \langle x, K(T) x \rangle \ge 0$.

Due to item~(4) in Remark~\ref{rem-linear-assumptions}, the fact that $0\in\cA_T$ with $\bbe[\int_0^T \langle 0, Y_s^{(x,0)} \rangle ds ] =0$ for all $x\in\bbr^d$ and all $T\in[0,\infty)$, the time-invariance of the system, and \eqref{eq:VeqKtT} we can argue similar to the proof of the dynamic programming principle in~\cite{tang2015dpp} (see, in particular, item (i) in Theorem~4.1 and items (i) and (ii) in Theorem~2.4 in~\cite{tang2015dpp}) 
to obtain for all $T\in [0,\infty)$, $h\in [0,T]$, $x\in \R^d$ that
\begin{equation}\label{eq:dpp}
\begin{split}
    \langle x, K(T) x\rangle& = V(T,x) 
    =\inf_{u\in \cA_T}\bbe\bigg[\int_0^{h}\langle u_s,Y^{(x,u)}_s\rangle ds + \langle X^{(x,u)}_{h}, K(T-h) X^{(x,u)}_{h}\rangle \bigg].
\end{split}
\end{equation}
We take the limit $T\to \infty$ in \eqref{eq:dpp} and apply Fatou's lemma (recall that $K$ is negative semidefinite) to obtain for all $h\in [0,\infty)$, $x\in \R^d$, and $u\in \cA$ that
\begin{equation*}
\begin{split}
    \langle x, \overline K x\rangle
    &= \lim_{T\to \infty}\langle x, K(T) x\rangle 
    \le \lim_{T\to \infty}\E\bigg[\int_0^{h}\langle u_s,Y^{(x,u)}_s\rangle ds + \langle X^{(x,u)}_{h}, K(T-h) X^{(x,u)}_{h}\rangle \bigg]\\
    &\le \E\bigg[\int_0^{h}\langle u_s,Y^{(x,u)}_s\rangle ds + \langle X^{(x,u)}_{h}, \overline{K} X^{(x,u)}_{h}\rangle \bigg].
\end{split}
\end{equation*}
This shows for all $h\in [0,\infty)$, $x\in \R^d$, and $u\in \cA$ that
\begin{equation*}
    \begin{split}
        \mathcal{H}(x) 
        & = - \langle x, \overline{K} x \rangle
        \ge 
        - \E\bigg[\int_0^{h}\langle u_s,Y^{(x,u)}_s\rangle ds 
        - \mathcal{H}(X^{(x,u)}_{h}) \bigg] .
    \end{split}
\end{equation*}
Therefore, Lemma~\ref{lemma-linear-stoch-passive-lmi} and Theorem~\ref{thm-linear-3}
imply that the SLTIS~\eqref{system} is stochastically passive with respect to~$\cH$ and that $\mathfrak M_Q\le 0$. It remains to show minimality of $Q$. To this end let $\tilde Q \in \R^{d\times d}$ be positive semidefinite and satisfy $\mathfrak M_{\tilde Q}\le 0$. It follows from \eqref{eq:val_fct_bdd_below_c} and an application of It\^o's formula (similarly to the derivation of \eqref{eq:Ito_quadratic}) for all $x\in \R^d$ that
\begin{equation*}
\begin{split}
&\frac{1}{2}\langle x,Qx\rangle =\mathcal H(x)=\lim_{T\to \infty}\sup_{u\in\cala_T} \bbe\bigg[-\int_0^T \langle u_s, Y_s^{(x,u)} \rangle ds \bigg]\\
&=   
\lim_{T\to \infty}\sup_{u\in\cala_T}
\frac{1}{2}
\bbe\bigg[\langle x,\tilde Qx\rangle - \langle X_T^{(x,u)},\tilde Q X_T^{(x,u)}\rangle 
+\int_0^T \Big\langle \begin{pmatrix}
 X_s^{(x,u)}\\
u_s 
\end{pmatrix}, \mathfrak{M}_{\tilde Q}\begin{pmatrix}
 X_s^{(x,u)}\\
u_s 
\end{pmatrix} \Big\rangle ds
\bigg] 
\le \frac{1}{2}\langle x,\tilde Qx\rangle.
\end{split}
\end{equation*}
This completes the proof.
\end{proof}

\begin{remark}\label{rem:available_storage}
    \begin{enumerate}
        \item 
     Note that the storage function $\cH\colon \R^d \to [0,\infty)$, $\cH(x)=\frac12 \langle Q x, x \rangle$ that we construct in Proposition~\ref{propo-linear-quadraticstorage-1} equals the so-called available storage, which is defined as 
     $H_a\colon \bbr^d \to [0,\infty]$, 
     $$H_a(x)=\sup_{u\in\cala, T>0} \bbe\bigg[-\int_0^T \langle u_s, Y_s^{(x,u)} \rangle ds \bigg].$$
    Moreover, \eqref{eq:val_fct_bdd_below_c} shows that any storage function $H$ that is bounded from below and with respect to which the system is stochastically passive satisfies for all $x\in\mathbb{R}^d$ that 
    $H(x)-\inf_{y\in\bbr^d} H(y) \ge H_a(x) \ge 0$.
    We further remark that the assumption in Proposition~\ref{propo-linear-quadraticstorage-1} that the SLTIS~\eqref{system} is stochastically passive with respect to a storage function that is bounded from below can be replaced by the assumption that the available storage is finite, i.e., $\forall x \in\mathbb{R}^d\colon H_a(x)<\infty$; the reason is that stochastic passivity is only used to obtain the finiteness of the value function in  \eqref{eq:val_fct_bdd_below_c}.
    In particular, we conclude that under the assumptions of Proposition~\ref{propo-linear-quadraticstorage-1} (with the stochastic passivity assumption removed), we have that $\forall x \in\mathbb{R}^d\colon H_a(x)<\infty$ if and only if 
    the SLTIS~\eqref{system} is stochastically passive with respect to the available storage~$H_a$.

    \item
    As an additional tool for analyzing the stochastic passivity of the SLTIS \eqref{system}, one can also consider the associated (generalized) Riccati equation
\begin{equation}\label{eq:riccati}
\begin{split}
    A^\top Q+ QA+\sum_{j=1}^k\cA^\top(j) Q \cA(j)
    -\mathcal{S}(Q)(\mathcal{N}(Q))^{\dagger}(\mathcal{S}(Q))^\top&=0,\\
    \mathcal{N}(Q)&\le 0,\\
    \mathcal{S}(Q)(\Id-\mathcal{N}(Q)(\mathcal{N}(Q))^\dagger)&=0,
\end{split}
\end{equation}
where
\begin{equation}
    \begin{split}
\mathcal{S}(Q)&=QB+\sum_{j=1}^k\cA^\top(j) Q \cB(j) -C^\top,\quad
\mathcal{N}(Q)=\sum_{j=1}^k\cB^\top(j) Q \cB(j) -(D+D^\top),
    \end{split}
\end{equation}
and $(\mathcal{N}(Q))^\dagger$ is the Moore--Penrose inverse of $\mathcal{N}(Q)$.
This Riccati equation can informally be derived from the above discussed optimal stochastic linear-quadratic control problem 
$$\sup_{u\in\cala, T>0} \bbe\bigg[-\int_0^T \langle u_s, Y_s^{(x,u)} \rangle ds \bigg]$$
(see, e.g., \cite{ait2000well}, \cite{sun2018stochastic} or \cite{sun2020stochastic} and the references therein; see also the book \cite{damm2004book} for further references and the numerical solution in the regular case). 
It follows directly from the extended Schur lemma (see, e.g., \cite[Theorem 1]{albert1969conditions} or \cite[Lemma 3]{rami2000linear}) that a solution $Q$ to the Riccati equation \eqref{eq:riccati} satisfies the LMI $\mathfrak M_Q\le 0$.
\end{enumerate}
\end{remark}

The next goal is to ensure that the matrix $Q$ constructed in Proposition~\ref{propo-linear-quadraticstorage-1} is not only positive semidefinite but positive definite. 
Note that similar results in the literature, 
such as \cite[Lemma 1]{willems1972dissipative}, \cite[Example~6]{beattie2025port} and \cite[Proposition~3.2.12]{EhKr_vandSchaft2016l2} for the deterministic case 
and \cite[Theorem~3.2]{EhKr_Haddad2023DissStochDynSys} for the stochastic case,  
contain an additional observability assumption. 
For this reason we present the following definition of stochastic observability. This concept allows us to prove in Corollary \ref{cor-linear-quadraticstorage-3} below that stochastic passivity and observability imply that $Q$ is positive definite.

\begin{definition}
    Suppose that Assumption~\ref{ass-linear-system} is fulfilled and that $0\in\cala$. 
    Let $x_0\in\bbr^d$. 
    We call $x_0$ \emph{unobservable} if 
    $Y^{(x_0,0)}_t=0$ $\IP$-a.s.\ for all $t\in[0,\infty)$.
\end{definition}

\begin{definition}\label{def-observable}
    Suppose that Assumption~\ref{ass-linear-system} is fulfilled and that $0\in\cala$.
    We say that the SLTIS \eqref{system} is \emph{observable} if $\{x_0\in\bbr^d \, \vert \, x_0 \text{ unobservable } \}=\{0\}$.
\end{definition}

\begin{remark}
    Suppose that Assumption~\ref{ass-linear-system} is fulfilled and that $0\in\cala$.
    \begin{enumerate}
        \item[(1)]
        Note that $x_0=0 \in \bbr^d$ is unobservable.

        \item[(2)]
    Note that the SLTIS \eqref{system} is observable if and only if for all $x_0\in\bbr^d$ we have that $Y^{(x_0,0)}_t=0$ $\IP$-a.s.\ for all $t\in[0,\infty)$ implies $x_0=0$.
    We moreover mention that in the literature such a system is sometimes called exactly observable (see, e.g., \cite[Definition~5]{zhang2004stabilizability}) to distinguish between different observability notions.
        \end{enumerate}
\end{remark}

Observability of the SLTIS~\eqref{system} can be 
verified 
by a rank criterion that only involves (products of) the system matrices $A,\fA(1),\ldots,\fA(k),C$. 
We present this result in Lemma~\ref{lemma-linear-observable-rank-cond} below. For completeness, we also provide its proof. For further reference, see \cite[Proof of Lemma~3]{chen2004stochastic} and the 
dissertation \cite{liu1999diss}. 
For further discussion of conditions for observability in stochastic linear systems we refer to, for example, \cite{chen2004stochastic,hou2012some,li2010unified,zhang2004stabilizability}. 
In the proof of Lemma~\ref{lemma-linear-observable-rank-cond}, we use the following result.

\begin{lemma}\label{lemma-linear-prep-observable-rank-cond}
    Suppose that Assumption~\ref{ass-linear-system} and condition \eqref{int-cond-u-linear} are fulfilled and that $0\in\cala$. 
    Let $T \in (0,\infty)$, $x_0 \in \bbr^d$, and $\widetilde{C} \in \bbr^{n\times d}$. 
    Assume that for all $s\in[0,T)$ it holds $\IP$-a.s.\ that $\widetilde{C} X_s^{(x_0,0)} = 0$.
    Then for all $s\in[0,T)$ and all $j\in\{1,\ldots,k\}$ it holds $\IP$-a.s.\ that $\widetilde{C} A X_s^{(x_0,0)} = 0$ and $\widetilde{C} \fA(j) X_s^{(x_0,0)} = 0$.
\end{lemma}

\begin{proof}
    Let $r\in[0,T)$. 
    Note that the assumption that for all $s\in[0,T)$ it holds $\IP$-a.s.\ that $\widetilde{C} X_s^{(x_0,0)} = 0$
    implies for all $t\in(r,T)$ that $\IP$-a.s.
    \begin{equation}\label{eq:proof-obs-diff-zero}
        \begin{split}
            0 & = \widetilde{C} X_t^{(x_0,0)} - \widetilde{C} X_r^{(x_0,0)}
            = \int_r^t \widetilde{C} A X_s^{(x_0,0)} ds
            + \sum_{j=1}^k \int_r^t \widetilde{C} \fA(j) X_s^{(x_0,0)} d W_s^j .
        \end{split}
    \end{equation}
    This and Proposition~\ref{prop-can-decomp} show that for all $t\in(r,T)$ and  all $l\in\{1,\ldots,n\}$ it holds $\IP$-a.s.\ that 
     \begin{equation*}
         \sum_{j=1}^k \int_r^t \big(\widetilde{C} \fA(j) X_s^{(x_0,0)}\big)_l d W_s^j = 0 .
     \end{equation*}
     In particular, also the quadratic variation of this stochastic integral is zero, which yields for all $t\in(r,T)$ and all $l\in\{1,\ldots,n\}$ that $\IP$-a.s. 
     \begin{equation}\label{eq:proof-obs-quadratic-var-zero}
        0 
        = \sum_{j=1}^k \int_r^t \big\lvert \big( \widetilde{C} \fA(j) X_s^{(x_0,0)} \big)_l \big\rvert^2 ds .
    \end{equation}
    Observe that for all $t\in(r,T)$, $l\in\{1,\ldots,n\}$ and $j\in\{1,\ldots,k\}$
    it holds $\IP$-a.s.\ that 
    \begin{equation*}
        \frac{1}{t-r} \int_r^t \big\lvert \big( \widetilde{C} \fA(j) X_s^{(x_0,0)} \big)_l \big\rvert^2 ds
        \le \lVert \widetilde{C} \fA(j) \rVert^2 \sup_{s\in[r,T]} \lVert X_s^{(x_0,0)} \rVert^2 .
    \end{equation*}
    Due to the moment estimate~\eqref{moment-estimate-linear-system}, the continuity of the state, and \eqref{eq:proof-obs-quadratic-var-zero} 
    we can hence apply the dominated convergence theorem to obtain for all $l\in\{1,\ldots,n\}$ and all $j\in\{1,\ldots,k\}$ that
    \begin{equation*}
        0 = \lim_{t\downarrow r} \bbe\bigg[ \frac{1}{t-r} \int_r^t  \big\lvert \big( \widetilde{C} \fA(j) X_s^{(x_0,0)} \big)_l \big\rvert^2 ds \bigg]
        = \bbe\big[ \big\lvert \big( \widetilde{C} \fA(j) X_r^{(x_0,0)} \big)_l \big\rvert^2 \big] .
    \end{equation*}
    This shows that $\widetilde{C} \fA(j) X_r^{(x_0,0)} = 0$ $\IP$-a.s.\ for all $j\in\{1,\ldots,k\}$.
    Moreover, note that this, continuity of the state, and \eqref{eq:proof-obs-diff-zero} imply that $\IP$-a.s.
    \begin{equation*}
        0 = \lim_{t\downarrow r} \int_r^t \widetilde{C} A X_s^{(x_0,0)} ds 
        = \widetilde{C} A X_r^{(x_0,0)} .
    \end{equation*}
    This completes the proof. 
\end{proof}

\begin{lemma}\label{lemma-linear-observable-rank-cond}
    Suppose that Assumption~\ref{ass-linear-system} and condition \eqref{int-cond-u-linear} are fulfilled and that $0\in\cala$. Let 
    \begin{equation*}
        S = \bigcup_{i\in\mathbb{N}} \bigcup_{q\in \{0,1\}^i} \bigcup_{p\in\{0,1\}^{i\times k}} \bigg\{ \prod_{l=1}^i \biggl(A^{q_l} \prod_{j=1}^k \fA^{p_{l,j}}(j)\biggr)\bigg\} .
    \end{equation*}
    Let $\{\widetilde{A}_1,\ldots,\widetilde{A}_N\}$ be any finite subset of $S$
    and define 
    \begin{equation*}
        M = (C^\top , \widetilde{A}_1^\top C^\top , \ldots, \widetilde{A}_N^\top C^\top).
    \end{equation*}
    \begin{enumerate}
        \item[(i)]
        Let $x_0\in\bbr^d$. It then holds that $x_0$ is unobservable if and only if 
        for all 
        $\widetilde{A} \in S$
        it holds that $C\widetilde{A}x_0=0$.
        
        \item[(ii)]
        If $\rank(M)=d$, then the SLTIS~\eqref{system} is observable. 
    \end{enumerate}
\end{lemma}

\begin{proof}
    To prove (i), let $x_0 \in \bbr^d$. 
    Assume first that $x_0$ is unobservable. 
    Then it holds by definition that $C X_t^{(x_0,0)}=0$ $\IP$-a.s.\ for all $t\in[0,\infty)$.
    We can thus use an induction argument relying on  
    Lemma~\ref{lemma-linear-prep-observable-rank-cond} to demonstrate for all $\widetilde{A} \in S$ that $C\widetilde{A}x_0=0$. 
    
    Assume now that for all $\widetilde{A} \in S$ it holds that $C\widetilde{A}x_0=0$.
    Consider the Picard iterations defined by 
    \begin{equation*}
    \begin{split}
        X_t^{(x_0,0,0)}&=x_0, \quad t\in[0,\infty),\\
        X_t^{(x_0,0,m)}&=
        x_0
        +\int_0^t A X_s^{(x_0,0,m-1)} ds 
        + \sum_{j=1}^k \int_0^t \fA(j) X_s^{(x_0,0,m-1)} dW_s^j , 
        \quad t\in[0,\infty),\, m\in\IN.
    \end{split}
    \end{equation*}
    We can show by induction that for all $\widetilde{A} \in S$, $m\in\IN_0$, and $t\in [0,\infty)$ it holds $\IP$-a.s.\ that $C\widetilde{A}X_t^{(x_0,0,m)}=0$. 
    Since we have for all $t\in [0,\infty)$ that $X_t^{(x_0,0)}$ is the $\IP$-a.s.\ limit of $(X_t^{(x_0,0,m)})_{m\in\IN_0}$ (see, e.g., 
    \cite[Section~5.2]{karatzas1991book}), 
    we conclude that for all $t\in [0,\infty)$ it holds $\IP$-a.s.\ that $C X_t^{(x_0,0)} = 0$.
    This shows (i).

    To prove (ii), note first that (i) ensures that any unobservable $x_0\in\bbr^d$ satisfies $M^\top x_0=0$.
    It thus holds that $\{x_0\in\bbr^d \, \vert \, x_0 \text{ unobservable } \} \subseteq \ker(M^\top)$.
    Moreover, note that 
    $$\rank(M)=\rank(M^\top)=d- \dim\ker(M^\top).$$ 
    If $\rank(M)=d$, this shows that $\ker(M^\top)=\{0\}$, and hence 
    $$\{x_0\in\bbr^d \, \vert \, x_0 \text{ unobservable} \} = \{0\}.$$
    This completes the proof.
\end{proof}

\begin{remark}
    Suppose that Assumption~\ref{ass-linear-system} and condition \eqref{int-cond-u-linear} are fulfilled and that $0\in\cala$. 
    \begin{enumerate}
    \item[(1)]
    Note that by the Cayley--Hamilton theorem, if $\overline{A} \in \{A,\fA(1),\ldots,\fA(k)\}$ and $m\in\IN$ such that $m\ge d$, then $\overline{A}^m$ can be expressed as a linear combination of the matrices $\overline{A}^{0},\overline{A}^{1},\ldots,\overline{A}^{d-1}$. 
    In Lemma~\ref{lemma-linear-observable-rank-cond}
    we therefore, for each of the matrices $A$, $\fA(1),\ldots,\fA(k)$, only need to consider exponents up to $d-1$.

    \item[(2)]
    Observe that in view 
    of item (ii) in Lemma~\ref{lemma-linear-observable-rank-cond}, 
    a sufficient condition for the SLTIS~\eqref{system} to be observable is that 
    $\rank(C^\top, A^\top C^\top,\ldots, (A^{d-1})^\top C^\top )=d$, which is equivalent to the respective deterministic linear time-invariant system (i.e., the SLTIS~\eqref{system} with $\fA(j)=0$ for all $j=1,\ldots,k$ and deterministic controls $u$) being observable (see, e.g., Theorem~24 in Section 6.2 of~\cite{sontag1998book}).
    \end{enumerate}
\end{remark}

The next result, Lemma~\ref{lemma-linear-quadraticstorage-2}, shows that if an observable SLTIS is stochastically passive with respect to a nonnegative
storage function, then this storage function is positive definite. We subsequently combine this result with  Proposition~\ref{propo-linear-quadraticstorage-1} to show that stochastic passivity with respect to any storage function bounded from below and observability ensure that the SLTIS is stochastically passive with respect to a quadratic storage function represented by a positive definite matrix (see Corollary \ref{cor-linear-quadraticstorage-3}).

\begin{assumption}\label{ass-linear-system-controls-locsquareint}
    The set of admissible controls $\cala$ is the set of all $\bbr^n$-valued progressively measurable processes $u$ that satisfy 
    \begin{equation*}
        \bbe\bigg[ \int_0^t \lVert u_s \rVert^2 ds \bigg] < \infty \quad \forall t \in \bbr_+ .
    \end{equation*}
\end{assumption}

\begin{lemma}\label{lemma-linear-quadraticstorage-2}
    Suppose that Assumptions \ref{ass-linear-system} 
    and \ref{ass-linear-system-controls-locsquareint}
    are fulfilled. 
    Assume that the SLTIS~\eqref{system} is stochastically passive with respect to a nonnegative storage function $H\colon \R^d \to [0,\infty)$.
    If the SLTIS~\eqref{system} is observable, then it holds for all $x\in\bbr^d\backslash\{0\}$ that $H(x)>0$. 
\end{lemma}

\begin{proof}
Let $x \in \R^d$ such that 
$H(x)=0$. 
Let $a=(1+\lVert D \rVert)^{-1}$ (recall Remark~\ref{rem-notation}).
It then holds that 
$\lVert -a D \rVert = (1+\lVert D \rVert)^{-1} \lVert D \rVert < 1$.
Hence, $\Id+aD=\Id-(-a)D$ is invertible. 
We can thus consider the $\R^{d}$-valued continuous process $\overline{X}=(\overline{X}_s)_{s\in [0,\infty)}$ given by 
\begin{equation}\label{eq:1648}
    \begin{split}
        d\overline{X}_t &=\big(A - a B(\Id+aD)^{-1} C \big)\overline{X}_tdt+\sum_{j=1}^k \big(\fA(j) - a \fB(j) (\Id+aD)^{-1} C \big) \overline{X}_t dW_t^j, 
        \quad t \in [0,\infty) , \\
        \overline{X}_0&=x .
    \end{split}
\end{equation}
Note that, indeed, there exists a unique strong solution of \eqref{eq:1648} and it moreover holds for all $t\in[0,\infty)$ that $\mathbb{E}[\sup_{s\in[0,t]} \lVert \overline{X}_s \rVert^2 ]< \infty$ (cf., e.g., Remark~\ref{rem-linear-assumptions}). 
Next, define $\overline{u}=-a(\Id+aD)^{-1} C \overline{X}$ and note that $\overline{u}\in\cA$. 
Furthermore, let $\overline{Y}_t=C\overline{X}_t+D\overline{u}_t$, $t\in[0,\infty)$. 
Observe that $\overline{X}$ satisfies for all $t\in [0,\infty)$ that 
\begin{equation*}
    \begin{split}
    d\overline{X}_t
    & = \big(A \overline{X}_t - Ba (\Id+aD)^{-1} C \overline{X}_t \big) dt 
    + \sum_{j=1}^k \big(\fA(j) \overline{X}_t 
    - \fB(j) a (\Id+aD)^{-1} C \overline{X}_t \big) dW_t^j \\
    & = \big(A \overline{X}_t + B\overline{u}_t \big) dt 
    + \sum_{j=1}^k \big(\fA(j) \overline{X}_t 
    + \fB(j) \overline{u}_t \big) dW_t^j .
    \end{split}
\end{equation*}
Therefore, $(\overline{X},\overline{Y})$ is the unique strong solution (cf.\ Remark~\ref{rem-linear-assumptions}) to the SLTIS~\eqref{system} in the form of Assumption~\ref{ass-linear-system} with the initial value $x$ and the control $\overline{u}$. 
Moreover, note that for all $t\in [0,\infty)$ it holds that 
\begin{equation*}
    \begin{split}
        \overline{Y}_t & = C \overline{X}_t  - a D (\Id + aD)^{-1} C \overline{X}_t 
        = \big( \Id - a D (\Id + aD)^{-1} \big) C \overline{X}_t \\
        & = \big( \Id - (\Id + a D) (\Id + aD)^{-1} + (\Id + aD)^{-1} \big) C \overline{X}_t 
        = (\Id + aD)^{-1} C \overline{X}_t \\
        & = - a^{-1} \overline{u}_t .
    \end{split}
\end{equation*}
Stochastic passivity 
of the SLTIS~\eqref{system} with respect to $H$ 
and the fact that $H$ is nonnegative hence imply 
for all $T \in [0,\infty)$ that 
\begin{equation*}
    \begin{split}
        0 
        = H(x)
        & \ge \bbe\big[H(\overline{X}_T)\big]
        - \bbe\bigg[\int_0^T \langle \overline{u}_s, \overline{Y}_s \rangle ds \bigg] 
        \ge 
        - \bbe\bigg[\int_0^T \langle \overline{u}_s, \overline{Y}_s \rangle ds \bigg] 
        = a \bbe\bigg[\int_0^T \langle \overline{Y}_s, \overline{Y}_s \rangle ds \bigg] 
        \ge 0.
    \end{split}
\end{equation*}
It follows for all $T\in(0,\infty)$ that $\bbe[\int_0^T \lVert \overline{Y}_s \rVert^2 ds]  = 0$.
Since $\overline{Y}$ is continuous, we can show that this implies $\overline{Y}_s=0$ $\IP$-a.s.\ for all $s\in[0,\infty)$. 
It then also holds that $\overline{u}_s=0$ $\IP$-a.s.\ for all $s\in[0,\infty)$. 
Therefore, $x$ is unobservable. 
Now, the assumption that the SLTIS~\eqref{system} is observable implies that $x=0$.
\end{proof}

\begin{corollary}\label{cor-linear-quadraticstorage-3}
    Suppose that Assumptions \ref{ass-linear-system} 
    and \ref{ass-linear-system-controls-locsquareint} 
    are fulfilled and that $(\calf_t)_{t\ge 0}$ is the natural filtration of $W$ augmented by all the $\IP$-null sets in $\calf$. 
    Assume that the SLTIS~\eqref{system} is observable and stochastically passive with respect to 
    a storage function $H$ that is bounded from below.
    Then
    there exists a positive definite matrix $Q \in \R^{d\times d}$ such that the SLTIS~ \eqref{system} is stochastically passive with respect to $\cH\colon \R^d \to [0,\infty)$, $\cH(x)=\frac12 \langle Q x, x \rangle$.
\end{corollary}

\begin{proof}
    This follows from combining Proposition~\ref{propo-linear-quadraticstorage-1} and Lemma~\ref{lemma-linear-quadraticstorage-2}.
\end{proof}

\section{Linear stochastic port-Hamiltonian systems}\label{sec-linear-phs}

In this section we focus on a specific subclass of SLTIS \eqref{system} which we call stochastic, linear, time-invariant port-Hamiltonian systems. In Corollary \ref{cor:pass_and_obs_imply_phs}, we show that an observable SLTIS is stochastically passive if and only if it belongs to this subclass. We thereby extend known results from the deterministic setting to the stochastic case (see, e.g., \cite[Corollary~4]{beattie2025port} and \cite[Theorem~2.3]{vanderschaft2009port}). 

\begin{definition}\label{def-SLTIPHS}
    We say that the system~\eqref{system} is a \emph{stochastic, linear, time-invariant port-Hamiltonian system (SLTI-PHS)}, if there exist $J,R,Q,\overline{\fA}(j)\in \R^{d\times d}$, $j\in \{1,\ldots,k\}$, $F,P,\fB(j)\in \R^{d\times n}$, $j\in \{1,\ldots,k\}$, $S,N\in \R^{n\times n}$ such that $J=-J^\top$, $R=R^\top$, $Q=Q^\top\ge 0$, $S=S^\top$, $N=-N^\top$,
    \begin{equation}\label{eq:matrix_ineq_pH}
       \begin{pmatrix}
            R & P\\
           P^\top & S
        \end{pmatrix}
        \ge 0,
    \end{equation}
    and for all $x\in\bbr^d$, $v\in\bbr^n$, $i\in\{1,\dots,k\}$ it holds that
    \begin{equation}\label{eq:sltipH}
    \begin{split}
        & b(x,v) = \bigg(J-R-\frac{1}{2}\sum_{j=1}^k\overline{\fA}^\top(j) Q \overline{\fA}(j)\bigg)Q x +(F-P) v ,\\
        & \sigma^i(x,v) = \overline{\fA}(i)Q x + \fB(i) v, \\
        & f(x,v) = \bigg(F+P+\sum_{j=1}^k\overline{\fA}^\top(j) Q \fB(j)\bigg)^\top Q x
        +\bigg( S+N+\frac{1}{2}\sum_{j=1}^k\fB^\top(j) Q \fB(j)\bigg) v .
    \end{split}
    \end{equation}
 To emphasize the dependence on $Q$ we sometimes also say that \eqref{system} is an \emph{SLTI-PHS with $Q$}.   
\end{definition}

Note that in Definition~\ref{def-SLTIPHS}, the system \eqref{system} takes the form
\begin{align*}
\left\{
\begin{array}{rcl}
dX_t & = & \bigg[\bigg(J-R-\frac{1}{2}\sum_{j=1}^k\overline{\fA}^\top(j) Q \overline{\fA}(j)\bigg)Q X_t +(F-P) u_t \bigg] dt 
+ \sum_{j=1}^k ( \overline{\fA}(j)Q X_t + \fB(j) u_t ) dW_t^j
\\ Y_t & = & \bigg(F+P+\sum_{j=1}^k\overline{\fA}^\top(j) Q \fB(j)\bigg)^\top Q X_t
        +\bigg( S+N+\frac{1}{2}\sum_{j=1}^k\fB^\top(j) Q \fB(j)\bigg) u_t
\\ X_0 & = & x_0 \in \bbr^d
\\ u & \in & \cala.
\end{array}
\right.
\end{align*}

It is clear that any SLTI-PHS is, in particular, an SLTIS.

\begin{remark}\label{rem-SLTIPHS-params-SLTIS}
    Let 
    the system~\eqref{system} be an SLTI-PHS. 
    Then Assumption~\ref{ass-linear-system} is satisfied
    with
    \begin{equation*}
        \begin{split}
            & A = \bigg(J-R-\frac{1}{2}\sum_{j=1}^k\overline{\fA}^\top(j) Q \overline{\fA}(j)\bigg)Q, \quad
            B = F-P, \\
            &
            \fA(j) = \overline{\fA}(j)Q \quad \text{and} \quad
            \fB(j)=\fB(j), \quad j\in\{1,\dots,k\}, \\
            & C = \bigg(F+P+\sum_{j=1}^k\overline{\fA}^\top(j) Q \fB(j)\bigg)^\top Q, \quad
            D =  S+N+\frac{1}{2}\sum_{j=1}^k\fB^\top(j) Q \fB(j) .
        \end{split}
    \end{equation*}
    Moreover, we have that 
    \begin{equation*}
        \mathfrak{M}_Q
        =-2\begin{pmatrix}
    QRQ & QP\\
       P^\top Q & S
       \end{pmatrix} .
    \end{equation*}
\end{remark}

\begin{lemma}\label{lemma-SLTIPHS-LMI}
    Let $Q\in\bbr^{d\times d}$ be positive semidefinite and 
    assume that the system~\eqref{system} is an SLTI-PHS with $Q$.
    Then 
    we have that $\mathfrak{M}_Q\le 0$.
\end{lemma}

\begin{proof}
    Using Remark~\ref{rem-SLTIPHS-params-SLTIS}, 
    we obtain that 
    \begin{equation}
    \begin{split}
    \mathfrak{M}_Q
      =-2
      \begin{pmatrix}
    Q & 0\\
       0 & \Id
       \end{pmatrix}
      \begin{pmatrix}
    R & P\\
       P^\top & S
       \end{pmatrix}
       \begin{pmatrix}
    Q & 0\\
       0 & \Id
       \end{pmatrix}.
       \end{split}
    \end{equation}
    It hence follows from \eqref{eq:matrix_ineq_pH} that $\mathfrak{M}_Q\le 0$.
\end{proof}

Under appropriate assumptions on the set of admissible controls, we obtain the following passivity properties of SLTI-PHS.

\begin{corollary}\label{cor-SLTIPHS-passivityprop}
    Let $Q\in\bbr^{d\times d}$ be positive semidefinite and 
    assume that the system~\eqref{system} is an SLTI-PHS with $Q$.
    Suppose that Assumption~\ref{ass-control-cont-in-zero}
    and condition~\eqref{int-cond-u-linear} are fulfilled. Let $H(x)=\frac12 \langle Qx, x\rangle$, $x\in\bbr^d$.
    Then the following statements hold true:
        \begin{enumerate}
        \item[(i)] If Assumption~\ref{ass-control-cont-in-zero-and-integrable} is fulfilled, then the SLTI-PHS is stochastically passive with respect to the storage function~$H$.

        \item[(ii)] The SLTI-PHS has the passive locale supermartingale property with respect to the storage function~$H$.

        \item[(iii)] The SLTI-PHS has the passive supermartingale property with respect to the storage function~$H$.
    \end{enumerate}
\end{corollary}

\begin{proof}
    This is an immediate consequence of Lemma~\ref{lemma-SLTIPHS-LMI}, Theorem~\ref{thm-linear-1}, and Theorem~\ref{thm-linear-3}.
\end{proof}

In the next lemma we provide an algebraic  condition under which an SLTIS is an SLTI-PHS.
This generalizes results from the deterministic case (see, e.g., \cite[Theorem~2.3]{vanderschaft2009port}) to the stochastic case.

\begin{lemma}\label{lemma:lmiimpliesphs}
    Suppose that Assumption~\ref{ass-linear-system} is fulfilled.
    Assume that there exists a positive semidefinite matrix
    $Q \in \R^{d\times d}$ such that 
    $\ker(Q)\subseteq \ker(A) \cap \ker(C) \cap (\cap_{j=1}^k \ker(\fA(j)))$ 
    and $\mathfrak{M}_Q\le 0$.
    Then the SLTIS~\eqref{system} is an SLTI-PHS with~$Q$.
\end{lemma}

\begin{proof}
    Observe that $\ker(Q)\subseteq \ker(A) \cap \ker(C) \cap (\cap_{j=1}^k \ker(\fA(j)))$ implies that there exist $\overline{A}, \overline{\fA}(j) \in \R^{d\times d}$, $j\in \{1,\ldots,k\}$, and
    $\overline{C} \in \R^{n\times d}$ such that
    $\overline{\fA}(j) Q = \fA(j)$, $j\in \{1,\ldots,k\}$,
    $\overline{A} Q = A$,
    and $\overline{C} Q = C$.
    Denote by $\bild(Q)$ the image of $Q$. 
    Let $G_1 \in \R^{d\times d}$ be defined by $G_1x = 0$, $x \in \R^d\backslash\bild(Q)$, 
    and $G_1x=\overline{A}x + \frac12 \sum_{j=1}^k \overline{\fA}^\top(j) Q \overline{\fA}(j)x$, $x \in \bild(Q)$.
    Let $G_3 \in \R^{n\times d}$ be defined by $G_3x = -B^\top x$, $x \in \R^d\backslash\bild(Q)$, and $G_3x=-\overline{C}x + \sum_{j=1}^k \fB^\top(j) Q \overline{\fA}(j)x$, $x \in \bild(Q)$.
    Moreover, let $G_2=B$ and $G_4=- D + \frac12 \sum_{j=1}^k \fB^\top(j) Q \fB(j)$.
    Now, define
    \begin{equation}
            G = \begin{pmatrix}
                G_1
                & G_2 \\
                G_3
                & G_4
            \end{pmatrix}
            .
    \end{equation}
    Since $\mathfrak{M}_Q\le 0$, it then holds for all $z \in \R^{d+n}$ that 
    \begin{equation}\label{eq:1703a}
        \begin{split}
            &
            z^\top
            \begin{pmatrix}
                Q & 0 \\
                0 & \Id
            \end{pmatrix}
            \big(G+G^\top\big)
            \begin{pmatrix}
                Q & 0 \\
                0 & \Id
            \end{pmatrix}
            z
            =
            z^\top
            \begin{pmatrix}
                Q(G_1 Q) + (G_1 Q)^\top Q
                & Q G_2 + Q G_3^\top \\
                G_3 Q + G_2^\top Q
                & G_4+ G_4^\top
            \end{pmatrix}
            z
            \\
            & =
            z^\top
            \begin{pmatrix}
			QA+A^\top Q+\sum_{j=1}^k\fA^\top(j) Q \fA(j) & QB-C^\top +\sum_{j=1}^k\fA^\top(j) Q \fB(j)\\
			B^\top Q-C +\sum_{j=1}^k\fB^\top(j) Q \fA(j) & - (D+D^\top) +\sum_{j=1}^k\fB^\top(j) Q \fB(j)
		      \end{pmatrix}
            z 
            \le 0 .
        \end{split}
    \end{equation}
    Moreover, $\mathfrak{M}_Q\le 0$ shows for all $x \in \R^d\backslash\bild(Q)$ and $y \in \R^n$ that 
    \begin{equation}\label{eq:1703b}
    \begin{split}
        &
        \begin{pmatrix}
            x^\top & y^\top
        \end{pmatrix}
        \big(G+G^\top\big)
        \begin{pmatrix}
            x\\
            y
        \end{pmatrix}
        = y^\top \big( G_4 + G_4^\top \big) y \\
        & =
        \begin{pmatrix}
            0^\top & y^\top
        \end{pmatrix}
        \begin{pmatrix}
			QA+A^\top Q+\sum_{j=1}^k\fA^\top(j) Q \fA(j) & QB-C^\top +\sum_{j=1}^k\fA^\top(j) Q \fB(j)\\
			B^\top Q-C +\sum_{j=1}^k\fB^\top(j) Q \fA(j) & - (D+D^\top) +\sum_{j=1}^k\fB^\top(j) Q \fB(j)
		\end{pmatrix}
        \begin{pmatrix}
            0\\
            y
        \end{pmatrix}
        \le 0 .
    \end{split}
    \end{equation}
    It follows from \eqref{eq:1703a} and \eqref{eq:1703b} that
    $G+G^\top$ is negative semidefinite.
    Let
    \begin{alignat*}{2}
        &J=\frac12 \big( G_1 - G_1^\top \big) \in \R^{d\times d}, \quad 
        &&R=-\frac{1}{2} \big( G_1 + G_1^\top \big) \in \R^{d\times d},\\
        &P=-\frac{1}{2} \big( G_2 + G_3^\top \big) \in \R^{d\times n},  \quad 
        &&F=\frac{1}{2}\big( G_2 - G_3^\top \big) \in \R^{d\times n}, \\
        &S=-\frac12\big( G_4 + G_4^\top \big) \in \R^{n\times n}, \quad
        &&N=-\frac12\big( G_4 - G_4^\top \big)\in \R^{n\times n}.
    \end{alignat*}
    Observe that $-J^\top=J$, $R^\top=R$, $S^\top=S$, $-N^\top = N$,
    and
    \begin{equation}
        \begin{pmatrix}
            R & P \\
            P^\top & S
        \end{pmatrix}
        =
        -\frac12 \big( G + G^\top \big)
        \ge 0 .
    \end{equation}
    Moreover, note that
    \begin{equation}
        \begin{split}
            A & = \overline{A} Q
            = G_1 Q - \frac12 \sum_{j=1}^k \overline{\fA}^\top(j)Q\overline{\fA}(j) Q
            = (J-R)Q - \frac12 \sum_{j=1}^k \overline{\fA}^\top(j)Q\overline{\fA}(j) Q , \\
            B & = G_2 = F-P,\\
            C & = \overline{C} Q
            = \bigg(- G_3^\top + \sum_{j=1}^k \overline{\fA}^\top(j) Q \fB^\top(j) \bigg)^\top Q
            = \bigg( F+P + \sum_{j=1}^k \overline{\fA}^\top(j) Q \fB^\top(j) \bigg)^\top Q, \\
            D & = - G_4 + \frac12 \sum_{j=1}^k \fB^\top(j) Q \fB(j)
            = S + N + \frac12 \sum_{j=1}^k \fB^\top(j) Q \fB(j) .
        \end{split}
    \end{equation}
    This shows that \eqref{system} is an SLTI-PHS.
\end{proof}

\begin{remark}\label{rem-SLTIPHS-kernel-condition}
    \begin{enumerate}
        \item[(1)] Suppose that Assumption~\ref{ass-linear-system} is fulfilled.
        If $Q \in \R^{d\times d}$ is positive definite and it holds 
        that 
        $\mathfrak{M}_Q\le 0$, then
        all assumptions of Lemma~\ref{lemma:lmiimpliesphs} are satisfied.

        \item[(2)] There exist examples where  
        the linear matrix inequality $\mathfrak{M}_Q\le 0$ admits a positive semidefinite solution $Q \in \bbr^{d\times d}$, but the SLTIS~\eqref{system} is not an SLTI-PHS (see, e.g., the deterministic example \cite[Example~2]{cherifi2024difference}).

        \item[(3)]
        The converse of Lemma~\ref{lemma:lmiimpliesphs} is true as well: If the system~\eqref{system} is an SLTI-PHS with $Q$, then Remark~\ref{rem-SLTIPHS-params-SLTIS} shows that the system~\eqref{system} is an SLTIS whose system matrices satisfy  
        $\ker(Q)\subseteq \ker(A) \cap \ker(C) \cap (\cap_{j=1}^k \ker(\fA(j)))$ 
        and Lemma~\ref{lemma-SLTIPHS-LMI} shows that $\mathfrak{M}_Q\le 0$.
    \end{enumerate}
\end{remark}

In terms of passivity properties, we obtain the following 
sufficient conditions for an SLTI-PHS.

\begin{corollary}\label{cor-passivity-to-sltiphs}
    Suppose that Assumptions~\ref{ass-linear-system}, \ref{ass-control-cont-in-zero} and condition~\eqref{int-cond-u-linear} are fulfilled.
    Assume that there exists a positive semidefinite matrix
    $Q \in \R^{d\times d}$ such that $H(x)=\frac{1}{2}\langle Qx,x\rangle$, $x\in \R^d$,
    and 
    \begin{equation}\label{eq:SLTIS-kernel-condition}
        \ker(Q)\subseteq \ker(A) \cap \ker(C) \cap (\cap_{j=1}^k \ker(\fA(j))).
    \end{equation}
    Moreover, assume that at least one of the following conditions (i)--(iii) is satisfied:
        \begin{enumerate}
        \item[(i)] Assumption~\ref{ass-control-cont-in-zero-and-integrable} is fulfilled 
        and the SLTIS~\eqref{system} is stochastically passive with respect to the storage function~$H$.

        \item[(ii)] The SLTIS~\eqref{system} has the passive locale supermartingale property with respect to the storage function~$H$.

        \item[(iii)] The SLTIS~\eqref{system} has the passive supermartingale property with respect to the storage function~$H$.
    \end{enumerate}
    Then the SLTIS~\eqref{system} is an SLTI-PHS with $Q$.
\end{corollary}

\begin{proof}
    The claim follows from Lemma~\ref{lemma:lmiimpliesphs}, Theorem~\ref{thm-linear-1}, and Theorem~\ref{thm-linear-3}.
\end{proof}

If the SLTIS is observable (recall Definition~\ref{def-observable}), then 
Corollary~\ref{cor-linear-quadraticstorage-3} together with  Theorem~\ref{thm-linear-3} and Lemma~\ref{lemma:lmiimpliesphs} 
allows us to 
start from stochastic passivity with respect to any storage function that is bounded from below.
This, combined with Corollary~\ref{cor-SLTIPHS-passivityprop}, yields the following characterization of SLTI-PHS.

\begin{corollary}\label{cor:pass_and_obs_imply_phs}
    Suppose that Assumptions \ref{ass-linear-system} and \ref{ass-linear-system-controls-locsquareint} are fulfilled and that $(\calf_t)_{t\ge 0}$ is the natural filtration of $W$ augmented by all the $\IP$-null sets in $\calf$. 
    Assume that the SLTIS~\eqref{system} is observable. 
    Then the following statements are equivalent:
    \begin{enumerate}
        \item[(i)]
        The SLTIS~\eqref{system} is stochastically passive with respect to a storage function~$H$ that is bounded from below.

        \item[(ii)]
        There exists a positive definite matrix $Q\in\bbr^{d\times d}$ such that the SLTIS~\eqref{system} is an SLTI-PHS with~$Q$.
        
        \item[(iii)]
        The SLTIS~\eqref{system} is an SLTI-PHS. 
    \end{enumerate}
\end{corollary}

\begin{proof}
    (i) $\Rightarrow$ (ii): 
    Note that 
    Corollary~\ref{cor-linear-quadraticstorage-3} shows that there exists a positive definite matrix $Q\in\bbr^{d\times d}$ such that the SLTIS~\eqref{system} is stochastically passive with respect to $\cH\colon \bbr^d\to [0,\infty)$, $\cH(x)=\frac12 \langle Q x, x \rangle$.
    Then Theorem~\ref{thm-linear-3} implies that $\mathfrak{M}_Q\le 0$.
    Since $Q$ is positive definite, (ii) now follows from Lemma~\ref{lemma:lmiimpliesphs}.

    (ii) $\Rightarrow$ (iii):
    This is clearly true.

    (iii) $\Rightarrow$ (i): 
    Corollary~\ref{cor-SLTIPHS-passivityprop} ensures that the SLTI-PHS with some $Q\ge 0$ is stochastically passive with respect to the storage function $H\colon\bbr^d\to [0,\infty)$, $H(x)=\frac12 \langle Qx,x\rangle$.
\end{proof}

Observe that, in general, it is not guaranteed that an SLTI-PHS admits a representation where the matrix $Q$ in Definition~\ref{def-SLTIPHS} is positive definite. 
At least, for any SLTI-PHS there exists a (possibly non-invertible) transformation into an SLTI-PHS with $Q=\Id$.

\begin{remark}\label{rem-SLTIPHS-transformationQ}
    If the system \eqref{system} is an SLTI-PHS with $Q\ge 0$, then it can be transformed into an SLTI-PHS with $\widetilde Q=\Id$.
    Indeed, set $\widetilde{X}=Q^{1/2}X$, where $Q^{1/2}$ is the square root of $Q$ (i.e., $Q^{1/2}$ is positive semidefinite with $Q^{1/2}Q^{1/2}=Q$). Then we obtain for all $t\in [0,\infty)$ that
    \begin{equation}
    \begin{split}
        d\widetilde{X}_t&=Q^{1/2}\bigg[\bigg(J-R-\frac{1}{2}\sum_{j=1}^k\overline{\fA}^\top(j) Q \overline{\fA}(j)\bigg)Q X_t+(F-P)u_t\bigg]dt
        +\sum_{j=1}^kQ^{1/2}(\overline{\fA}(j)Q X_t+ \fB(j) u_t)dW_t^j\\
        &=\bigg[\bigg(\widetilde{J}-\widetilde{R}-\frac{1}{2}\sum_{j=1}^k\widetilde{\fA}^\top(j) \widetilde{\fA}(j) \bigg) \widetilde{X}_t + (\widetilde{F}-\widetilde{P}) u_t \bigg] dt 
        + \sum_{j=1}^k (\widetilde{\fA}(j) \widetilde{X}_t+ \widetilde{\fB}(j) u_t)dW_t^j
        \end{split}
    \end{equation}
    and
    \begin{equation}
        Y_t=\bigg(\widetilde{F}+\widetilde{P}+\sum_{j=1}^k\widetilde{\fA}^\top(j) \widetilde{\fB}(j) \bigg)^\top \widetilde{X}_t + \bigg( S+N+\frac{1}{2}\sum_{j=1}^k\widetilde{\fB}^\top(j)  \widetilde{\fB}(j) \bigg) u_t,
    \end{equation}
    with
    \begin{equation*}
        \begin{split}
            & \widetilde{J}=Q^{1/2}JQ^{1/2},
            \quad
            \widetilde{R}=Q^{1/2}RQ^{1/2},
            \quad
            \widetilde{F}=Q^{1/2}F,
            \quad
            \widetilde{P}=Q^{1/2}P,
            \\
            &
            \widetilde{\fA}(j)=Q^{1/2}\overline{\fA}(j) Q^{1/2},
            \quad
            \widetilde{\fB}(j)=Q^{1/2}\fB(j) .
        \end{split}
    \end{equation*}
    Moreover, one directly verifies that $\widetilde{R}$, $\widetilde{P}$ and $S$ satisfy \eqref{eq:matrix_ineq_pH}.
\end{remark}

We next discuss mean-square stabilization properties for SLTI-PHS. 

\begin{remark}\label{rem-SLTIPHS-stability}
        By applying established results from the literature one can derive stability and stabilization properties for SLTI-PHS in the mean-square sense. 
        For instance, if there is no control in the diffusion part (i.e., $\fB(j)=0$ for all $j\in\{1,\ldots,k\}$), $Q$ is positive definite, and there exists a constant $c \in \bbr$ such that $R+c(F-P)(F-P)^\top$ is positive definite,  
        then it follows from, e.g., \cite[Theorem 1]{rami2000linear} that the SLTI-PHS is asymptotically mean-square stabilizable. 
        In particular, if both $Q$ and $R$ in Definition~\ref{def-SLTIPHS} are positive definite, then the uncontrolled stochastic differential equation \eqref{SDE} is asymptotically mean-square stable.
\end{remark}

In addition to passivity, a crucial feature of deterministic PHS is the fact that any power-conserving interconnection of deterministic PHS yields again a deterministic PHS. 
We show that this property also holds in the stochastic context for SLTI-PHS.

\begin{remark}\label{rem-SLTI-interconnection}
    We first explain how we interconnect two SLTIS. Let 
    \begin{equation*}
        A(i),B(i),(\fA(j,i))_{j=1,\ldots,k},(\fB(j,i))_{j=1,\ldots,k},C(i),D(i)
    \end{equation*}
    denote the system matrices of dimensions $d(i)$ (state) and $n(i)$ (control) of the SLTIS (i), where $i\in\{1,2\}$ (cf.\ Assumption~\ref{ass-linear-system} and~\eqref{eq:SLTIS}). 
    Assume for all $i\in\{1,2\}$ that the control set $\cala(i)$ of the SLTIS~(i) 
    satisfies Assumption~\ref{ass-linear-system-controls-locsquareint} (with $n$ replaced by $n(i)$). 
    To interconnect these two systems, let 
    $\hat{n} \in \bbn$ with $\hat{n}\le\min\{n(1),n(2)\}$ 
    and $K \in \bbr^{2\hat{n}\times 2\hat{n}}$. 
    We take (for simplicity) the first $\hat{n}$ components of the controls for connecting and leave the remaining components as external inputs for the connected system, i.e., for each $i\in\{1,2\}$ we split $u(i)\in\cala(i)$ into the $\bbr^{\hat{n}}$-valued process $u^c(i)$ and the $\bbr^{n(i)-\hat{n}}$-valued process $u^e(i)$ such that $u(i)=((u^c(i))^\top,(u^e(i))^\top)^\top$. 
    We also partition the system matrices accordingly, e.g., $B(i)=(B^c(i),B^e(i))$ with $B^c(i) \in \bbr^{d(i)\times \hat{n}}$ and $B^e(i) \in \bbr^{d(i)\times (n(i)-\hat{n})}$, and 
    \begin{equation*}
        C(i) = \begin{pmatrix}
            C^c(i) \\
            C^e(i)
        \end{pmatrix},
        \quad 
        D(i) = \begin{pmatrix}
            D^c(i) & D^{ce}(i) \\
            D^{ec}(i) & D^e(i)
        \end{pmatrix}
        .
    \end{equation*}
    Note that for each $i\in\{1,2\}$, initial value $x_0(i) \in \bbr^{d(i)}$ and control $u(i) \in\cala(i)$ this also splits the output into $Y(i)=((Y^c(i))^\top,(Y^e(i))^\top)^\top$ 
    with 
    \begin{equation*}
        \begin{split}
            Y^c(i) & = C^c(i) X(i) + D^c(i) u^c(i) + D^{ce}(i) u^e(i),\\
            Y^e(i) & = C^e(i) X(i) + D^{ec}(i) u^c(i) + D^{e}(i) u^e(i) .
        \end{split}
    \end{equation*}
    If the matrix 
    \begin{equation}\label{eq:matrix-to-be-inverted}
        \mathfrak{D} :=
        \Id - K \begin{pmatrix}
            D^c(1) & 0 \\
            0 & D^c(2)
        \end{pmatrix}
    \end{equation}
    is invertible, then 
    the interconnection 
    \begin{equation}\label{eq:interconnection-equation}
        \begin{pmatrix}
            u^c(1) \\
            u^c(2)
        \end{pmatrix}
        = K \begin{pmatrix}
            Y^c(1)\\
            Y^c(2)
        \end{pmatrix}
    \end{equation}
    yields an SLTIS 
    with the state $X=((X(1))^\top,(X(2))^\top)^\top$, the input $u=u^e=((u^e(1))^\top,(u^e(2))^\top)^\top$ and the output $Y=Y^e=((Y^e(1))^\top,(Y^e(2))^\top)^\top$.  
    This can be shown by 
    noticing that \eqref{eq:interconnection-equation} is equivalent to 
    \begin{equation*}
        \begin{split}
            \mathfrak{D} \begin{pmatrix}
            u^c(1) \\
            u^c(2)
            \end{pmatrix}
            & = K 
            \begin{pmatrix}
                C^c(1) & 0 \\
                0 & C^c(2)
            \end{pmatrix}
            \begin{pmatrix}
                X(1) \\
                X(2)
            \end{pmatrix}
            + K \begin{pmatrix}
                D^{ce}(1) & 0 \\
                0 & D^{ce}(2)
            \end{pmatrix}
            \begin{pmatrix}
                u^e(1) \\
                u^e(2)
            \end{pmatrix}
            .
        \end{split}
    \end{equation*} 
    Note that the connected SLTIS 
    reads 
    \begin{align}\label{eq:SLTIS-interconnected}
    \left\{
    \begin{array}{rcl}
    dX_t & = & \big[ \big( A+B^c\mathfrak{D}^{-1} KC^c \big)  X_t  
    + \big( B^c\mathfrak{D}^{-1} KD^{ce} + B^e \big) u_t \big] dt \\
    & & + \sum_{j=1}^k \big[ \big( \mathfrak{A}(j)+\fB^c(j) \mathfrak{D}^{-1} KC^c \big) X_t + \big( \fB^c(j) \mathfrak{D}^{-1} KD^{ce} + \fB^e(j) \big) u_t \big] dW_t^j
    \\ Y_t & = & \big( C^e+D^{ec}\mathfrak{D}^{-1} KC^c \big) X_t
    + \big( D^{ec}\mathfrak{D}^{-1} KD^{ce}+D^e \big) u_t
    \\ X_0 & = & x_0 \in \bbr^{d(1)+d(2)}
    \\ u & \in & \cala,
    \end{array}
    \right.
    \end{align}
    where $\cala$ satisfies Assumption~\ref{ass-linear-system-controls-locsquareint} (with $n$ replaced by $n(1)+n(2)-2\hat{n}$) and 
    where for any matrices $G(1),G(2)$, we denote 
    \begin{equation*}
        G = 
        \begin{pmatrix}
            G(1) & 0 \\
            0 & G(2)
        \end{pmatrix}
        .
    \end{equation*}
    Moreover, note that if we use all control components of the initial systems for connecting, then the connected system will have no input and no output (but we still refer to it as an SLTIS).
\end{remark}

\begin{proposition}\label{lemma_interconnection}
    For each $i\in\{1,2\}$ consider an SLTI-PHS (i) of dimensions $d(i)$ (state) and $n(i)$ (control) 
    with $N(i)=0$ in Definition~\ref{def-SLTIPHS} 
    and 
    assume that the control set $\cala(i)$ 
    satisfies Assumption~\ref{ass-linear-system-controls-locsquareint}.  
    Let $\hat{n} \in \{1,\ldots,\min\{n(1),n(2)\}\}$
    and let $K \in \bbr^{2\hat{n}\times 2\hat{n}}$ satisfy $K=-K^\top$.  
    Then the system~\eqref{eq:SLTIS-interconnected} (i.e., the interconnection of the SLTI-PHS (1) and the SLTI-PHS (2) via~\eqref{eq:interconnection-equation}) is an SLTI-PHS.
\end{proposition}

\begin{proof}
    We use the notation of Remark~\ref{rem-SLTI-interconnection}. 
    In particular, note that $\mathfrak{D}=\Id - K D^c$.
    
    Suppose that $v \in\bbr^{2\hat{n}}$ satisfies $\mathfrak{D}v=0$. 
    It then holds that $v=KD^cv$. 
    Since $K$ is skew symmetric, it follows that 
    $$v^\top (D^c)^\top v 
    = v^\top (D^c)^\top K D^c v 
    = \tfrac12 v^\top (D^c)^\top K D^c v + \tfrac12 v^\top (D^c)^\top K^\top D^c v 
    =0.$$
    For each $i\in\{1,2\}$, the fact that we have an SLTI-PHS 
    with $N(i)=0$ shows that $D(i)$ is positive semidefinite. Hence, $D^c$ is positive semidefinite. We thus obtain that $D^c v = 0$. This and $v=KD^cv$ 
    demonstrate that $v=0$. Therefore, $\mathfrak{D}$ admits an inverse.  

    Note that for all $u(i) \in\cala(i)$ 
    and $x_0(i)\in\bbr^{d(i)}$, $i\in\{1,2\}$, it follows from~\eqref{eq:interconnection-equation} and 
    $K=-K^\top$ that 
    \begin{equation*}
        \langle u^c_s, Y^c_s \rangle = \langle K Y^c_s, Y^c_s \rangle = 0, \quad s\in[0,\infty). 
    \end{equation*}
    This implies for all $u(i) \in\cala(i)$ 
    and $x_0(i)\in\bbr^{d(i)}$, $i\in\{1,2\}$, that (recall $Y=Y^e$) 
    \begin{equation}\label{eq:adding-components-to-the-supply}
        \langle u^e_s , Y_s \rangle 
        = \langle u^e_s , Y^e_s \rangle +  \langle u^c_s, Y^c_s \rangle 
        = \langle u_s(1), Y_s(1) \rangle 
        + \langle u_s(2), Y_s(2) \rangle, \quad s\in[0,\infty). 
    \end{equation}
    From Corollary~\ref{cor-SLTIPHS-passivityprop} we have for all $i\in\{1,2\}$ that the SLTI-PHS (i) is stochastically passive with respect to the storage function $H_i(x)=\frac12 \langle Q(i) x, x\rangle$, $x\in\bbr^{d(i)}$, where $Q(i) \in \bbr^{d(i)\times d(i)}$ comes from Definition~\ref{def-SLTIPHS}. 
    Let $Q=\diag(Q(1),Q(2))$ and $H(x)=\frac12 \langle Q x, x\rangle$, $x\in\bbr^{d(1)+d(2)}$.
    Due to \eqref{eq:adding-components-to-the-supply}, it holds for all $x_0\in\bbr^{d(1)+d(2)}$, $u^e \in\cala$, and $t\in[0,\infty)$ that
    \begin{equation*}
        \begin{split}
            H(X_{t}) - \int_0^t \langle u^e_s, Y_s \rangle ds
            & = H_1(X_{t}(1)) 
            - \int_0^t \langle u_s(1), Y_s(1) \rangle ds 
            + H_2(X_{t}(2)) - \int_0^t \langle u_s(2), Y_s(2) \rangle ds .
        \end{split}
    \end{equation*}
    Therefore, stochastic passivity of the SLTI-PHS (i) with respect to $H_i$ for all $i\in\{1,2\}$ implies that the system~\eqref{eq:SLTIS-interconnected} is stochastically passive with respect to $H$.

    Suppose that $x=((x(1))^\top,(x(2))^\top)^\top \in \bbr^{d(1)+d(2)}$ such that $Qx=0$. It then holds for all $i\in\{1,2\}$ that $x(i)\in\ker(Q(i))$. This and Remark~\ref{rem-SLTIPHS-params-SLTIS} imply for all $i\in\{1,2\}$ that $x(i)\in\ker(A(i))\cap \ker(C(i)) \cap (\cap_{j=1}^k \ker(\fA(j,i)))$. 
    It follows that 
    $x\in\ker(A)\cap\ker(C)\cap(\cap_{j=1}^k \ker(\fA(j)))$. 
    Moreover, we can show that $x \in \ker(C^c)$ and $x \in \ker(C^e)$.  
    We hence obtain that the SLTIS~\eqref{eq:SLTIS-interconnected} satisfies the condition~\eqref{eq:SLTIS-kernel-condition}. 
    The claim then follows from the stochastic passivity of the SLTIS~\eqref{eq:SLTIS-interconnected} and Corollary~\ref{cor-passivity-to-sltiphs}.
\end{proof}

To conclude this section, we provide generalizations of two classic examples of linear PHS to SLTI-PHS 
and exploit the interconnection property to obtain a third example. 

\begin{example}
    \begin{enumerate}
        \item[(1)] We consider a stochastic variant of the mass-spring-damper system. To this end, we denote by $m>0$ the mass of the object, by $c>0$ the damping constant, and by $\kappa>0$ the spring constant. Let 
        \begin{equation*}
            Q=\begin{pmatrix}
                \kappa & 0\\
                0& \frac{1}{m}
            \end{pmatrix}
        \end{equation*}
        and assume that $\sigma \in [-\sqrt{2mc}, + \sqrt{2mc}]$. 
        Then \begin{equation}\label{eq:massspringdamperwithnoise}
        \begin{split}
            dX_t
        &=\left(\begin{pmatrix}
                0&1\\
                -1 & -c
            \end{pmatrix}
             Q X_t
        +\begin{pmatrix}
            0\\1
        \end{pmatrix}
        u_t\right)dt
        + \begin{pmatrix}
            0 & 0 \\
            0 & \sigma 
        \end{pmatrix}
        QX_tdW_t,\\
        Y_t&=\begin{pmatrix}
            0 & 1
        \end{pmatrix} QX_t, 
        \end{split}
        \end{equation}
        is an SLTI-PHS. 
        The first component of the state in \eqref{eq:massspringdamperwithnoise} is the displacement of the mass and the second component is the momentum, which here is subject to multiplicative noise. 
        The control $u$ is interpreted as an external force applied to the system while the output gives the velocity of the mass. 
        The Hamiltonian $H$ is the sum of the potential and kinetic energy, namely 
        \begin{equation*}
            H(x) = \frac12 \langle Qx,x\rangle 
            = \frac{1}{2} \kappa q^2 + \frac{1}{2} m^{-1} p^2, \quad x = (q,p)^\top \in \bbr^2 . 
        \end{equation*}
        Note that if $\sigma=0$ and the controls are deterministic, then \eqref{eq:massspringdamperwithnoise}
        is the classic mass-spring-damper system in port-Hamiltonian formulation.
        Stochastic mass-spring-damper systems are also called stochastic oscillator in the literature, see, e.g., \cite{ryashko1997meansquare}. We remark that our condition on the size of the noise, $\sigma \in [-\sqrt{2mc}, + \sqrt{2mc}]$, which we require to obtain an SLTI-PHS, appears in \cite[Equation~(3.7)]{ryashko1997meansquare} as the condition for mean-square stability.

        \item[(2)] 
        We next consider a stochastic variant of a series RLC circuit.
        Let $r>0$ denote the resistance, $l>0$ the inductance, and $c>0$ the capacitance. 
        Define  
        \begin{equation*}
            Q=\begin{pmatrix}
                \frac{1}{c}& 0\\
                0& \frac{1}{l}
            \end{pmatrix}
            .
        \end{equation*}
        Let $\sigma_1,\sigma_2 \in \bbr$ satisfy $2r \ge \frac{1}{c} \sigma_1^2 + \frac{1}{l} \sigma_2^2$. 
        Then \begin{equation}\label{eq:rlcwithnoise}
        \begin{split}
            dX_t
        &=\left(\begin{pmatrix}
                0&1\\
                -1 & -r
            \end{pmatrix}
             Q X_t
        +\begin{pmatrix}
            0\\1
        \end{pmatrix}
        u_t\right)dt
        + \begin{pmatrix}
            0 & \sigma_1 \\
            0 & \sigma_2 
        \end{pmatrix}
        QX_tdW_t,\\
        Y_t&=\begin{pmatrix}
            0 & 1
        \end{pmatrix} QX_t, 
        \end{split}
        \end{equation}
        is an SLTI-PHS. 
        The first component of the state in \eqref{eq:rlcwithnoise} is the charge on the capacitor and the second component is the magnetic flux through the inductor. 
        The current is given by the output, and as an input one applies a voltage source. 
        The Hamiltonian 
            \begin{equation*}
                H(x) = \frac12 \langle Qx,x\rangle 
                = \frac{1}{2} c^{-1} q^2 + \frac{1}{2} l^{-1} \phi^2, \quad x = (q,\phi)^\top \in \bbr^2 , 
            \end{equation*}
        is the sum of the electric and magnetic energy. 
        Observe that by setting $\sigma_1=0=\sigma_2$ and allowing only deterministic controls, we recover 
        the deterministic port-Hamiltonian formulation of an RLC circuit.
        For further stochastic RLC models, see, for example, \cite[Section~1.9.7]{damm2004book}, \cite[Section~1.1]{ugurinovskii1999absolutestabil}, or \cite[Section~4]{EhKr_cordoni2023weak}. 
    
        \item[(3)]
        Proposition~\ref{lemma_interconnection} demonstrates that we can construct further SLTI-PHS examples by coupling. In particular, we can couple systems which come from different physical domains, such as the mass-spring-damper system (1) and the electrical circuit (2): 
        Since in both (1) and (2) there is no control in the output, we have that Proposition~\ref{lemma_interconnection} shows that interconnecting (1) and (2) via 
        \begin{equation*}
            K=\begin{pmatrix}
                0 & 1\\
                -1 & 0
            \end{pmatrix}
        \end{equation*}
        yields an SLTI-PHS with four-dimensional state (and no external inputs and outputs).  
        Note that in the interconnection the output current from (2) acts as an input force for (1), and the velocity of the mass (in opposite direction) plays the role of a voltage source for (2).
    \end{enumerate}
\end{example}

\begin{appendix}

\section{Semimartingale theory}\label{app-semimartingales}

In this appendix we provide the required results about semimartingale theory. Let $(\Omega,\calf,(\calf_t)_{t \in \bbr_+},\bbp)$ be a stochastic basis satisfying the usual conditions, see \cite[Def. I.1.3]{Jacod-Shiryaev}. Let us recall some notation:
\begin{itemize}
\item We denote by $\calm$ the space of all uniformly integrable martingales.

\item We denote by $\calm_{\loc}$ the space of all local martingales.

\item We denote by $\calm_{\loc}^0$ the space of all local martingales $M$ with $M_0 = 0$.

\item We denote by $\calv^+$ the convex cone of all c\`{a}dl\`{a}g, adapted and increasing processes $A$ with $A_0 = 0$.

\item We denote by $\calv^-$ the convex cone of all c\`{a}dl\`{a}g, adapted and decreasing processes $A$ with $A_0 = 0$.

\item We denote by $\calv$ the space of all c\`{a}dl\`{a}g, adapted processes $A$ with locally finite variation such that $A_0 = 0$.

\item We denote by $\scra^+$ the convex cone of all $A \in \calv^+$ such that $A_{\infty} \in L^1$.

\item We denote by $\scra^-$ the convex cone of all $A \in \calv^-$ such that $A_{\infty} \in L^1$.

\item We denote by $\scra$ the space of all $A \in \calv^-$ such that ${\rm Var}(A)_{\infty} \in L^1$.
\end{itemize}

\begin{definition}
Every process $X$ of the form $X = X_0 + M + A$ with $M \in \calm_{\loc}^0$ and $A \in \calv$ is called a \emph{semimartingale}.
\end{definition}

\begin{definition}
Every process $X$ of the form $X = X_0 + M + A$ with $M \in \calm_{\loc}^0$ and a predictable process $A \in \calv$ is called a \emph{special semimartingale}.
\end{definition}

We denote by $\cals$ the space of all semimartingales, and by $\cals_p$ the space of all special semimartingales.

\begin{proposition}\cite[Cor. I.3.16]{Jacod-Shiryaev}\label{prop-can-decomp}
For every special semimartingale $X$ the semimartingale decomposition $X = X_0 + M + A$ with $M \in \calm_{\loc}^0$ and a predictable process $A \in \calv$ is unique. It is called the \emph{canonical decomposition} of $X$.
\end{proposition}

Let $\cals^-$ be the convex cone of all supermartingales.

\begin{definition}
A process $X$ is \emph{of class (D)} if the set of random variables $\{ X_T : T \text{ is a finite stopping time} \}$ is uniformly integrable.
\end{definition}

We can now restate the Doob-Meyer decomposition theorem by means of special semimartingales as follows.

\begin{proposition}\label{prop-Doob-Meyer}
For an adapted c\`{a}dl\`{a}g process $X$ the following statements are equivalent:
\begin{enumerate}
\item[(i)] $X$ is a supermartingale of class (D).

\item[(ii)] $X$ is a special semimartingale with canonical decomposition
\begin{align*}
X = X_0 + M + A
\end{align*}
such that $X_0 \in L^1$, $M \in \calm$ and $A \in \scra^-$.
\end{enumerate}
\end{proposition}

\begin{proof}
(i) $\Rightarrow$ (ii): This is a consequence of \cite[Thm. I.3.15]{Jacod-Shiryaev}.

\noindent(ii) $\Rightarrow$ (i): By \cite[Prop. I.1.47]{Jacod-Shiryaev} the martingale $M$ is of class (D). Furthermore, the process $A$ is of class (D) because $|A_T| \leq |A_{\infty}|$ for every finite stopping time $T$. Therefore, the semimartingale $X$ is of class (D). Moreover, for all $s,t \in \bbr_+$ with $s \leq t$ and all $B \in \calf_s$ we obtain
\begin{align*}
\bbe[X_t \bbI_B] &= \bbe[X_0 \bbI_B] + \bbe[M_t \bbI_B] + \bbe[A_t \bbI_B]
\leq \bbe[X_0 \bbI_B] + \bbe[M_s \bbI_B] + \bbe[A_s \bbI_B] = \bbe[X_s \bbI_B],
\end{align*}
proving that $X$ is a supermartingale.
\end{proof}

The next goal is to provide a local version of Proposition \ref{prop-Doob-Meyer}. For this purpose, we prepare two auxiliary results.

\begin{lemma}\label{lemma-contained}
We have $\cals_{\loc}^- \subset \cals_p$.
\end{lemma}

\begin{proof}
This is the statement on page 32 in \cite{Jacod} following Prop. 2.18.
\end{proof}

\begin{lemma}\label{lemma-sup-loc-D}
Every local supermartingale $X \in \cals_{\loc}^-$ is locally a supermartingale of class (D).
\end{lemma}

\begin{proof}
By localization we may assume that $X$ is a supermartingale, and then the statement follows from proof of \cite[Prop. 2.18]{Jacod}.
\end{proof}

Now, we can present a local version of the Doob-Meyer decomposition theorem in terms of special semimartingales as follows.

\begin{proposition}\label{prop-super-MT}
For an adapted c\`{a}dl\`{a}g process $X$ the following statements are equivalent:
\begin{enumerate}
\item[(i)] $X$ is a local supermartingale.

\item[(ii)] $X$ is a special semimartingale with canonical decomposition
\begin{align*}
X = X_0 + M + A
\end{align*}
such that $X_0 \in L^1$ and $A \in \calv^-$.
\end{enumerate}
\end{proposition}

\begin{proof}
(i) $\Rightarrow$ (ii): By Lemma \ref{lemma-contained} we have $X \in \cals_p$. Let us denote by $X = X_0 + M + A$ the canonical decomposition of $X$. Since $X \in \cals_{\loc}^-$, we have $X_0 \in L^1$. By Lemma \ref{lemma-sup-loc-D} there exists a localizing sequence $(T_n)_{n \in \bbn}$ such that $X^{T_n}$ is a supermartingale of class (D) for each $n \in \bbn$. By Proposition \ref{prop-Doob-Meyer}, for each $n \in \bbn$ the process $X^{T_n}$ is a special semimartingale with canonical decomposition $X^{T_n} = X_0 + M(n) + A(n)$ such that $M(n) \in \calm$ and $A(n) \in \scra^-$. Since $\calm_{\loc}$ and the class of all predictable processes from $\calv$ are stable under stopping, the canonical decomposition of $X^{T_n}$ can also be expressed as $X^{T_n} = X_0 + M^{T_n} + A^{T_n}$ for each $n \in \bbn$. By the uniqueness of the canonical decomposition (see Proposition \ref{prop-can-decomp}) we deduce that $M^{T_n} = M(n)$ and $A^{T_n} = A(n)$ for each $n \in \bbn$. Therefore, we conclude that $A \in \calv^-$.

\noindent (ii) $\Rightarrow$ (i):  By \cite[Lemma I.3.10]{Jacod-Shiryaev} we have $A \in \scra_{\loc}^-$. Hence, there exists a localizing sequence $(T_n)_{n \in \bbn}$ such that $M^{T_n} \in \calm$ and $A^{T_n} \in \scra^-$ for each $n \in \bbn$. For each $n \in \bbn$, the canonical decomposition of the special semimartingale $X^{T_n}$ is given by $X^{T_n} = X_0 + M^{T_n} + A^{T_n}$, and thus, by Proposition \ref{prop-Doob-Meyer} it follows that $X^{T_n} \in \cals^-$. Consequently, we have $X \in \cals_{\loc}^-$.
\end{proof}

\begin{lemma}\label{lemma-supermartingale-integrable}
Let $X = X_0 + M + A$ be a local supermartingale. Suppose that $M$ is a martingale and that $A \in \calv^-$ satisfies $A_t \in L^1$ for all $t \geq 0$. Then $X$ is a supermartingale.
\end{lemma}

\begin{proof}
By assumption we have $X_t \in L^1$ for all $t \geq 0$. Moreover, for all $s,t \in \bbr_+$ with $s \leq t$ we have
\begin{align*}
\bbe[X_t | \calf_s] &= X_0 + \bbe[M_t | \calf_s] + \bbe[A_t | \calf_s] \leq X_0 + M_s + \bbe[A_s | \calf_s]
= X_0 + M_s + A_s = X_s,
\end{align*}
showing that $X$ is a supermartingale.
\end{proof}

\begin{lemma}\label{lemma-supermartingale-below}
Let $X$ be a local supermartingale of the form $X = Y + A$ such that $Y$ is a special semimartingale which is bounded from below, and $A \in \calv$ satisfies $|A| \leq B$ for some $B \in \calv^+$ with $B_t \in L^1$ for all $t \in \bbr_+$. Then $X$ is a supermartingale.
\end{lemma}

\begin{proof}
Let $(T_n)_{n \in \bbn}$ be a localizing sequence such that $X^{T_n}$ is a supermartingale for each $n \in \bbn$. Furthermore, let $s,t \in \bbr_+$ with $s \leq t$ be arbitrary. By Fatou's lemma we have
\begin{align*}
\bbe[Y_t | \calf_s] = \bbe \bigg[ \lim_{n \to \infty} Y_{t \wedge T_n} | \calf_s \bigg] \leq \liminf_{n \to \infty} \bbe [ Y_{t \wedge T_n} | \calf_s ].
\end{align*}
Moreover, dominated convergence ensures that
\begin{align*}
\bbe[A_t | \calf_s] = \bbe \bigg[ \lim_{n \to \infty} A_{t \wedge T_n} | \calf_s \bigg] = \lim_{n \to \infty} \bbe [ A_{t \wedge T_n} | \calf_s ].
\end{align*}
Therefore, we obtain
\begin{align*}
\bbe[X_t | \calf_s] &= \bbe \bigg[ \lim_{n \to \infty} Y_t^{T_n} | \calf_s \bigg] + \bbe \bigg[ \lim_{n \to \infty} A_t^{T_n} | \calf_s \bigg]
\leq \liminf_{n \to \infty} \bbe [ Y_t^{T_n} | \calf_s ] + \lim_{n \to \infty} \bbe [ A_t^{T_n} | \calf_s ]
\\ &\leq \liminf_{n \to \infty} \bbe[X_t^{T_n} | \calf_s] \leq \liminf_{n \to \infty} X_s^{T_n} = X_s,
\end{align*}
showing that $X$ is a supermartingale.
\end{proof}

\begin{remark}
Note that Lemma \ref{lemma-supermartingale-below} generalizes the well-known result that every local supermartingale which is bounded from below is a supermartingale.
\end{remark}

Recall that for a process $X$ we denote by $X^*$ the running maximum
\begin{align*}
X_t^* := \sup_{s \in [0,t]} |X_s|, \quad t \in \bbr_+.
\end{align*}

\begin{lemma}\label{lemma-martingale-suff}
Let $M \in \calm_{\loc}$ be a local martingale such that $M_t^* \in L^1$ for all $t \in \bbr_+$. Then $M$ is a martingale.
\end{lemma}

\begin{proof}
Let $t \in \bbr_+$ be arbitrary. Then we have $M_T^t \leq M_t^*$ for every finite stopping time $T$, showing that the family
\begin{align*}
\{ M_T^t : \text{$T$ is a finite stopping time} \}
\end{align*}
is uniformly integrable. Thus the local martingale $M^t$ is a process of class (D), and hence, by \cite[Prop. I.1.47.c]{Jacod-Shiryaev} we deduce that $M^t \in \calm$. Since $t \in \bbr_+$ was arbitrary, it follows that $M$ is a martingale.
\end{proof}

\section{An existence result for SDEs}\label{app-SDEs}

In this appendix we provide an existence result for SDEs, which we require in the paper. Let $(\Omega,\calf,(\calf_t)_{t \in \bbr_+},\bbp)$ be a stochastic basis satisfying the usual conditions, and consider an $\bbr^d$-valued SDE of the form
\begin{align}\label{SDE-tilde}
\left\{
\begin{array}{rcl}
dX_t & = & \widetilde{b}(t,X_t)dt + \widetilde{\sigma}(t,X_t)dW_t
\\ X_0 & = & x_0
\end{array}
\right.
\end{align}
driven by an $\bbr^k$-valued Wiener process $W$ with coefficients
\begin{align*}
&\widetilde{b} : \Omega \times \bbr_+ \times \bbr^d \to \bbr^d,
\\ &\widetilde{\sigma} : \Omega \times \bbr_+ \times \bbr^d \to \bbr^{d \times k}.
\end{align*}

\begin{theorem}\label{thm-SDE-2}
We assume that the following conditions are fulfilled:
\begin{enumerate}
\item For each $x \in \bbr^d$ the processes
\begin{align*}
&\Omega \times \bbr_+ \to \bbr^d, \quad (\omega,t) \mapsto \widetilde{b}(\omega,t,x),
\\ &\Omega \times \bbr_+ \to \bbr^{d \times k}, \quad (\omega,t) \mapsto \widetilde{\sigma}(\omega,t,x)
\end{align*}
are progressively measurable.

\item For each $R \in \bbr_+$ there is an adapted, nonnegative process $\widetilde{L}^R \geq 0$ such that for all $\omega \in \Omega$ and $t \in \bbr_+$ we have
\begin{align}\label{cond-LR-SDE}
\int_0^t |\widetilde{L}_s^R(\omega)|^2 ds < \infty,
\end{align}
and for all $\omega \in \Omega$, all $t \in \bbr_+$ and all $x,y \in \bbr^d$ with $\| x \|, \| y \| \leq R$ we have
\begin{align}\label{cond-1-SDE}
\| \widetilde{b}(\omega,t,x)-\widetilde{b}(\omega,t,y) \| + \| \widetilde{\sigma}(\omega,t,x) - \widetilde{\sigma}(\omega,t,y) \| \leq \widetilde{L}_t^R(\omega) \| x-y \|.
\end{align}
\item There is an adapted, nonnegative process $\widetilde{L} \geq 0$ such that for all $\omega \in \Omega$ and $t \in \bbr_+$ we have
\begin{align}\label{cond-L-SDE}
\int_0^t |\widetilde{L}_s(\omega)|^2 ds < \infty,
\end{align}
and for all $\omega \in \Omega$, all $t \in \bbr_+$ and all $x \in \bbr^d$ we have
\begin{align}\label{cond-2-SDE}
\| \widetilde{b}(\omega,t,x) \| + \| \widetilde{\sigma}(\omega,t,x) \| \leq \widetilde{L}_t(\omega) (1 + \| x \|).
\end{align}
\end{enumerate}
Then existence and uniqueness of strong solutions for the SDE \eqref{SDE-tilde} holds true.
\end{theorem}

\begin{proof}
For each $\omega \in \Omega$ and $t \in \bbr_+$ the mappings
\begin{align*}
&\bbr^d \to \bbr^d, \quad x \mapsto \widetilde{b}(\omega,t,x),
\\ &\bbr^d \to \bbr^{d \times k}, \quad x \mapsto \widetilde{\sigma}(\omega,t,x)
\end{align*}
are continuous by virtue of \eqref{cond-1-SDE}. We define the adapted, nonnegative process $\widetilde{K} \geq 0$ as
\begin{align*}
\widetilde{K}_t(\omega) := \widetilde{L}_t(\omega) + |\widetilde{L}_t(\omega)|^2, \quad (\omega,t) \in \Omega \times \bbr_+.
\end{align*}
Then by \eqref{cond-L-SDE} we have
\begin{align*}
\int_0^t \widetilde{K}_s(\omega) ds < \infty, \quad (\omega,t) \in \Omega \times \bbr_+.
\end{align*}
Therefore, by \eqref{cond-2-SDE} for all $\omega \in \Omega$ and all $T,R \in \bbr_+$ we obtain
\begin{align*}
\int_0^T \sup_{\| x \| \leq R} \big\{ \| \widetilde{b}(\omega,t,x) \| + \| \widetilde{\sigma}(\omega,t,x) \|^2 \big\} dt &\leq 2 \int_0^T \sup_{\| x \| \leq R} \widetilde{K}_t(\omega) (1 + \| x \|)^2 dt
\\ &\leq 2 (1 + R)^2 \int_0^T \widetilde{K}_t(\omega) dt < \infty,
\end{align*}
showing that condition (3.1) from \cite{Liu-Roeckner} is fulfilled. For each $R \in \bbr_+$ we define the adapted, nonnegative process $\widetilde{K}^R \geq 0$ as
\begin{align*}
\widetilde{K}_t^R(\omega) := 2 \widetilde{L}_t^R(\omega) + |\widetilde{L}_t^R(\omega)|^2 + 4 \widetilde{L}_t(\omega) + 2 |\widetilde{L}_t(\omega)|^2, \quad (\omega,t) \in \Omega \times \bbr_+.
\end{align*}
Then by \eqref{cond-LR-SDE} and \eqref{cond-L-SDE} we have
\begin{align*}
\int_0^T \widetilde{K}_t^R(\omega) dt < \infty.
\end{align*}
Furthermore, by \eqref{cond-1-SDE} for all $\omega \in \Omega$, all $t,R \in \bbr_+$ and all $x,y \in \bbr^d$ with $\| x \|, \| y \| \leq R$ we have
\begin{align*}
&2 \la x-y, \widetilde{b}(\omega,t,x) - \widetilde{b}(\omega,t,y) \ra + \| \widetilde{\sigma}(\omega,t,x) - \widetilde{\sigma}(\omega,t,y) \|^2
\\ &\leq 2 \| \widetilde{b}(\omega,t,x) - \widetilde{b}(\omega,t,y) \| \cdot \| x-y \| + \| \widetilde{\sigma}(\omega,t,x) - \widetilde{\sigma}(\omega,t,y) \|^2
\\ &\leq \big( 2 \widetilde{L}_t^R(\omega) + |\widetilde{L}_t^R(\omega)|^2 \big) \| x-y \|^2 \leq \widetilde{K}_t^R(\omega) \| x-y \|^2.
\end{align*}
Moreover, by \eqref{cond-2-SDE} for all $\omega \in \Omega$, all $t \in \bbr_+$ and all $x \in \bbr^d$ we have
\begin{align*}
2 \la x,\widetilde{b}(\omega,t,x) \ra + \| \widetilde{\sigma}(\omega,t,x) \|^2 &\leq 2 \| \widetilde{b}(\omega,t,x) \| \cdot \| x \| + \| \widetilde{\sigma}(\omega,t,x) \|^2
\\ &\leq 2 \widetilde{L}_t(\omega) \| x \| (1 + \| x \|) + |\widetilde{L}_t(\omega)|^2 \, (1 + \| x \|)^2
\\ &\leq \big( 2 \widetilde{L}_t(\omega) + |\widetilde{L}_t(\omega)|^2 \big) (1 + \| x \|)^2
\\ &\leq \widetilde{K}_t^1(\omega) (1 + \| x \|^2).
\end{align*}
Consequently, applying \cite[Thm. 3.1.1]{Liu-Roeckner} concludes the proof.
\end{proof}

\end{appendix}

\section*{Acknowledgments}

Julia Ackermann and Thomas Kruse acknowledge funding by the Deutsche For\-schungsgemeinschaft (DFG, German Research Foundation) -- Project-ID 531152215 -- CRC 1701.

\bibliographystyle{plain}
\bibliography{literature}

@article{di2025port,
  title={Port-{H}amiltonian Neural Networks: From Theory to Simulation of Interconnected Stochastic Systems},
  author={Di Persio, Luca and Ehrhardt, Matthias and Outaleb, Youness and Rizzotto, Sofia},
  journal={arXiv:2509.06674},
  year={2025}
}

@article{ryashko1997meansquare,
    AUTHOR = {Ryashko, Lev B. and Schurz, Henri},
     TITLE = {Mean square stability analysis of some linear stochastic systems},
   JOURNAL = {Dynam. Systems Appl.},
  FJOURNAL = {Dynamic Systems and Applications},
    VOLUME = {6},
      YEAR = {1997},
    NUMBER = {2},
     PAGES = {165--189},
      ISSN = {1056-2176},
   MRCLASS = {60H10 (34F05 65U05 93E15)}
}

@article{ait2000well,
  title={Well-posedness and attainability of indefinite stochastic linear quadratic control in infinite time horizon},
  author={Ait Rami, Mustapha and Zhou, Xun Yu and Moore, JB},
  journal={Systems \& control letters},
  volume={41},
  number={2},
  pages={123--133},
  year={2000},
  publisher={Elsevier}
}

@article{ugurinovskii1999absolutestabil,
    AUTHOR = {Ugrinovskii, Valery A. and Petersen, Ian R.},
     TITLE = {Absolute stabilization and minimax optimal control of uncertain systems with stochastic uncertainty},
   JOURNAL = {SIAM J. Control Optim.},
  FJOURNAL = {SIAM Journal on Control and Optimization},
    VOLUME = {37},
      YEAR = {1999},
    NUMBER = {4},
     PAGES = {1089--1122},
      ISSN = {0363-0129,1095-7138},
   MRCLASS = {93E20 (93D15)},
  MRNUMBER = {1681087},
MRREVIEWER = {Edwin\ Engin\ Yaz},
       DOI = {10.1137/S0363012996309964},
       URL = {https://doi.org/10.1137/S0363012996309964},
}

@book{sun2020stochastic,
  title={Stochastic linear-quadratic optimal control theory: Open-loop and closed-loop solutions},
  author={Sun, Jingrui and Yong, Jiongmin},
  year={2020},
  publisher={Springer Nature}
}

@article{sun2018stochastic,
  title={Stochastic linear quadratic optimal control problems in infinite horizon},
  author={Sun, Jingrui and Yong, Jiongmin},
  journal={Applied Mathematics \& Optimization},
  volume={78},
  number={1},
  pages={145--183},
  year={2018},
  publisher={Springer}
}

@article{albert1969conditions,
  title={Conditions for positive and nonnegative definiteness in terms of pseudoinverses},
  author={Albert, Arthur},
  journal={SIAM Journal on Applied Mathematics},
  volume={17},
  number={2},
  pages={434--440},
  year={1969},
  publisher={SIAM}
}

@article{rami2000linear,
  title={Linear matrix inequalities, {R}iccati equations, and indefinite stochastic linear quadratic controls},
  author={Ait Rami, Mustapha and Zhou, Xun Yu},
  journal={IEEE Transactions on Automatic Control},
  volume={45},
  number={6},
  pages={1131--1143},
  year={2000},
  publisher={IEEE}
}

@INPROCEEDINGS{vanderschaftsurvey,
  author={van der Schaft, Arjan},
  title={Port-Hamiltonian systems: an introductory survey},
  booktitle={Proceedings of the International Congress
of Mathematicians},
  address={Madrid, Spain},
  year={2006},
  volume={},
  number={},
  pages={1339-1365},
  doi={10.4171/022-3/65}
}

@book{Liu-Roeckner,
  title={Stochastic Partial Differential Equations: An Introduction},
  author={Liu, Wei and R\"{o}ckner, Michael},
  year={2015},
  publisher={Springer}
}

@book{Jacod,
  title={Calcul stochastique et problemes de martingales},
  author={Jacod, Jean},
  volume={714},
  year={2006},
  publisher={Springer}
}

@book{Jacod-Shiryaev,
  title={Limit theorems for stochastic processes},
  author={Jacod, Jean and Shiryaev, Albert},
  volume={288},
  year={2013},
  publisher={Springer Science \& Business Media}
}

@book{karatzas1991book,
    AUTHOR = {Karatzas, Ioannis and Shreve, Steven E.},
     TITLE = {Brownian motion and stochastic calculus},
    SERIES = {Graduate Texts in Mathematics},
    VOLUME = {113},
   EDITION = {Second},
 PUBLISHER = {Springer-Verlag, New York},
      YEAR = {1991},
     PAGES = {xxiv+470},
      ISBN = {0-387-97655-8},
       DOI = {10.1007/978-1-4612-0949-2},
       URL = {https://doi.org/10.1007/978-1-4612-0949-2},
}

@book{sontag1998book,
    AUTHOR = {Sontag, Eduardo D.},
     TITLE = {Mathematical control theory},
    SERIES = {Texts in Applied Mathematics},
    VOLUME = {6},
   EDITION = {Second},
      NOTE = {Deterministic finite-dimensional systems},
 PUBLISHER = {Springer-Verlag, New York},
      YEAR = {1998},
     PAGES = {xvi+531},
      ISBN = {0-387-98489-5},
       DOI = {10.1007/978-1-4612-0577-7},
       URL = {https://doi.org/10.1007/978-1-4612-0577-7},
}

@book{damm2004book,
    AUTHOR = {Damm, Tobias},
     TITLE = {Rational Matrix Equations in Stochastic Control},
    SERIES = {Lecture Notes in Control and Information Sciences},
   EDITION = {First},
 PUBLISHER = {Springer Berlin, Heidelberg},
      YEAR = {2004},
     PAGES = {xv+200},
      ISBN = {978-3-540-40001-1},
       DOI = {10.1007/b10906},
       URL = {https://doi.org/10.1007/b10906},
}

@book{jacob2012book,
    AUTHOR = {Jacob, Birgit and Zwart, Hans J.},
     TITLE = {Linear port-{H}amiltonian systems on infinite-dimensional spaces},
    SERIES = {Operator Theory: Advances and Applications},
    VOLUME = {223},
      NOTE = {Linear Operators and Linear Systems},
 PUBLISHER = {Birkh\"auser/Springer Basel AG, Basel},
      YEAR = {2012},
     PAGES = {xii+217},
      ISBN = {978-3-0348-0398-4},
   MRCLASS = {93-02 (47D03 93-01 93C25)},
  MRNUMBER = {2952349},
MRREVIEWER = {Bernhard\ M.\ Maschke},
       DOI = {10.1007/978-3-0348-0399-1},
       URL = {https://doi.org/10.1007/978-3-0348-0399-1},
}

@article{lanchares2023thermo,
  title     = {Stochastic thermodynamics: dissipativity, accumulativity, energy storage and entropy production},
  author    = {Lanchares, Manuel and Haddad, Wassim M},
  journal   = {Philos. Trans. A Math. Phys. Eng. Sci.},
  publisher = {The Royal Society},
  volume    =  {381},
  number    =  {2256},
  pages     = {},
  year      =  {2023},
    doi = {10.1098/rsta.2022.0284},
url = {https://doi.org/10.1098/rsta.2022.0284}
}

@article{fang2023weak,
    AUTHOR = {Fang, Zhou and Gao, Chuanhou and Dochain, Denis},
     TITLE = {Stochastic weak passivity for weakly stabilizing stochastic systems with nonvanishing noise},
   JOURNAL = {Systems Control Lett.},
  FJOURNAL = {Systems \& Control Letters},
    VOLUME = {180},
      YEAR = {2023},
     PAGES = {Paper No. 105606, 10},
      ISSN = {0167-6911,1872-7956},
   MRCLASS = {93E15 (60H10)},
  MRNUMBER = {4633740},
MRREVIEWER = {Ju-Liang\ Yin},
       DOI = {10.1016/j.sysconle.2023.105606},
       URL = {https://doi.org/10.1016/j.sysconle.2023.105606},
}

@article{ruediger2024carfollowing,
doi = {10.1088/1751-8121/ad5d2f},
url = {https://dx.doi.org/10.1088/1751-8121/ad5d2f},
year = {2024},
month = {jul},
publisher = {IOP Publishing},
volume = {57},
number = {29},
pages = {295203},
author = {Rüdiger, Barbara and Tordeux, Antoine and Ugurcan, Baris E},
title = {Stability analysis of a stochastic port-Hamiltonian car-following model},
journal = {Journal of Physics A: Mathematical and Theoretical},
}

@article{beattie2025port,
  title={{P}ort-{H}amiltonian realizations of nonminimal linear time invariant systems},
  author={Beattie, Christopher and Mehrmann, Volker and Xu, Hongguo},
  journal={arXiv:2201.05355v3},
  year={2025}
}

@article{cherifi2024difference,
    AUTHOR = {Cherifi, Karim and Gernandt, Hannes and Hinsen, Dorothea},
     TITLE = {The difference between port-{H}amiltonian, passive and
              positive real descriptor systems},
   JOURNAL = {Math. Control Signals Systems},
  FJOURNAL = {Mathematics of Control, Signals, and Systems},
    VOLUME = {36},
      YEAR = {2024},
    NUMBER = {2},
     PAGES = {451--482},
      ISSN = {0932-4194,1435-568X}
}

@article{willems1972dissipative,
  title={Dissipative dynamical systems part II: {L}inear systems with quadratic supply rates},
  author={Willems, Jan C},
  journal={Archive for rational mechanics and analysis},
  volume={45},
  pages={352--393},
  year={1972},
  publisher={Springer}
}

@article{sun2016open,
  title={Open-loop and closed-loop solvabilities for stochastic linear quadratic optimal control problems},
  author={Sun, Jingrui and Li, Xun and Yong, Jiongmin},
  journal={SIAM Journal on Control and Optimization},
  volume={54},
  number={5},
  pages={2274--2308},
  year={2016},
  publisher={SIAM}
}

@article{tang2015dpp,
    AUTHOR = {Tang, Shanjian},
     TITLE = {Dynamic programming for general linear quadratic optimal
              stochastic control with random coefficients},
   JOURNAL = {SIAM J. Control Optim.},
  FJOURNAL = {SIAM Journal on Control and Optimization},
    VOLUME = {53},
      YEAR = {2015},
    NUMBER = {2},
     PAGES = {1082--1106},
      ISSN = {0363-0129,1095-7138},
   MRCLASS = {93E20 (49K45 49L20 49N10 60H10)},
  MRNUMBER = {3337990},
       DOI = {10.1137/140979940},
       URL = {https://doi.org/10.1137/140979940}
}

@article{zhang2004stabilizability,
    AUTHOR = {Zhang, Weihai and Chen, Bor-Sen},
     TITLE = {On stabilizability and exact observability of stochastic
              systems with their applications},
   JOURNAL = {Automatica J. IFAC},
  FJOURNAL = {Automatica. A Journal of IFAC, the International Federation of
              Automatic Control},
    VOLUME = {40},
      YEAR = {2004},
    NUMBER = {1},
     PAGES = {87--94},
      ISSN = {0005-1098,1873-2836},
   MRCLASS = {93D15 (93E03 93E15)},
  MRNUMBER = {2143992},
       DOI = {10.1016/j.automatica.2003.07.002},
       URL = {https://doi.org/10.1016/j.automatica.2003.07.002},
}

@article{chen2004stochastic,
    AUTHOR = {Chen, Bor-Sen and Zhang, Weihai},
     TITLE = {Stochastic {$H_2/H_\infty$} control with state-dependent
              noise},
   JOURNAL = {IEEE Trans. Automat. Control},
  FJOURNAL = {Institute of Electrical and Electronics Engineers.
              Transactions on Automatic Control},
    VOLUME = {49},
      YEAR = {2004},
    NUMBER = {1},
     PAGES = {45--57},
      ISSN = {0018-9286,1558-2523},
   MRCLASS = {93E20 (93B36)},
  MRNUMBER = {2028541},
MRREVIEWER = {Zidong\ Wang},
       DOI = {10.1109/TAC.2003.821400},
       URL = {https://doi.org/10.1109/TAC.2003.821400},
}

@article{liu1999diss,
  title={Backward stochastic differential equation and stochastic control system},
  author={Liu, Yazeng},
  journal={Shandong University},
  year={1999}
}

@article{hou2012some,
    AUTHOR = {Hou, Ting and Zhang, Weihai and Ma, Hongji},
     TITLE = {Some properties of exact observability of linear stochastic
              systems and their applications},
   JOURNAL = {Asian J. Control},
  FJOURNAL = {Asian Journal of Control},
    VOLUME = {14},
      YEAR = {2012},
    NUMBER = {3},
     PAGES = {868--873},
      ISSN = {1561-8625,1934-6093},
   MRCLASS = {93B07 (60F25 60H10 93E03)},
  MRNUMBER = {2926017},
MRREVIEWER = {Xu\ Liu},
       DOI = {10.1002/asjc.391},
       URL = {https://doi.org/10.1002/asjc.391},
}

@article{li2010unified,
    AUTHOR = {Li, Zhao-Yan and Wang, Yong and Zhou, Bin and Duan, Guang-Ren},
     TITLE = {On unified concepts of detectability and observability for
              continuous-time stochastic systems},
   JOURNAL = {Appl. Math. Comput.},
  FJOURNAL = {Applied Mathematics and Computation},
    VOLUME = {217},
      YEAR = {2010},
    NUMBER = {2},
     PAGES = {521--536},
      ISSN = {0096-3003,1873-5649},
   MRCLASS = {60G35 (93C05 93E10)},
  MRNUMBER = {2678564},
       DOI = {10.1016/j.amc.2010.05.086},
       URL = {https://doi.org/10.1016/j.amc.2010.05.086},
}

@inbook{vanderschaft2009port,
  author    = {Van der Schaft, Arjan},
  title     = {Modeling and Control of Complex Physical Systems: The Port-Hamiltonian Approach},
  chapter   = {Port-Hamiltonian Systems},
  publisher = {Springer},
  year      = {2009}
}

@article{EhKr_cordoni2019stochastic,
  author={Cordoni, Francesco and Di Persio, Luca and Muradore, Riccardo},
  title={Stochastic Port-{H}amiltonian Systems},
  fjournal={Journal of Nonlinear Science},
  journal={J. Nonlinear Sci.},
  volume={32},
  number={6},
  pages={1--53},
  year={2022},
  publisher={Springer},
  doi={10.1007/s00332-022-09853-2}
}

@article{EhKr_cordoni2022discrete,
  author={Cordoni, Francesco Giuseppe and Di Persio, Luca and Muradore, Riccardo},
  title={Discrete stochastic port-{H}amiltonian systems},
  journal={Automatica},
  volume={137},
  pages={110122},
  year={2022},
  publisher={Elsevier},
doi={10.1016/j.automatica.2021.110122}
}

@article{EhKr_cordoni2020variable,
  title={A variable stochastic admittance control framework with energy tank},
  author={Cordoni, Francesco and Di Persio, Luca and Muradore, Riccardo},
  journal={IFAC-PapersOnLine},
  volume={53},
  number={2},
  pages={9986--9991},
  year={2020},
  publisher={Elsevier},
  doi={10.1016/j.ifacol.2020.12.2716}
}

@article{EhKr_cordoni2021bilateral,
  author={Cordoni, Francesco and Di Persio, Luca and Muradore, Riccardo},
  title={Bilateral teleoperation of stochastic port-{H}amiltonian systems using energy tanks},
  fjournal={International Journal of Robust and Nonlinear Control},
  journal={Int. J. Robust Nonlin. Control},
  volume={31},
  number={18},
  pages={9332-9357},
  year={2021},
  publisher={Wiley Online Library},
  doi={10.1002/rnc.5780}
}

@article{EhKr_cordoni2021stabilization,
  author={Cordoni, Francesco and Di Persio, Luca and Muradore, Riccardo},
  title={Stabilization of bilateral teleoperators with asymmetric stochastic delay},
  fjournal={Systems \& Control Letters},
  journal={Systems Control Lett.},
  volume={147},
  pages={104828},
  year={2021},
  publisher={Elsevier},
  doi={10.1016/j.sysconle.2020.104828}
}

@article{EhKr_cordoni2023weak,
author = {Cordoni, Francesco and Di Persio, Luca and Muradore, Riccardo},
title = {Weak Energy Shaping for Stochastic Controlled Port-{H}amiltonian Systems},
fjournal = {SIAM Journal on Control and Optimization},
journal = {SIAM J. Contr. Optim.},
volume = {61},
number = {5},
pages = {2902-2926},
year = {2023},
doi = {10.1137/22M1482585}
}

@article{EhKr_haddad2018energy,
   author={Haddad, Wassim M and Rajpurohit, Tanmay and Jin, Xu},
  title={Energy-based feedback control for stochastic port-controlled {H}amiltonian systems},
  journal={Automatica},
  volume={97},
  pages={134--142},
  year={2018},
  publisher={Elsevier},
doi={10.1016/j.automatica.2018.07.031}
}

@article{EhKr_satoh2017input,  
  author={Satoh, Satoshi},
  title={Input-to-state stability of stochastic port-{H}amiltonian systems using stochastic generalized canonical transformations},
  fjournal={International Journal of Robust and Nonlinear Control},
  fjournal={International Journal of Robust Nonlinear Control},
  journal={Int. J. Robust Nonlin. Control},
  volume={27},
  number={17},
  pages={3862--3885},
  year={2017},
  publisher={Wiley Online Library},
  doi={10.1002/rnc.3769}
}

@article{EhKr_satoh2014bounded,
author = {Satoshi Satoh and Masami Saeki},
title = {Bounded stabilisation of stochastic port-{H}amiltonian systems},
fjournal = {International Journal of Control},
journal = {Int. J. Contr.},
volume = {87},
number = {8},
pages = {1573-1582},
year = {2014},
publisher = {Taylor & Francis},
doi = {10.1080/00207179.2014.880127}
}

@article{EhKr_satoh2012passivity,
  author={Satoh, Satoshi and Fujimoto, Kenji},
  title={Passivity based control of stochastic port-{H}amiltonian systems},
  fjournal={IEEE Transactions on Automatic Control},
   journal={IEEE Trans. Autom. Control},
  volume={58},
  number={5},
  pages={1139--1153},
  year={2012},
  publisher={IEEE},
  doi={10.1109/TAC.2012.2229791}
}

@article{EhKr_fang2017stabilization,
  title={Stabilization of input-disturbed stochastic port-{H}amiltonian systems via passivity},
  author={Fang, Zhou and Gao, Chuanhou},
  fjournal={IEEE Transactions on Automatic Control},
  journal={IEEE Trans. Automat. Contr.},
  volume={62},
  number={8},
  pages={4159--4166},
  year={2017},
  publisher={IEEE},
  doi={10.1109/TAC.2017.2676619}
}

@article{EhKr_Eh24,
title={The Collective Dynamics of a Stochastic Port-{H}amiltonian Self-Driven Agent Model in One Dimension},
  author={Ehrhardt, Matthias and Kruse, Thomas and Tordeux, Antoine},
  fjournal={ESAIM: Mathematical Modelling and Numerical Analysis},
  journal={ESAIM: Math. Model. Numer. Anal.},
volume = {58},
pages = {515-544},
  year={2024},
  doi={10.1051/m2an/2024004}
}

@article{EhKr_ackermann_tordeux24,
author={Ackermann, Julia and Ehrhardt, Matthias and Kruse, Thomas and Tordeux, Antoine},
  title={Stabilisation of stochastic single-file dynamics using port-{H}amiltonian systems},
journal = {IFAC-PapersOnLine},
volume = {58},
number = {17},
pages = {145-150},
year = {2024},
note = {26th International Symposium on Mathematical Theory of Networks and Systems MTNS 2024},
issn = {2405-8963},
doi = {https://doi.org/10.1016/j.ifacol.2024.10.128},
}

@article{EhKr_DELVENNE2014123,
author={Delvenne, Jean-Charles and Sandberg, Henrik},
title = {Finite-time thermodynamics of port-{H}amiltonian systems},
fjournal = {Physica D: Nonlinear Phenomena},
journal = {Physica D: Nonlin. Phenom.},
volume = {267},
pages = {123-132},
year = {2014},
note = {Evolving Dynamical Networks},
doi = {10.1016/j.physd.2013.07.017}
}

@article{EhKr_florchinger1999passive,
author = {Florchinger, Patrick},
title = {A Passive System Approach to Feedback Stabilization of Nonlinear Control Stochastic Systems},
fjournal = {SIAM Journal on Control and Optimization},
journal = {SIAM J. Control Optim.},
volume = {37},
number = {6},
pages = {1848-1864},
year = {1999},
doi = {10.1137/S0363012997317478}
}

@article{florchinger2016global,
    AUTHOR = {Florchinger, Patrick},
     TITLE = {Global asymptotic stabilisation in probability of nonlinear
              stochastic systems via passivity},
   JOURNAL = {Internat. J. Control},
  FJOURNAL = {International Journal of Control},
    VOLUME = {89},
      YEAR = {2016},
    NUMBER = {7},
     PAGES = {1406--1415},
       DOI = {10.1080/00207179.2015.1132009}
}

@article{EhKr_Haddad2023DissStochDynSys,
title = {Dissipative stochastic dynamical systems},
journal = {Systems \& Control Letters},
volume = {172},
pages = {105451},
year = {2023},
issn = {0167-6911},
doi = {10.1016/j.sysconle.2022.105451},
author = {Manuel Lanchares and Wassim M. Haddad}
}

@book{EhKr_khasminskii2011stochastic,
  title={Stochastic stability of differential equations},
  author={Khasminskii, Rafail},
  volume={66},
  year={2011},
  publisher={Springer Science \& Business Media},
doi={10.1007/978-3-642-23280-0}
}

@book{EhKr_vandSchaft2016l2,
  title={$L_2$-gain and passivity techniques in nonlinear control},
  author={Van der Schaft, Arjan},
  year={2016},
  publisher={Springer},
  doi={10.1007/978-1-4471-0507-7}
}

\end{document}